\documentclass[11pt]{article}
\usepackage[english]{babel}
\usepackage{amsmath,amsthm}
\usepackage{amsfonts}
\usepackage{mathrsfs}
\usepackage{color}
\usepackage{bbm}
\usepackage{enumerate}
\usepackage{amsrefs}

\newtheorem{theorem}{Theorem}[section]
\newtheorem{corollary}[theorem]{Corollary}
\newtheorem{lemma}[theorem]{Lemma}

\newtheorem{assumption}[theorem]{Assumption}
\theoremstyle{definition}

\theoremstyle{remark}
\newtheorem{remark}[theorem]{Remark}

\numberwithin{equation}{section}

\begin{document}
\def\Pro{{\mathbb{P}}}
\def\E{{\mathbb{E}}}
\def\e{{\varepsilon}}
\def\veps{{\varepsilon}}
\def\ds{{\displaystyle}}
\def\nat{{\mathbb{N}}}
\def\Dom{{\textnormal{Dom}}}
\def\dist{{\textnormal{dist}}}
\def\R{{\mathbb{R}}}
\def\O{{\mathcal{O}}}
\def\T{{\mathcal{T}}}
\def\Tr{{\textnormal{Tr}}}
\def\I{{\mathcal{I}}}
\def\A{{\mathcal{A}}}
\def\H{{\mathcal{H}}}
\def\S{{\mathcal{S}}}
\def\sgn{{\textnormal{sign}}}

\title{Metastability and exit problems for systems of stochastic reaction-diffusion equations}

%\title{Metastability theory for SPDEs with multiple stable equilibria}%
\author{M. Salins K. Spiliopoulos\footnote{Boston University, Department of Mathematics and Statistics, 111 Cummington Mall, Boston, 02215.
E-mails: msalins@bu.edu, kspiliop@math.bu.edu. K.S. was partially supported by NSF DMS 1550918 and Simons Foundation Award  672441.}}
%\thanks{K.S. was partially supported by NSF DMS 1550918 and Simons Foundation Award  672441}
%\runtitle{Exit problems for SRDE}
%
%% indicate corresponding author with \corref{}
%% \author{\fnms{John} \snm{Smith}\corref{}\ead[label=e1]{smith@foo.com}\thanksref{t1}}
%% \thankstext{t1}{Thanks to somebody}
%% \address{line 1\\ line 2\\ printead{e1}}
%% \affiliation{Some University}
%
%\author{\fnms{Michael} \snm{Salins} \corref{}\ead[label=e1]{msalins@bu.edu}}
%\address{\printead{e1}\\
%111 Cummington Mall\\
%Boston, MA 02215\\
%USA}
%\and
%\author{\fnms{Konstantinos} \snm{Spiliopoulos}\ead[label=e2]{kspiliop@bu.edu}}
%\thankstext{t1}{K.S. was partially supported by NSF DMS 1550918 and Simons Foundation Award  672441}
%\address{\printead{e2}\\
%111 Cummington Mall\\
%Boston, MA 02215\\
%USA}
%\affiliation{Boston University}
%
%\runauthor{Salins and Spiliopoulos}

\maketitle

\begin{abstract}
In this paper we develop a metastability theory for %infinite dimensional dynamical systems.
a class of stochastic reaction-diffusion equations exposed to small multiplicative noise. We consider the case where the unperturbed reaction-diffusion equation features multiple asymptotically stable equilibria. When the system is exposed to small stochastic perturbations, it is likely to stay near one equilibrium for a long period of time, but will eventually transition to the neighborhood of another equilibrium. We are interested in studying the exit time from the full domain of attraction (in a function space) surrounding an equilibrium and therefore do not assume that the domain of attraction  features uniform attraction to the equilibrium. This means that the boundary of the domain of attraction is allowed to contain saddles and limit cycles.  Our method of proof is purely infinite dimensional, i.e., we do not go through finite dimensional approximations. In addition, we address the multiplicative noise case and we do not impose gradient type of assumptions on the nonlinearity. We prove large deviations logarithmic asymptotics for the exit time and for the exit shape, also characterizing the most probable set of shapes of solutions at the time of exit from the domain of attraction.
\end{abstract}

%\begin{keyword}[class=MSC]
%\kwd[Primary ]{60F10, 60H15}
%\kwd{35R60, 60G40}
%%\kwd[; secondary ]{}
%\end{keyword}


%\end{frontmatter}

% AOS,AOAS: If there are supplements please fill:
%\begin{supplement}[id=suppA]
%  \sname{Supplement A}
%  \stitle{Title}
%  \slink[doi]{10.1214/00-AOASXXXXSUPP}
%  \sdatatype{.pdf}"
%  \sdescription{Some text}
%\end{supplement}
\section{Introduction}

In this paper we study the following  $r$-dimensional system of stochastic reaction-diffusion equations  exposed to small noise on the spatial domain $\xi \in [0,L]$, $X^\e(t,\xi) = (u_1^\e(t,\xi),...,u_r^\e(t,\xi))$
\begin{equation} \label{eq:intro-ac}
  \begin{cases}
    \frac{\partial u^\e_i}{ \partial t} (t,\xi) = \mathcal{A}_i u^\e_i(t,\xi) + f_i(X^\e(t,\xi)) + \sqrt{\e}\sum_{j=1}^r g_{ij}(X^\e(t,\xi))\frac{\partial w_j}{\partial t}(t,\xi),\\
    u^\e_i(0,\xi) = x_i(\xi), \ \ \ \ u^\e_i(t,0) = u^\e_i(t,L) = 0.
  \end{cases}
\end{equation}

A uniform large deviations principle for a similar system of equations was investigated in \cite{Cerrai-RocknerLDP} (see also \cite{BudhirajaDupuisSalins2018,s-1992,kx-1996}).

The differential operators $\mathcal{A}_i$ are elliptic second-order operators such that for ${\varphi} \in C^2([0,L])$
\begin{equation*}
  \mathcal{A}_i {\varphi}(\xi)= a^i(\xi)\frac{d^2 {\varphi}}{d \xi^2}(\xi) + b^i(\xi)\frac{d{\varphi}}{d\xi}(\xi).
\end{equation*}

In the above expression $a^i(\xi)$ belong to $C^1([0,L])$, $b^i$ belong to $C([0,L])$ and the $a^i(\xi)$ are positive and strictly bounded from below by a positive number.

The vector-valued function $f: \mathbb{R}^r \to \mathbb{R}^r$ given by $f = (f_1,...,f_r)$ is locally Lipschitz continuous and satisfies certain dissipativity properties that are specified in Assumption \ref{assum:nonlinear} below. The matrix-valued function $g=(g_{ij}) :\mathbb{R}^r \to \mathbb{R}^{r\times r}$ is locally Lipschitz continuous and uniformly elliptic.

The driving noises $w_i$ $i=1,..,r$ are independent cylindrical Wiener processes on the Hilbert space $L^2([0,L]:\mathbb{R})$. The vector-valued random field $w = (w_1,...,w_r)$ is a cylindrical Wiener processes on the Hilbert space $H:=L^2([0,T]:\mathbb{R}^r)$.

We let $E$ denote the function space $C_0([0,L]:\mathbb{R}^r)$ of continuous vector-valued functions on $[0,L]$ that are equal to zero on the endpoints, endowed with the supremum norm. For any initial data $x \in E$, we denote the solution to \eqref{eq:intro-ac} by $X^\e_x(t,\xi) = (u^\e_1,...,u^\e_r)$ to emphasize the dependence on the initial data {$x(\xi)=(x_{1}(\xi),\cdots, x_{r}(\xi))$}. The process $X^\e_x(t,\cdot)$ has solutions that are almost surely continuous in time and space (see \cite{c-2003}) and therefore we can consider $X^\e_x(t):=X^\e_x(t,\cdot)$ as $E$-valued for any $t \geq 0$.

%In the above equation, $f$ is a locally Lipschitz function featuring certain dissipativity properties (see Assumption \ref{assum:nonlinear}), $g$ is a Lipschitz continuous function that is bounded from above and below (see Assumption \ref{A:DiffusionCoeff}), the parameter $\e>0$ and $\frac{\partial w}{\partial t}$ is a space-time white noise.

The purpose of this paper is to develop a metastability theory for $X^\e_x$ in the infinite dimensional function space $E$ that is analogous to the metastability theory for finite dimensional diffusions developed by \cite{FWbook}. In contrast to previous characterizations of the exit time in infinite dimensions, see \cite{Barret2015, BerglundGentz2013, BerglundWeber2017, cm-1997,BudhirajaDupuisSalins2018,dpz-1991, f-1988,g-2005,g-2008,l-2018}, we do not assume that the domain of attraction in $E$ has uniform attraction to the equilibrium, our method of proof is purely infinite dimensional, i.e. we do not go through finite dimensional approximations, we allow multiplicative noise and we do not impose gradient type of assumptions on the nonlinearity. In addition, we study, for the first time in this paper, the logarithmic asymptotic of the exit shape distribution and we characterize the most likely exit behaviors from the domain of attraction.

A set ${D_{u}}\subset E$ surrounding an asymptotically stable equilibrium $x_*$ of the unperturbed system $X^0_x$ is said to be uniformly attractive if
\begin{equation} \label{eq:unif-attr-def}
  \lim_{t \to \infty} \sup_{x \in {D_{u}}} |X^0_x(t) - x_*|_E = 0.
\end{equation}

In this paper we study the problem of exiting from a full domain of attraction. Specifically, if $x_* \in E$ is an asymptotically stable equilibrium of the unperturbed system $X^0_x$, we study the exit from
\[D := \{x \in E: \lim_{t \to \infty} |X^{0}_x(t) - x_*|_E = 0\}.\]

 Such a set $D$ is not uniformly attractive and
the boundary of $D$ can contain unstable equilibria (saddles) or limit cycles.
The exit problem from such a $D$ is often referred to in the literature as the case of a characteristic boundary. In the finite dimensional case, the characteristic boundary case was studied by Day in \cite{Day1990}. Genuinely new ideas are needed in passing to infinite dimensions since (a): the domain of attraction is not compact, and (b): in infinite dimensions one cannot connect any two points in $E$ with a controlled path (both of (a) and (b) are standard tools used in finite dimensions).

The stochastically perturbed system $X^\e_x(t)$, when $\e>0$ is small but positive, is likely to stay in $D$ for long periods of time. Once the system exits $D$ it can enter a different metastable domain. Due to certain non-degeneracy assumptions on the stochastic perturbations, the exit time $\tau^\e_x := \inf\{t>0: X^\e_x(t) \not \in D\}$ is finite with probability one, but diverges as $\e \to 0$. We will characterize the exponential divergence rates
\begin{equation} \label{eq:intro-log-of-tau}
  \lim_{\e \to 0} \e\log\E \tau^\e_x \text{ and } \lim_{\e \to 0} \e \log \tau^\e_x
\end{equation}
as well as the limiting distribution of the exit shape
$X^\e_x(\tau^\e_x)$ (Theorems \ref{thm:meanExitAsymptotics} and \ref{thm:exit-shape}).

In this paper we use recent uniform large deviations principle results for stochastic partial differential equations, \cite{BudhirajaDupuisSalins2018, Cerrai-RocknerLDP}, to prove logarithmic asymptotics for the mean exit time and for its corresponding distribution function, as the noise converges to zero. In addition, we prove, for the first time in infinite dimensions without uniform attraction,  lower and upper bounds for the large deviations principle of the exit shape (analogous to the exit position in finite dimensions).  Under additional conditions such lower and upper bounds match, yielding a large deviations principle for the exit shape.

To the best of our knowledge one of the first papers that studied large deviations for infinite dimensional systems with an eye towards metastability is \cite{FarisLasinio1982}. In \cite{FarisLasinio1982} the authors study the stochastic Allen-Cahn equation, characterize in detail the set of stable and unstable equilibrium points, properties of the equilibrium action, and obtain upper and lower bounds on the probabilities of tunnleing from one stable equilibrium to another.

Detailed results on mean exit time asymptotics for stochastic reaction diffusion equations (of the type considered in \cite{FarisLasinio1982}) can also be found in \cite{BerglundGentz2013}, see also \cite{Barret2015, BerglundWeber2017} and \cite{Debussche2013,Hogele2019}. The setup in \cite{BerglundGentz2013} is restricted to additive noise, with the nonlinear term in the equation being of gradient type.  The method of proof in \cite{BerglundGentz2013} is based on approximating the mean exit time of the infinite dimensional process by a sequence of mean exit times of appropriate finite dimensional problems making use of the potential theoretic approach to metastability developed in \cite{BovierEckhoffGayrandKlein2004,BovierGayrandKlein2005}. In this paper, we address the multiplicative noise case, we do not impose gradient type of restrictions on the nonlinear term and our method of proof is purely infinite dimensional without going through the finite dimensional approximation.  In addition, we obtain the logarithmic asymptotics of the exit shape distribution.

We comment on the assumptions that we impose on this model. We assume that the differential operator $\mathcal{A}$ is elliptic. Such an operator generates an analytic $C_0$ semigroup on $E$.  We assume that the nonlinear function $f$ is locally Lipschitz continuous and features a super-linear dissipativity property (see Assumption \ref{assum:nonlinear}). One important consequence of this kind of dissipativity is that it forces the solutions into a finite part of the space with overwhelming probability (see \eqref{eq:X-bound-sup}-\eqref{eq:X-big-zero-prob}). The compactness of the semigroup along with this dissipativity leads to certain compactness and tightness properties of the solutions. Any nonlinear odd-degree polynomial with negative leading coefficient such as $f(x) = -x^3 + x$ features super-linear dissipativity. The assumption that $f$ features super-linear dissipativity enables the proof of exit time and exit shape asymptotics from unbounded domains of the function space. If $f$ does not feature super-linear dissipativity (for example, if $f$ is globally Lipschitz continuous) then the methods developed in this paper can be adapted to prove exit time and exit shape asymptotics from bounded domains of attraction only. Since we expect the full domains of attraction around an attractor to be unbounded, we focus on the super-linear dissipativity setting.

We assume that the spatial dimension is $1$ so that the system exposed to space-time white noise is well-defined. If the spatial dimension were greater than 1, then an appropriately degenerate colored noise would be required for function-valued solutions to exist (See Hypothesis 2 of \cite{Cerrai-RocknerLDP}), however, the methods used in the current paper would not work if the noise were degenerate. We also assume that the diffusion matrix $g$ is invertible and that the inverse in uniformly bounded in operator norm (Assumption \ref{A:DiffusionCoeff}). The non-degeneracy of the space-time white noise and non-degeneracy of $g$ enable us to prove important results about the ability to connect controlled trajectories (Theorem \ref{thm:control}). We also make two reasonable assumptions about the boundary of the domain of attraction $D$. We assume that if $x$ is an initial condition that does not belong to $\bar{D}$ then deterministic control problems starting at $x$ cannot reach $\partial D$ with arbitrarily small controls (Assumption \ref{A:AttractionProperty}). Essentially this assumption says that $\partial D$ acts as a separatrix between different attractive domains in the function space. We also assume that there exists a finite collection $K_1,...,K_n \subset \partial D$ of so-called $V$ equivalence classes that contain all of the $\omega$-limit points of the unperturbed system $X^0_x$ for $x \in \partial D$ (see Assumpiton \ref{A:K_equivalenceClassV}).

All of the assumptions outlined in the previous paragraphs are important for our analysis, but they are also very general. We do not assume that $f$ has any gradient structure. We allow for multiplicative noise and allow for general non-degenerate, bounded, and Lipschitz continuous diffusion matrices $g$. As mentioned  before, we study the exit from a full domain of attraction around an asymptotically stable solution. Furthermore, in contrast to the assumptions imposed in most of the investigations of exit place asymptotics in the finite dimensional setting (eg. \cite{FWbook,Day1990}), we make no assumptions about the regularity of the boundary of $D$. In lieu of assuming boundary regularity we propose two new notions of the quasipotential to characterize the exit shape problem. Without assuming any regularity properties of the boundary of $D$ we cannot quite prove a LDP for the exit shape, but rather prove the result with a lower bound and upper bound that do not necessarily match (Theorem \ref{thm:exit-shape}). We expect analogous results to hold for finite dimensional exit place asymptotics from regions with irregular characteristic boundaries. In Section \ref{S:ExitShape_AdditionalCond} we explore which extra requirements on $D$ lead to an LDP for the exit shape.

The rest of the paper is organized as follows.  In Section \ref{S:Assumptions} we fix our notations and we describe our main assumptions about $\mathcal{A}$, $f$, $g$, and $\partial w/\partial t$ from \eqref{eq:intro-ac}. %that hold throughout the paper.
%Assumptions that are only needed for specific sections of the article will presented in the corresponding sections.
Our other two main assumptions about the domain of attraction $D$ are presented in Section \ref{S:LDP} after the quasipotential $V(x,y)$ has been defined. In Section \ref{S:MildFormulation} we describe the mild formulation of (\ref{eq:intro-ac}) and establish some key uniform bounds on its solution that will be used throughout the paper (see Theorem \ref{T:BOUNDEDNESSPROPOFX}).

In Section \ref{S:LDP}, we recall the related uniform large deviations results of \cite{BudhirajaDupuisSalins2018, Cerrai-RocknerLDP}, and define the quasipotentials $V(x,y)$, $\tilde{V}_D(x,y)$, and $\hat{V}_D(x,y)$ in terms of optimal control problems with respect to the large deviations rate function. Then we present the main results of our paper, Theorems \ref{thm:meanExitAsymptotics} and \ref{thm:exit-shape}, which characterize the exponential divergence rate of the exit time and the mean exit time \eqref{eq:intro-log-of-tau}, and the large deviations upper and lower bounds of the distribution of the exit shape $X^\e_x(\tau^\e_x)$, in terms of these quasipotentials.

Section \ref{S:controlability} contains a key result of the paper on controllability of the underlying controlled infinite dimensional dynamical system as well as results on continuity, lower semicontinuity, and compactness properties of the  the quasipotentials.  % that appear in the metastability results of the paper.
In particular, even though, and in contrast to the finite dimensional case, we cannot connect any two arbitrary elements of the domain of attraction with a controlled trajectory, we can still construct a control that connects initial conditions to certain given controlled paths after a short time and with minimal additional energy, given that the initial points are close in the uniform norm. The results of Section \ref{S:controlability} may be of independent interest.  %In addition, we establish continuity and regularity properties of the related quasipotential functionals. % Section \ref{S:TimeEnterBoundedBalls} shows that the time to enter a ball of finite radius in the uniform norm is small, uniformly in the initial conditions.

Sections \ref{S:ExitTimeAsymptotics} and \ref{S:ExitShape} contain the proofs of  Theorems \ref{thm:meanExitAsymptotics} and \ref{thm:exit-shape} respectively. In Section \ref{S:ExitTimeAsymptotics} we prove that the exponential divergence rates of the exit time and the mean exit time \eqref{eq:intro-log-of-tau} are characterized in terms of the quasipotential.
% logarithmic asymptotics for the mean exit time and for its distribution function analogous to the ones that hold in the finite dimensional case. %The lack of uniformity of the domain of attraction and the lack of continuity of the quasipotential complicates the proof.
What allows one to derive the desired results are the controllability results, the continuity and regularity properties of the quasipotential and the uniform bounds established in Section \ref{S:controlability}.

 In Section \ref{S:ExitShape} we prove large deviations lower and upper bounds  for the distribution of the exit shape $X^\e_x(\tau^\e_x)$, analogous to the exit place in finite dimensions. Due to the lack of uniform attraction to the stable equilibrium,  the characteristic boundary of $D$, and the infinite dimensionality of our setting, the proof here is considerably more complicated than in the analogous finite dimensional case of \cite{FWbook} and \cite{Day1990}.  Again we make extensive use of the controllability result of Section \ref{S:controlability} and we study the logarithmic asymptotics of appropriate Markov chains, analogous in spirit to the Markov chain constructed in \cite{FWbook} but substantially different in order to deal with the infinite dimensional aspect of the problem and due to the inability to consider reflected process (in contrast to what was done in \cite{Day1990}).
 Then, in Section \ref{S:ExitShape_AdditionalCond}, we demonstrate that if we impose an additional assumption about boundary regularity of the domain of attraction $D$ and controllability of the processes near the boundary, that the large deviations lower and upper bounds in Theorem \ref{thm:exit-shape} match, leading to a large deviations principle for the exit shape distribution.

\section{Assumptions and notation}\label{S:Assumptions}
For any Banach spaces $\mathcal{X}$, $\mathcal{Y}$ let $\mathscr{L}(\mathcal{X},\mathcal{Y})$ be the set of bounded linear operators from $\mathcal{X} \to \mathcal{Y}$ endowed with the norm
\[|B|_{\mathscr{L}(\mathcal{X},\mathcal{Y})} := \sup_{|x|_\mathcal{X}=1} |Bx|_{\mathcal{Y}}.\]
We let $\mathscr{L}(\mathcal{X}):=\mathscr{L}(\mathcal{X},\mathcal{X})$.

For any Banach space $\mathcal{X}$ define the distance between a set $K \subset \mathcal{X}$ and $y \in \mathcal{X}$
\[\dist_\mathcal{X}(y,K):= \inf_{x \in K} |x-y|_E.\]

As introduced in the previous section, let $E:=C_0([0,L]:\mathbb{R}^r)$ endowed with the supremum norm
\[|x|_E:=\sup_{\xi \in [0,L]}|x(\xi)|.\]

Let $H:=L^2([0,L]:\mathbb{R}^r)$  with the norm
\[|x|_H:= \left(\int_0^L |x(\xi)|^2 d\xi \right)^{1/2}\]
and denote the associated inner product by
\[\left<x,y\right>_H := \int_0^L x(\xi)\cdot y(\xi)d\xi.\]

In the  expressions above $| \ |$ denotes the $\mathbb{R}^r$ norm and $\cdot$ denotes the $\mathbb{R}^r$ inner product.

Arguing as in Section 2 of \cite{Cerrai-RocknerLDP} we can assume without loss of generality that the differential operator $\mathcal{A}$ is in divergence form and its realization $A$ in $H$ with imposed boundary conditions is self-adjoint. Then there exists a complete orthonormal system of $H$ such that $Ae_k = -\alpha_k e_k$.  Define fractional powers $(-A)^\delta e_k = \alpha_k^\delta e_k$. Define fractional Sobolev spaces for $\delta \in \mathbb{R}$ $H^\delta$ is the completion of $C^\infty_0([0,L]:\mathbb{R}^r)$ under the norm
\begin{equation} \label{eq:Sobolev-spaces}
|x|_{H^\delta}^2 = \sum_k \alpha_k^\delta \left<x,e_k\right>_H^2.
\end{equation}

For any subset $D \subset E$ we let $\bar{D}$ denote the closure of $D$ in $E$. We let $D^c$ denote the complement of $D$ in $E$.
For any $x \in E$ and $\delta>0$ we denote the open ball by
\[B(x,\delta):= \{y \in E: |x-y|_E<\delta\}.\]

Assume that the nonlinearity $f:\mathbb{R}^r \to \mathbb{R}^r$ satisfies the following assumption.
\begin{assumption} \label{assum:nonlinear}
  \begin{enumerate}[(a)]
    \item $f$ is Locally Lipschitz continuous. {For any $R>0$ there exists $C=C(R)$ such that whenever $x,y \in \mathbb{R}^r$ with $|x|\leq R$ and $|y|\leq R$,
    \begin{equation}
      |f(x)-f(y)| \leq C|x-y|.
    \end{equation}}
    \item There exist $\rho_{*},\lambda,C>0$ such that for any $x, h \in \mathbb{R}^r$, %\textcolor{red}{[K: I changed $\rho$ to $\rho_{*}$ here and elsewhere in the paper where this constant appears.]}
    \begin{equation} \label{eq:f-dissip}
      (f(x+h) - f(x))\cdot\frac{h}{|h|} \leq - \lambda |h|^{1+\rho_{*}} + C(1 + |x|^{1+\rho_{*}}).
    \end{equation}
    and
   \begin{equation}
      |f(x)| \leq C(1+|x|^{1+\rho_{*}}).
   \end{equation}
  \end{enumerate}
\end{assumption}

Assumption \ref{assum:nonlinear} is satisfied, for example, when the $f_i$ are polynomials with odd degree and a negative coefficient for the leading term. If $\rho_{*}=2m\geq 2$ is an even number, $\lambda_i>0$, $i\in \{1,...,r\}$ and $p_i(x)$ for $i \in \{1,...,r\}$ are polynomials on $\mathbb{R}^r$ with degree less than or equal to $2m$  then
\begin{equation}
  f_i(x) = -\lambda_i x_i^{1+2m} + p_i(x)
\end{equation}
will satisfy our assumptions.
As we will show below, see Theorem \ref{T:BOUNDEDNESSPROPOFX}, property \eqref{eq:f-dissip} will guarantee that the process spends most of its time on a finite part of the space.

Define the Nemyitskii operator $F: E \to E$ by for $x \in E$ and $\xi \in [0,L]$,
\begin{equation} \label{eq:F-def}
  F(x)(\xi) = f(x(\xi)).
\end{equation}
Under Assumption \ref{assum:nonlinear}, $F$ has the properties that
\begin{enumerate}
  \item $F$ is locally Lipschitz continuous. For any $R>0$ there exists $\kappa=\kappa(R)$ such that whenever $x,y \in E$ with $|x|_E\leq R$ and $|y|_E\leq R$, $|F(x)-F(y)|_E \leq \kappa |x-y|_E.$
  \item For any $x,h \in E$ and any %\textcolor{red}{[K: I changed $x^{*}$ to $y$ here.]}
  $\delta \in \partial |h|_E := \{y \in E^\star: |y|_{E^\star} = 1 \text{ and } \left<h,y\right>=|h|_E\},$
      \begin{equation} \label{eq:F-dissip}
        \left<F(x+h) - F(x), \delta \right> \leq -\lambda |h|_E^{1+\rho_{*}} + C(1 + |x|_E^{1+\rho_{*}}).
      \end{equation}
  \item There exists $C>0$ such that for any $x \in E$,
        \begin{equation} \label{eq:F-growth}
          |F(x)|_E \leq C(1 + |x|_E^{1 + \rho_{*}}).
        \end{equation}
\end{enumerate}

In regards to the diffusion coefficient $g$ we impose Assumption \ref{A:DiffusionCoeff}.
\begin{assumption}\label{A:DiffusionCoeff}
  There exists $0<\kappa_0< \kappa_1$ such that for all $x,h \in \mathbb{R}^r$,
  \begin{equation} \label{eq:g-bound}
   \kappa_0|h| \leq |g(x)h| \leq \kappa_1|h|.
  \end{equation}
  There exists $\kappa>0$ such that for any $x, y \in \mathbb{R}^r$
  \begin{equation} \label{eq:g-Lip}
    |g(x) - g(y)|_{\mathscr{L}(\mathbb{R}^r)} \leq \kappa |x-y|.
  \end{equation}
\end{assumption}

For $x \in E$,$h \in H$ and $\xi \in [0,L]$ define
\begin{equation} \label{eq:G-cap-def}
  [G(x)h](\xi):= g(x(\xi))h(\xi).
\end{equation}
$G$ is a mapping from $E \to \mathscr{L}(H)$. Furthermore, it satisfies the following properties.
\begin{enumerate}
  \item For any $x \in E$ and $h \in H$,
  \begin{equation} \label{eq:G-bound}
    \kappa_0 |h|_H \leq |G(x)h|_H \leq \kappa_1 |h|_H
  \end{equation}
  \item For any $x, y \in E$,
  \begin{equation} \label{eq:G-Lip}
    |G(x) - G(y)|_{\mathscr{L}(H)} \leq \kappa |x-y|_E.
  \end{equation}
  \item For any $x \in E$, $G(x)$ is invertible and
  \begin{equation}
    [G^{-1}(x)h](\xi) = g^{-1}(x(\xi))h(\xi)
  \end{equation}
  and
  \begin{equation} \label{eq:G-inv-norm}
    |G^{-1}(x)|_{\mathscr{L}(H)} \leq \frac{1}{\kappa_0}.
  \end{equation}
\end{enumerate}
\[\]

\begin{remark}\label{R:boundednessG}
We mention here that the upper boundedness assumption on $G(x)$ has been mainly utilized in the proofs of Theorem \ref{thm:stoch-conv-bound}, Theorem \ref{T:BOUNDEDNESSPROPOFX}, part 3, and Theorem \ref{thm:exp-est-outside-ball}. At the expense of slightly more elaborate estimates, one could weaken the global upper bound assumption to potentially allow some growth and as long as the statements of Theorem \ref{thm:stoch-conv-bound}-\ref{thm:exp-est-outside-ball} hold, the rest of the results of the paper will also hold. We chose to assume boundedness to simplify those proofs and thus focus on the rest of the results which also constitute the main novelty of the paper.
\end{remark}

We close this section by describing the space-time white noises. For any $i \in \{1,...,r\}$, $\frac{\partial w_i}{\partial t}$ are independent and have the property that for any $T>0$ and any deterministic $\varphi \in L^2([0,L]\times[0,T])$,
\[\int_0^T \int_0^L \varphi(t,\xi) \frac{\partial w_i}{\partial t}(t,\xi)d\xi dt\]
is a Gaussian random variable. They have covariances
\begin{align*}
  &\E \left(\int_0^T \int_0^L \varphi(t,\xi) \frac{\partial w_i}{\partial t}d\xi dt \right) \left( \int_0^T \int_0^L \psi(t,\xi) \frac{\partial w_j}{\partial t}d\xi dt\right) \\
  &= \delta_{ij}\int_0^T \int_0^L \varphi(t,\xi)\psi(t,\xi)d\xi dt.
\end{align*}
where $\delta_{ij} = 1$ if $i=j$ and $\delta_{ij}=0$ for $i \not =j$.
Recall that the eigenfunctions of $A$,
$\{e_k\}_{k=1}^\infty$, form a complete orthonormal basis of $H:=L^2([0,L]:\mathbb{R}^r)$. For $i \in \{1,...,r\}$, let $e_{k,i}(\xi)$ be the $i$th component of $e_k(\xi)$  then
\begin{equation*}
  \beta_k(t):= \sum_{i=1}^r \int_0^t \int_0^L e_{k,i}(\xi) \frac{\partial w_i}{\partial t}(t,\xi)d\xi dt
\end{equation*}
is a family of independent identically distributed one-dimensional Brownian motions. Let $\{\mathcal{F}_t\}$ be the natural filtration of this collection of Brownian motions. Let $w(t)$ be the formal sum
\begin{equation*}
  w(t)= \sum_{k=1}^\infty e_k \beta_k(t).
\end{equation*}

For any $\mathcal{F}_t$-adapted $H$-valued process, $\varphi(t)$ satisfying $\E \int_0^T |\varphi(t)|_H^2 dt<+\infty$, the stochastic integral is defined to be the $L^2(\Omega)$ limit
\begin{equation*}
  \int_0^T \left<\varphi(s),dw(s)\right>_H = \sum_{k=1}^\infty \int_0^T \left<\varphi(s),e_k\right>_H d\beta_k(s).
\end{equation*}

\section{The mild solution and its properties}\label{S:MildFormulation}

As stated in the previous section, $A$ is the realization of $\mathcal{A}=\text{diag}(\mathcal{A}_1,...,\mathcal{A}_r)$ in $E$ with the imposed boundary conditions. $A$ generates a $C_0$ semigroup in $E$, which we call $S(t)$. $S(t)$ has many regularization properties including (see for example \cite{CerraiSmoothing1999})
\begin{equation} \label{eq:semigroup-regularity}
  |S(t)h|_E \leq C t^{-1/4} |h|_H.
\end{equation}
%where $H = L^2([0,L])$.
$S(t)$ is also a compact semigroup in the sense that for any $t>0$ the image $\{S(t)x : |x|_E\leq 1\}$ is pre-compact in $E$. Let $F:E \to E$ be defined by \eqref{eq:F-def} and $G: E \to \mathscr{L}(H)$ be defined by \eqref{eq:G-cap-def}.

The mild solution of \eqref{eq:intro-ac} is the solution to the integral equation
\begin{equation} \label{eq:mild}
  X^\e_x(t) = S(t)x + \int_0^t S(t-s)F(X^\e_x(s))ds + \sqrt{\e}\int_0^t S(t-s)G(X^\e_x(s))dw(s).
\end{equation}
The subscript $x$ represents the initial $E$-valued condition.

In the sequel, we establish certain boundedness and regularizing properties of $X^\e_x(t)$, the solution to (\ref{eq:mild}). Let us define  $Y^\e_x(t)$ to be the stochastic convolution
\begin{equation}\label{eq:stoch-conv-def}
      Y^\e_x(t): = \sqrt{\e}\int_0^t S(t-s) G(X^\e_x(s))dw(s).
\end{equation}

Using the classical stochastic factorization method and the fact that $g$ is uniformly bounded above \eqref{eq:g-bound} we obtain Theorem \ref{thm:stoch-conv-bound} (see Lemma 4.1 in \cite{CerraiRDEAveraging1}).
\begin{theorem} \label{thm:stoch-conv-bound}
  There exists  $p>1$ and a constant $C=C(p)>0$ such that for any $x \in E$ and $\e>0$,
  \begin{equation} \label{eq:stoch-conv-bound}
    \E\sup_{t \in [0,T]}|Y^\e_x(t)|_E^p \leq C\e^{\frac{p}{2}} T.
  \end{equation}
\end{theorem}

The main result of this section is the following theorem.
\begin{theorem}\label{T:BOUNDEDNESSPROPOFX}
  The mild solution $X^\e_x$ solving \eqref{eq:mild} satisfies the following bounds almost surely
  \begin{enumerate}
    \item There exists $C>0$ such that or any $x \in E$, $\e>0$ and $t>0$,
    \begin{equation} \label{eq:X-bound-x}
      \sup_{s \in [0,t]} |X^\e_x(s)|_E \leq C \left(1 + |x|_E +  \sup_{s \in [0,t]} |Y^\e_x(s)|_E\right).
    \end{equation}
    \item There exists $C>0$ such that for any $x \in E$, $\e>0$, and $t>0$,
    \begin{equation} \label{eq:X-bound-sup}
      |X^\e_x(t)|_E \leq C \left(1 + t^{-\frac{1}{\rho_{*}}} +  \sup_{s \in [0,t]} |Y^\e_x(s)|_E \right).
    \end{equation}
    \item  For $R_0$ large enough, we have
    \begin{equation}
     \sup_{|x|_E>R_0} \Pro(|X^\e_x(1)|_E>R_0) \leq C\e^{\frac{p}{2}}. \label{eq:X-big-zero-prob}
   \end{equation}
  \end{enumerate}
\end{theorem}

The proof of Theorem \ref{T:BOUNDEDNESSPROPOFX} will be given in Appendix \ref{App:MildSolution}.

 In addition, using an estimate from  \cite{Cerrai-RocknerLDP}, we can show that the probability of $X^\e_x$ staying outside of a bounded set is exponentially small. This result uses the fact that $g$ is bounded, but as \cite{Cerrai-RocknerLDP} demonstrated, stopping time arguments can be used to apply this theorem to the case where $g$ is unbounded.
\begin{theorem}[Theorem 3.2 of \cite{Cerrai-RocknerLDP}] \label{thm:exp-est-outside-ball}
  For any $T>0$, there exist $c_1,c_2, \lambda>0$ such that for any $\delta>0$ and $u \in C([0,T]:E)$ that is adapted to the filtration of $w(t)$,
  \begin{equation*}
    \Pro \left(\sup_{t \in [0,T]} \left|\int_0^t S(t-s)G(u(s))dw(s) \right|_E \geq \delta \right) \leq c_1 \exp \left(-\frac{\delta^2}{c_2 T^\lambda} \right).
  \end{equation*}
\end{theorem}

From this we can immediately get exponential estimates on the probability that $|X^\e_x(T)|_E$ is large.
\begin{corollary} \label{cor:exp-est-X}
  For any $T>0$,
  \begin{equation*}
    \limsup_{R \to \infty}\limsup_{\e \to 0} \sup_{x \in E}\e \log \Pro \left(|X^\e_x(T)|_E \geq R \right) \leq -\infty.
  \end{equation*}
\end{corollary}

\begin{proof}
  By Theorem \ref{T:BOUNDEDNESSPROPOFX}, for any $x \in E$, $\e>0$, $R>0$, and $T>0$,
  \begin{equation*}
    \Pro\left(|X^\e_x(T)|_E \geq R \right) \leq \Pro \left(\sup_{s \in [0,T]} |Y^\e_x(s)|_E \geq \frac{R}{C} -(1 + T^{-\frac{1}{\rho_{*}}}) \right)
  \end{equation*}
  where $Y^\e_x$ is given by \eqref{eq:stoch-conv-def}. By Theorem \ref{thm:exp-est-outside-ball}, letting $\tilde R = \frac{R}{C} -(1 + T^{-\frac{1}{\rho_{*}}})$,
  \begin{align*}
    \Pro\left( |X^\e_x(T)|_E \geq R \right) \leq c_1 \exp \left(-\frac{\tilde{R}^2}{c_2\e T^\lambda} \right).
  \end{align*}

  For arbitrarily large $V>0$, we can choose $R$ large enough so that $\frac{\tilde{R}^2}{c_2 T^\lambda}>V$. Then we can conclude that
  \[\limsup_{\e \to 0} \sup_{x \in E}\e \log \Pro \left(|X^\e_x(T)|_E \geq R \right) \leq -V.\]

  Our result follows because $V>0$ was arbitrary.
\end{proof}

\section{Uniform large deviations principle and metastability}\label{S:LDP}

Define the rate functions for $x \in E$, $T>0$ and $\varphi \in C([0,T]:E)$,
\begin{equation} \label{eq:rate-fct-def}
  I_x^T(\varphi) = \inf\left\{\frac{1}{2}\int_0^T |u(t)|_H^2 dt: u \in L^2([0,T]:H), \varphi = X^{0,u}_x \right\}
\end{equation}
where for any $u \in L^2([0,T]:H)$, $X^{0,u}_x$ solves the controlled problem
\begin{equation} \label{eq:control}
  X^{0,u}_x(t) = S(t)x + \int_0^t S(t-s)F(X^{0,u}_x(s))ds + \int_0^t S(t-s)G(X^{0,u}_x(s))u(s)ds.
\end{equation}

For $x \in E$, $T\geq 0$, and $s \geq 0$,  denote the level sets of the rate function by
\[\Phi_x^T(s) = \{\varphi \in C([0,T]:E): I_x(\varphi) \leq s\}.\]
 The following uniform large deviations result over bounded sets is well known.
\begin{theorem}[See \cite{Cerrai-RocknerLDP,BudhirajaDupuisSalins2018}] \label{thm:LDP}
  For any fixed $T>0$, $X^\e_x$ satisfies a uniform large deviations principle in $C([0,T]:E)$ with respect to the rate function $I_x^T$ uniformly over bounded subsets of $E$.
  \begin{enumerate}
    \item For any $\delta>0$, $s_0>0$, and bounded subset $E_0 \subset E$,
    \begin{align}
      \liminf_{\e \to 0} \inf_{x \in E_0} \inf_{\varphi \in \Phi_x^T(s_0)} \left(\e \log \Pro(|X^\e_x - \varphi|_{C([0,T]:E)}<\delta) + I_x^T(\varphi)  \right) &\geq 0. \label{eq:ldp-low}
    \end{align}
    \item For any $\delta>0$ and $s_0>0$, and bounded subset $E_0 \subset E$,
    \begin{equation} \label{eq:ldp-up}
      \limsup_{\e \to 0} \sup_{x \in E_0} \sup_{s \in [0,s_0]} \left(\e \log \Pro(\dist_{C([0,T]:E)}(X^\e_x, \Phi_x^T(s)) \geq \delta) + s \right) \leq 0.
    \end{equation}
  \end{enumerate}
\end{theorem}

Let us next present our main results of this work on metastability.

Let $x_*$ be an asymptotically stable equilibrium of the unperturbed system {and recall $D$, its domain of attraction,}
\begin{equation}
  D := \{x \in E: \lim_{t \to\infty} |X^0_x(t) - x_*|_E = 0\}.\label{Eq:DomainAttraction}
\end{equation}
\begin{assumption} \label{assum:D-boundary}
  Assume that for every $y \in \partial D$ and $\delta>0$
  \[B(y,\delta) \cap (\bar D)^c \not = \emptyset.\]
\end{assumption}
Assumption \ref{assum:D-boundary} is a technical assumption that prevents pathological cases.

We are interested in the time it takes for the stochastic system to exit this attracting set
\begin{equation*}
  \tau^\e_x: = \inf\{t>0: X^\e_x(t) \not \in D\}.
\end{equation*}

We characterize the asymptotic behavior of the exit time $\tau^\e_x$ and the exit shape $X^\e_x(\tau^\e_x)$ in terms of functionals $V, \tilde{V}_D$, and $\hat{V}_D$, which we call quasipotentials.

  For any $x, y \in E$, define
\begin{equation} \label{eq:quasipotential-def}
  V(x,y) : = \inf \left\{I_x^T(\varphi): T\geq 0, \varphi \in C([0,T]:E), \varphi(0) = x, \varphi(T) = y \right\}.
\end{equation}

Without confusion, for any subsets $D_1, D_2 \subset E$  we define
\begin{equation} \label{eq:quasipotential-def-sets}
  V(D_1,D_2) := \inf_{x \in D_1} \inf_{y \in D_2} V(x,y).
\end{equation}

In addition to the assumptions of Section \ref{S:Assumptions},  we shall also make the following assumption about the domain of attraction $D$ defined in \eqref{Eq:DomainAttraction}.
\begin{assumption}\label{A:AttractionProperty}
For any $x \not \in \bar{D}$, $V(x,\partial D)>0$.
\end{assumption}

\begin{remark} \label{rem:stay-away}
  Assumption \ref{A:AttractionProperty} guarantees that $\partial D$ acts like a separatrix.
  An important consequence of Assumption \ref{A:AttractionProperty} is that for any $x \not \in \bar D$ and $t \geq 0$.
  \begin{equation} \label{eq:stay-away-from-boundary}
    \dist_E(X^0_x(t), \partial D)>0.
  \end{equation}

Because $D$ was defined in \eqref{Eq:DomainAttraction} to be the domain of attraction for $x_*$, $V(x,\partial D)>0$ and \eqref{eq:stay-away-from-boundary}  also holds for all $x \in D$. Indeed, this is the content of Lemma \ref{L:PositiveQuasipotnetial} in Section \ref{S:controlability}.
\end{remark}

Theorem \ref{thm:meanExitAsymptotics} characterizes the logarithmic asymptotic of the exit time $\tau^\e_x$.
\begin{theorem} \label{thm:meanExitAsymptotics}
Let Assumptions \ref{assum:nonlinear}, \ref{A:DiffusionCoeff}, \ref{assum:D-boundary},  and \ref{A:AttractionProperty} hold.  For any $x \in D$,
  \begin{equation*} %\label{eq:MeanExitTime}
    \limsup_{\e \to 0} \e \log \E \tau^\e_x = V(x_*, \partial D)
  \end{equation*}
  and for any $\gamma>0$,
  \begin{equation*} %\label{eq:CDF-exit-time}
    \lim_{\e \to 0} \Pro( V(x_*, \partial D) - \gamma<\e \log \tau^\e_x < V(x_*, \partial D) + \gamma) =1.
  \end{equation*}
\end{theorem}

The proof of Theorem \ref{thm:meanExitAsymptotics} is in Section \ref{S:ExitTimeAsymptotics}.

Let us next analyze the distribution of $X^\e_x(\tau^\e_x)$. Because $X^\e_x(\tau^\e_x)$ takes values in $E=C_{0}([0,L]:\mathbb{R}^r)$, we call this object the exit shape.

As in \cite{f-1988,FWbook,Day1990} define an equivalence relation $\sim$ such that
\begin{equation} \label{eq:equivalence}
  x \sim y \text{ if and only if } V(x,y)  =V(y,x) = 0.
\end{equation}

Any equivalence class defined with respect to this equivalence relation is called a $V$ equivalence class.
\begin{assumption}\label{A:K_equivalenceClassV}
There exists a finite number of  $V$ equivalence classes $K_1, ...., K_N \subset E$ such  that every $\omega$-limit point of $X^0_x(t)$ for $x \in \partial D$ is contained in one of the $K_i$.
\end{assumption}

\begin{remark}
  If $y \in \partial D$ is an equilibrium of the unperturbed system, meaning $X^0_y(t)=y$ for all $t>0$, then $K_i=\{y\}$ should be one of these equivalence classes. Notice that such a $y$ is an unstable equilibrium because $y \in \partial D$ where $D$ is the set of attraction for $x_*$.   If there exists a $y \in \partial D$ and a period $t_0>0$ such that $X^0_y(t_0) = y$, then the set $K_i = \{X^0_y(t): t \in [0,t_0]\}$ is called a limit cycle and could also be an example of a $V$ equivalence class.
\end{remark}

The next lemma demonstrates that each $K_i \subset \partial D$.
\begin{lemma} \label{lem:K_i-in-boundary}
  If $K$ is a $V$ equivalence class such that $K \cap \partial D \not = \emptyset$, then $K \subset \partial D$.
\end{lemma}
\begin{proof}
 { Let $y \in K \cap \partial D$ and $z\in K$. If $z \not \in \partial D$ then Assumption \ref{A:AttractionProperty} and Remark \ref{rem:stay-away} guarantee that $V(z,y)>0$ so $z \not \sim y$. Therefore, $z\in\partial D$ which implies $K \subset \partial D$.}
\end{proof}

%We assume that  there are a finite number of these equivalence classes.

The exit shape asymptotics are described by the quasipotentials $\tilde{V}_{D}(x,y)$ and $\hat{V}_{D}(x,y)$ as defined below. For any $\rho>0$, and $x,y \in \bar{D}$ define
\begin{align} \label{eq:V-D_rho}
  \tilde{V}_D^\rho(x,y) := \inf \Bigg\{I_x^T(\varphi)&: T\geq 0, \varphi \in C([0,T]:E), \varphi(0)=x, \varphi(T)=y,\nonumber\\
   &\varphi(t) \in D \cup B(x,\rho)\cup B(y,\rho)  \text{ for } t \in [0,T] \Bigg\}
\end{align}
and for any $x,y \in \bar{D}$, set
\begin{equation}\label{eq:V-tilde-D}
  \tilde{V}_D(x,y):=\lim_{ \rho \to 0} \tilde{V}^\rho_D(x,y).
\end{equation}

Notice that if $x,y \in D$, then $\tilde{V}_D(x,y)$
has the explicit representation
\begin{align*}
  \tilde{V}_D(x,y) = \inf\Big\{I^T_x(\varphi): &T\geq0, \varphi \in C([0,T]:E), \varphi(0)=x, \varphi(T)=y,\\
   &\varphi(t) \in D \text{ for } t \in [0,T]\Big\},
\end{align*}
because there exists $\rho_0>0$ such that for all $\rho \in (0,\rho_0)$, $B(x,\rho)\cup B(y,\rho) \subset D$. In general, for $0< \rho_1<\rho_2$, and $x, y \in \bar{D}$
 \begin{equation} \label{eq:V-inequality}
   V(x,y) \leq \tilde{V}^{\rho_2}_D(x,y) \leq \tilde{V}^{\rho_1}_D(x,y) \leq \tilde{V}_D(x,y). %\leq V_D(x,y).
 \end{equation}

In addition, let us set
\begin{align}\label{eq:hat-V-D_rho}
  \hat{V}^\rho_D(x, y):=  \inf \Big\{ &I^T_x(\varphi): T\geq0, \varphi(0)=x, \varphi(T)=y, \nonumber\\
  &\varphi(t) \in D \cup B(x,\rho) \cup B(y,\rho) \cup \bigcup_{i=1}^N B(K_i, \rho) \text{ for } t \in [0,T]\Big\}.
\end{align}
and
\begin{equation}\label{eq:V-hat-D}
  \hat{V}_D(x,y):= \lim_{ \rho \to 0} \hat{V}^\rho_D(x,y).
\end{equation}

{Clearly for any $x,y\in\bar{D}$ we have that
\[
\hat{V}_D(x,y)\leq \tilde{V}_D(x,y).
\]
}

As usual, for any $D_1, D_2 \subset \bar{D}$,
\begin{align}
  \tilde{V}_D(D_1,D_2) &:= \inf_{x \in D_1} \inf_{y \in D_2} \tilde{V}_D(x,y),\label{eq:V-tilde-of-sets}\\
  \hat{V}_D(D_1,D_2) &:= \inf_{x \in D_1} \inf_{y \in D_2} \hat{V}_D(x,y),\label{eq:V-hat-of-sets}
\end{align}

In other words, $\tilde{V}^\rho_D(x,y)$ and $\hat{V}^\rho_D(x,y)$ are the minimum actions of any path that connects $x$ and $y$ without leaving $D \cup B(x,\rho)\cup B(y,\rho)$ and $D \cup B(x,\rho)\cup B(y,\rho) \cup \bigcup_{i=1}^N B(K_i, \rho)$ respectively. So when $\rho$ is small and $x$ and $y$ are on the boundary of $D$, the paths are allowed to barely leave $D$, but only when they are near $x$ or $y$ (or one of the $K_i$). The expansion around these points and sets enables us to prove our theorems without assumptions on the regularity of the boundary of $D$.

Now we are ready to state our result on the exit shape  $X^\e_x(\tau^\e_x)$ asymptotics.
\begin{theorem} \label{thm:exit-shape}
  Let Assumptions \ref{assum:nonlinear}, \ref{A:DiffusionCoeff}, \ref{assum:D-boundary},  \ref{A:AttractionProperty}, and \ref{A:K_equivalenceClassV} hold. The exit shape $X^\e_x(\tau^\e_x)$ satisfies a large deviations lower bound with respect to the rate function $\tilde{J}: \partial D \to [0,\infty]$ given by
  \begin{equation*}
    \tilde{J}(y): = \tilde{V}_D(x_*, y) - V(x_*, \partial D).
  \end{equation*}
  In particular, for any $\delta>0$, $x \in D$ and $ y \in \partial D$,
  \begin{equation} \label{eq:shape-LDP-lower}
    \liminf_{\e  \to 0} \e \log \Pro(|X^\e_x(\tau^\e_x) - y|<\delta) \geq  \tilde{J}(y).%-(\tilde{V}_D(x_*,y) - V(x_*, \partial D)).
  \end{equation}
  The exit shape satisfies a large deviations upper bound with respect to the rate function $\hat{J}:  \partial D \to [0,+\infty]$ given by
  \begin{equation*}
    \hat{J}(y) := \hat{V}_D(x_*, y) - V(x_*, \partial D)
  \end{equation*}
  For any $\delta>0, s>0$, and $x \in D$,
  \begin{equation} \label{eq:shape-LDP-upper}
    \limsup_{\e \to 0} \e \log \Pro(\dist_E(X^\e_x(\tau^\e_x), \hat\Psi(s)) \geq \delta) \leq -s
  \end{equation}
  where $\hat \Psi(s) = \{ y \in \partial D: \hat J(y) \leq s\}$.
\end{theorem}
The proof of Theorem \ref{thm:exit-shape} is in Section \ref{S:ExitShape}.

An immediate corollary of the exit shape upper bound \eqref{eq:shape-LDP-upper} is that the limiting distribution of the exit shape $X^\e(\tau^\e_x)$ is concentrated on the set $\hat{\Psi}(0) = \{y \in \partial D: \hat{V}_D(x_*,y) = V(x_*,\partial D)\}$. The following corollary is  analogous to the classical result about the exit place limiting distribution in the finite dimensional case (see Theorems 4.2.1 and 6.5.2 of \cite{FWbook} and Theorem 5.7.11(b) of \cite{DemboZeitouni1997}).
\begin{corollary} \label{cor:exit-shape-limiting-dist}
  For any $\delta>0$, $x \in D$,
  \begin{equation}
    \lim_{\e \to 0} \Pro(\dist_E(X^\e_x(\tau^\e_x),\hat{\Psi}(0))>\delta) = 0.
  \end{equation}
  In particular, if there is a unique minimizer $y_* \in \partial D$ such that $\hat{V}_D(x_*,y_*) = V(x_*,\partial D)$, then $\hat{\Psi}(0) = \{y_*\}$ and $X^\e_x(\tau^\e_x)$ converges to $y_*$ in probability.
\end{corollary}

We will prove that $\hat{\Psi}(0)$ is non-empty and compact in Theorem \ref{thm:V-hat-compact-level-sets} in the following section.

We end this section with two remarks, where we discuss Theorem \ref{thm:exit-shape} as well as the different notions of quasipotentials appearing in the metastability results.
\begin{remark}
It is instructive to comment here on the different notions of quasipotentials appearing in the metastability results of Theorems \ref{thm:meanExitAsymptotics} and \ref{thm:exit-shape}. We use $V$ in Section \ref{S:controlability} to establish certain continuity and lower semicontinuity properties. In addition, $V$ is also used for Theorem \ref{thm:meanExitAsymptotics}. Notice however that
\[
V(x_{*},\partial D)=\tilde{V}_{D}(x_{*},\partial D)=\hat{V}_{D}(x_{*},\partial D)
\]
so any of these quasipotentials can be used for Theorem \ref{thm:meanExitAsymptotics}. %Then, $V_{\bar{D}}$ is used to define the equivalence classes $K_i$.
$\tilde{V}_{D}$ is used for the lower bound of the logarithmic asymptotics of the exit shape, whereas $\hat{V}_{D}$  is used for the upper bound of the logarithmic asymptotics of the exit shape.
\end{remark}

\begin{remark}
Notice that Theorem \ref{thm:exit-shape} does not give a large deviations principle for the exit shape of the random variable $X^\e_x(\tau^\e_x)$. The reason is that in general it is not clear whether, under our current assumptions, the functionals $\tilde{J}(y)$ and $\hat{J}(y)$ are equal. We will explore this possibility in Section \ref{S:ExitShape_AdditionalCond}.
\end{remark}

\section{Controllability properties and the quasipotential}\label{S:controlability}

In the case of finite dimensional diffusions considered by Freidlin and Wentzell \cite{FWbook}, the quasipotential (\ref{eq:quasipotential-def}) is jointly continuous. That is whenever $x_n \to x$ and $y_n \to y$, $V(x_n,y_n) \to V(x,y)$. This continuity, in the finite-dimensional case, is proven by defining a (straight-line) path with arbitrary small action that connects arbitrary close points $x$ and $y$. Unfortunately, $V(x,y)$ cannot be jointly continuous in $E \times E$ in the infinite dimensional case. In particular, the quasipotential $V(x,y)=+\infty$ whenever $y \not \in H^1$ and $x \not =y$. Recall the definition of $H^1$ from \eqref{eq:Sobolev-spaces}. % . %is equal to $+\infty$ on a dense subset of $E\times E$.

It will not be possible in the setting of this paper to connect any two elements of $E$ with a controlled path. However, certain continuity properties of the quasipotential are still true. The purpose of this section is to explore such properties, which, as will be seen later on, are key in the development of a metastability theory in infinite dimensions.

One of the key concepts here is that of controllability. We show that if $X^{0,u}_x(t)$ is a controlled process for some $x \in E$ and $ u \in L^2([0,T]:H)$ and $y \in E$ is close to $x$, then we can build a new control $v \in L^2([0,T]:H)$ such that the controlled process $X^{0,v}_y(t)$ joins $X^{0,u}_x(t)$ after a short period of time. We can also guarantee that the action of $v$ is not much bigger than the action of $u$. This is particularly of interest when $u=0$. We can show that whenever $y \in E$ is close to $x$, we can construct a control $v$ such that $X^{0,v}_y(t) = X^0_x(t)$ after a short amount of time. This means that while it is usually impossible to make a path that connects $x$ and $y$, it is possible to build a path that connects $y$ to the unperturbed dynamical system that starts at $x$.

Useful auxiliary properties of the skeleton %limiting controlled
equation \eqref{eq:control} $X^{0,u}_x$ are presented in Subsection \ref{SS:LimitingControl}. The controllability result is shown in Subsection \ref{SS:controllability}. Subsection \ref{SS:ContinuityQuasipotential} discusses continuity properties of the quasipotential $V$, defined by (\ref{eq:quasipotential-def}). Then, in Subsection \ref{SS:RegularityQuasipotential} we discuss regularity properties of the quasipotential $V$, defined by (\ref{eq:quasipotential-def}). {In Subsection \ref{SS:ContinuityQuasipotentialHatTilde} we prove lower semi-continuity and compactness of level sets for $\tilde{V}_{D}$ and $\hat{V}_{D}$, defined in (\ref{eq:V-tilde-D}) and (\ref{eq:V-hat-D}) respectively. In Subsection \ref{SS:EquivalenceClasses} we discuss properties of the $V$ equivalence classes $K_{i}$.}

\subsection{Some properties of the %limiting controlled
skeleton equation \eqref{eq:control} $X^{0,u}_x$}\label{SS:LimitingControl}

At this point, we collect some useful results for the skeleton equation \eqref{eq:control} $X^{0,u}_x$.
\begin{theorem}[A priori bounds] \label{THM:CONTROL-BOUNDS}
  There exists $C >0$ such that for any $T>0$, $x \in E$, and $u \in  L^2([0,T]:H)$,
  \begin{equation} \label{eq:control-sup-time-bound}
    |X^{0,u}_x|_{C([0,T]:E)} \leq C (1+|x|_E +T^{\frac{1}{4}}|u|_{L^2([0,T]:H)})
  \end{equation}
  and for any $t>0$,
  \begin{equation} \label{eq:control-sup-x-bound}
    \sup_{x \in E}|X^{0,u}_x(t)|_E \leq C (1 + t^{-\frac{1}{\rho_{*}}} +  |u|_{L^2([0,t]:H)}).
  \end{equation}
\end{theorem}

The proof of Theorem \ref{THM:CONTROL-BOUNDS} is in Appendix \ref{App:MildSolution}.

\begin{theorem}[Continuity with respect to initial condition and control, Theorem 8.4 in \cite{BudhirajaDupuisSalins2018}] \label{thm:continuity}
  For any fixed $T>0$, if $x_n \to x \in E$  and $u_n \rightharpoonup u$ weakly in $L^2([0,T]:H)$, then
  \begin{equation*}
    \lim_{n \to \infty} |X^{0,u_n}_{x_n} - X^{0,u}_x|_{C([0,T]:E)} = 0.
\end{equation*}
\end{theorem}
\begin{theorem}[Compactness, Theorem 8.5 in \cite{BudhirajaDupuisSalins2018}] \label{thm:compact}
  Let $R>0$,  $N>0$, and $T>0$. Whenever $|x_n|_E \leq R$ is a bounded sequence and $\{u_n\} \subset L^2([0,T]:H)$ with $|u_n|_{L^2([0,T]:H)}\leq N$, then there exists a subsequence such that for any $t\in (0,T]$, $X^{0,u_n}_{x_n}(t)$ converges in $E$.
\end{theorem}

A useful corollary of Theorems \ref{THM:CONTROL-BOUNDS} and \ref{thm:compact} is the following compactness result.
\begin{corollary} \label{cor:pre-compact-all-init-cond}
  For any $t>0$ and $N >0$, the set
  \begin{equation*}
    \left\{ X^{0,u}_x(t): x \in E \text{ and } |u|_{L^2([0,t]:H)}\leq N \right\}
  \end{equation*}
  is pre-compact in $E$.
\end{corollary}
\begin{proof}
  Let $x_n \in E$ and $u_n \in L^2([0,t]:H)$ with $|u_n|_{L^2([0,t]:H)}\leq N$. Let $y_n = X^{0,u_n}_{x_n}(t/2)$. By \eqref{eq:control-sup-x-bound}, there exists $R>0$ such that $|y_n|_E\leq R$. We define $\tilde{u}_n$ to be the translation $\tilde{u}_n(s) = u_n(s+t/2)$ for $s \in [0,t/2]$. Then it is clear that $X^{0,u_n}_{x_n}(t) = X^{0,\tilde{u}_n}_{y_n}(t/2)$. By Theorem  \ref{thm:compact}, there is a convergent subsequence.
\end{proof}

Another useful result of Theorems \ref{thm:continuity} and \ref{thm:compact} is Lemma \ref{L:PositiveQuasipotnetial}.
\begin{lemma}\label{L:PositiveQuasipotnetial}
For all $x\in D$, we have that $V(x,\partial D)>0$.
\end{lemma}
\begin{proof}[Proof of Lemma \ref{L:PositiveQuasipotnetial}]
By assumption, $x_*$ is an asymptotically stable equilibrium. This means that for all small $\rho>0$ there exists $\delta>0$ such that $\{x: |x-x_*|_E < 2 \rho\} \subset D$ and  whenever $|x-x_*|_E<\delta$, $|X^0_x(t) - x_*|_E< \rho$ for all $t>0$ and $\lim_{t \to \infty} |X^0_x(t) - x_*|_E = 0$. Additionally, there exists $T_2>0$ such that for any $x \in E$ such that $|x-x_*|_E < \delta$, $|X^0_x(T_2) - x_*|_E< \frac{\delta}{2}$. This uniformity is due to the compactness of the set $\{X^0_x(1): |x-x_*|_E<\delta\}$ (see Theorem \ref{thm:compact}).
By Theorem \ref{thm:continuity}, there exists $\upsilon>0$ such that if $|u|_{L^2([0,T]:H)}<\upsilon$ and $|x-x_*|_E< \delta$, then
\[\sup_{t \in [0,T_2]}|X^{0,u}_x(t) - x_*|_E< 2\rho \text{ and } |X^{0,u}_x(T_2) - x_*|_E< \delta.\]
Now, given any $x \in D$, because $D$ is the domain of attraction to $x_*$, there exists $T_1>0$ such that $|X^0_x(T_1) - x_*|_E< \frac{\delta}{2}$. By Theorem \ref{thm:continuity}, by possibly decreasing $\upsilon$ we can guarantee that for all $|u|_{L^2([0,T_1]:H)}< \upsilon$, $X^{0,u}_x(t) \in D$ for $t \in [0,T_1]$, and $|X^{0,u}_x(T_1) - x_*|_E< \delta$.
Then we know that $|X^{0,u}_x(T_1 + t)-x_*|_E<\rho$ for $t \in [0,T_2]$ and $|X^{0,u}_x(T_1  + T_2)-x_*|_E<\delta$ because $|X^{0,u}_x(T_1) - x_*|_E<\delta$. We can repeat this argument to show that for any $k \in \mathbb{N}$, and $u$ satisfying $|u|_{L^2([0,T_1 + (k+1)T_2]:H)}< \upsilon$, it follows that
$|X^{0,u}_x(T_1 + kT_2 + t) - x_*|_E<\rho $ for $t \in T_2$ and $|X^{0,u}_x(T_1 + (k+1)T_2) - x_*|_E<\delta$.
Now we can guarantee that for any $|u|_{L^2([0,+\infty):H)}< \upsilon$,
\[X^{0,u}_x(t) \in D \text{ for all } t\in [0,+\infty).\]
In particular this proves that $V(x,\partial D)>\upsilon$.
\end{proof}

In Theorem \ref{T:OpenAttractionSets} we prove that the set of attraction $D$ of $x_{*}$, an asymptotically stable equilibrium of the unperturbed system, given in \eqref{Eq:DomainAttraction} are open. A similar result for the case of the Allen-Cahn equation can be also found in \cite{FarisLasinio1982}.
\begin{theorem}\label{T:OpenAttractionSets}
  The set of attraction $D$ is open.
\end{theorem}

\begin{proof}
  Because $x_*$ is asymptotically stable, there exists $\rho>0$ such that the open ball $\{x \in E: |x-x_{*}|_E<\rho\} \subset D$.
  We need to show that if $y \in D$, then there exists $\delta>0$ such that $\{x \in E: |x-y|_E<\delta\}\subset D$. So, let $y \in D$. Because $X^0_y(t) \to x_*$, there exists $T_1 >0$ such that $|X^0_y(T_1)-x_{*}|_E < \frac{\rho}{2}$. By Theorem \ref{thm:continuity}, we can find $\delta>0$ such that for all $x$ satisfying $|x-y|_E < \delta$ it follows that $|X^0_x - X^0_y|_{C([0,T_{1}]:E)} < \frac{\rho}{2}$. Therefore, for all $x$ satisfying $|x-y|_E<\delta$, $|X^0_{x}(T_1) - x_{*}|< \rho$. Then because the $\rho$ ball around $x_{*}$ is in $D$, it follows that for all $x \in E$ such that $|x-y|_E<\delta$, $\lim_{t\to \infty} X^0_x(t) = x_{*}$. This proves that $D$ is open.
\end{proof}

\subsection{Controllability}\label{SS:controllability}

The main result of this subsection is the controlability Theorem \ref{thm:control}.
\begin{theorem} \label{thm:control}
  For any $\gamma>0$, $R>0$ and $N>0$, there exists $\delta>0$ such that whenever $|x-y|_E<\delta$, $T\geq \delta$,  $\frac{1}{2}|u|_{L^2([0,T]:H)}^2 \leq N$, and $|x|_E \leq R$, there exists $v \in L^2([0,T]:H)$ such that $X^{0,v}_y(t) = X^{0,u}_x(t)$ for all $t\in [\delta,T]$ and $\frac{1}{2}|v|_{L^2([0,T]:H)} \leq \frac{1}{2}|u|_{L^2([0,T]:H)}^2 + \gamma$.
\end{theorem}
\begin{proof}
  Fix $\gamma>0$, $R>0$, and $N>0$. Let $x,y \in E$ satisfying $|x|_E \leq R$ and $|y|_E \leq R$ and $u \in L^2([0,T]:H)$ satisfying $\frac{1}{2}|u|^2_{L^2([0,T]:H)}\leq N$. We will build a control $v$ such that $X^{0,v}_y(t) = X^{0,u}_x(t)$ after a certain time.

    Let $\varphi_n: [0,T] \to E$ solve the Cauchy problem
\[\begin{cases}
  \displaystyle{\frac{d}{dt} \varphi_n(t) = A \varphi_n(t) + F(X^{0,u}_x(t)) + G(X^{0,u}_x(t))u(t)}\\
   \displaystyle{\hspace{3cm}- \frac{\varphi_n(t) - X^{0,u}_x(t)}{|\varphi_n(t) - X^{0,u}_x(t)|_E \vee \frac{1}{n}},}\\
   \varphi_n(0) = y.
  \end{cases}
\]
Notice that a solution $\varphi_n$ exists and is unique because the mapping $\varphi \mapsto \frac{\varphi - X^{0,u}_x(t)}{|\varphi - X^{0,u}_x(t)|_E\vee \frac{1}{n}}$ is Lipschitz continuous.

Let $\tau_n = \inf\{t>0: |\varphi_n(t) - X^{0,u}_x(t)| \leq \frac{1}{n}\}$.
Notice that by uniqueness $\varphi_n(t) = \varphi_m(t)$ for $m>n$ and $t \in [0,\tau_n]$.
Furthermore, we can calculate that
\[\frac{d}{dt}\left(\varphi_n(t) - X^{0,u}_x(t)\right) = A\left(\varphi_n(t) - X^{0,u}_x(t)\right)- \frac{\varphi_n(t) - X^{0,u}_x(t)}{|\varphi_n(t) - X^{0,u}_x(t)|_E \vee \frac{1}{n}}. \]

For any $\delta_t \in \partial |\varphi_n(t) - X^{0,u}_x(t)|_E$,
\[\frac{d^{-}}{dt}|\varphi_n(t) - X^{0,u}_x(t)|_E \leq \left<A\left(\varphi_n(t) - X^{0,u}_x(t)\right), \delta_t \right>_{E,E_\star} - \frac{|\varphi_n(t) - X^{0,u}_x(t)|_E}{|\varphi_n(t) - X^{0,u}_x(t)|_E \vee \frac{1}{n}}.\]

For $t \in [0,\tau_n]$, therefore, it holds that
\[\frac{d^{-}}{dt} |\varphi_n(t) - X^{0,u}_x(t)|_E \leq -1\]

Because $|\varphi_n(0) - X^{0,u}_x(0)|_E = |y-x|_E$, it follows that for all $t \in [0,\tau_n]$
\[0\leq |\varphi_n(t) - X^{0,u}_x(t)|_E \leq |y-x|_E - t.\]

In particular, this proves that $\tau_n \leq |y-x|_E$ for all $n$. Because $\tau_n \leq \tau_{n+1}$, there exists
\[\tau = \lim_{n \to \infty} \tau_n \leq |y-x|_E.\]

Therefore, we can define
\[\varphi(t) := \begin{cases}
                     \varphi_n(t) & \text{ for } t \in [0,\tau_n],\\
                     X^{0,u}_x(t) & \text{ for } t > \tau
                \end{cases}.\]

This process solves the Cauchy problem
\begin{align*}
  \frac{d}{dt} \varphi(t) =   &A \varphi(t) + F(X^{0,u}_x(t)) + G(X^{0,u}_x(t))u(t)\\
                               & - \frac{\varphi(t) - X^{0,u}_x(t)}{|\varphi(t) - X^{0,u}_x(t)|_E}\mathbbm{1}_{\{|\varphi(t) - X^{0,u}_x(t)|_E >0\}},
                               \end{align*}

Define
\begin{align*}
  v(t) := &G^{-1}(\varphi(t))(F(X^{0,u}_x(t)) - F(\varphi(t))) + G^{-1}(\varphi(t)) G(X^{0,u}_x(t))u(t)\\
  &- G^{-1}(\varphi(t)) \left(\frac{\varphi(t) - X^{0,u}_x(t)}{|\varphi(t) - X^{0,u}_x(t)|_E}\right)\mathbbm{1}_{\{|\varphi(t) - X^{0,u}_x(t)|_E >0\}}
\end{align*}
and notice that $X^{0,v}_y(t) = \varphi(t)$.

This controlled system has the property that $X^{0,v}_y(t) = X^{0,u}_x(t)$ for all $t> \tau$. Because $\tau\leq |y-x|_E$, $X^{0,v}_y(t) = X^{0,u}_x(t)$ for all $t \geq |y-x|_E$.

  Now we analyze the $L^2([0,T]:H)$ norm of $v(t)$. First, observe that
  \[G^{-1}(X^{0,v}_y(t))G(X^{0,u}_x(t))u(t) = G^{-1}(X^{0,v}_y(t)) (G(X^{0,u}_x(t)) - G(X^{0,v}_y(t)))u(t) +u(t). \]

  By \eqref{eq:G-inv-norm} and \eqref{eq:g-Lip},
  \[|G^{-1}(X^{0,v}_y(t))G(X^{0,u}_x(t))u(t)|_H \leq \left( 1+ \frac{\kappa|X^{0,u}_x(t) - X^{0,v}_y(t)|_E}{\kappa_0}\right)|u(t)|_H.\]

  Therefore, due to the fact that $|X^{0,u}_x(t) - X^{0,v}_y(t)|_E \leq \delta \mathbbm{1}_{\{t\leq \delta\}}$
  \begin{equation} \label{eq:GGu-bound}
    |G^{-1}(X^{0,v}_y(t))G(X^{0,u}_x(t))u(t)|_H \leq  \left( 1+ \frac{\kappa \delta\mathbbm{1}_{\{t\leq \delta\}}}{\kappa_0}\right)|u(t)|_H.
  \end{equation}

  Using the fact that for any $x \in E$, $|x|_H \leq C |x|_E$,
  \begin{align} \label{eq:frac-bound}
    &\left|G^{-1}(X^{0,v}_y(t))\frac{X^{0,v}_y(t) - X^{0,u}_x(t)}{|X^{0,v}_y(t) - X^{0,u}_x(t)|_E} \right|_H\mathbbm{1}_{\{|X^{0,v}_y(t) - X^{0,u}_x(t)|_E>0\}} \nonumber \\
    &\leq \frac{1}{\kappa_0} \frac{|X^{0,v}_y(t) - X^{0,u}_x(t)|_H}{|X^{0,v}_y(t) - X^{0,u}_x(t)|_E} \mathbbm{1}_{\{t\leq \delta\}} \nonumber \\
    &\leq \frac{C}{\kappa_0}\mathbbm{1}_{\{t\leq \delta\}}.
  \end{align}

  By the a-priori bounds \eqref{eq:control-sup-time-bound}, there exists $\tilde{R}$ such that $|X^{0,u}_x(t)|_E, |X^{0,v}_y(t)|_E\leq \tilde{R}$ for all $t \in [0,T]$.
  Using the fact that $F$ is locally Lipschitz continuous (see Assumption \ref{assum:nonlinear}), there exists a $C=C(\tilde R)>0$ such that $|F(X^{0,u}_x(t)) - F(X^{0,v}_y(t))|_E \leq C|X^{0,u}_x(t)-X^{0,v}_y(t)|_E$. Therefore, also taking into account the continuous embedding of $E$ into $H$
 \begin{align} \label{eq:G-F-term}
    |G^{-1}(X^{0,v}_y(t))(F(X^{0,u}_x(t)) - F(X^{0,v}_y(t)))|_H &\leq \frac{C}{\kappa_0}|X^{0,u}_x(t) - X^{0,v}_y(t)|_E \nonumber\\
   & \leq \frac{C\delta}{\kappa_0}  \mathbbm{1}_{\{t \leq \delta\}}.
  \end{align}

  Combining \eqref{eq:GGu-bound}, \eqref{eq:frac-bound}, and \eqref{eq:G-F-term}, for any $t \in [0,T]$,
  \begin{equation*}% \label{eq:v-H-bound}
    |v(t)|_H \leq \left( 1+ \frac{\kappa \delta\mathbbm{1}_{\{t\leq \delta\}}}{\kappa_0}\right)|u(t)|_H + \frac{C(1+\delta)}{\kappa_0} \mathbbm{1}_{\{t\leq \delta\}}
  \end{equation*}

  Simple calculations show that for small $\delta$,
  \begin{align*}
    \frac{1}{2} \int_0^T |v(t)|_H^2 dt \leq \frac{1}{2} \int_0^T |u(t)|_H^2 dt  + C\delta^{\frac{1}{2}}(N + \sqrt{N} +1 ).
  \end{align*}
  We can choose $\delta_0$ small enough so that whenever $\delta = |x-y|_E< \delta_0$ the remainder is smaller than $\gamma$.
\end{proof}

\subsection{Continuity properties of the quasipotential  $V$}\label{SS:ContinuityQuasipotential}
We collect several results about the continuity, semicontinuity, and compactness of the quasipotential $V(\cdot,\cdot)$ defined in \eqref{eq:quasipotential-def}. The three main theorems in this section are Theorems \ref{thm:quasipoential-cont}, \ref{thm:quasipotential-comp}, and \ref{thm:quasipotential-lsc}. %The proofs of these theorems follow a series of preliminary lemmas. %We present some useful corollaries of these results at the end of the section.

\begin{theorem} \label{thm:quasipoential-cont}
  The quasipotential $V(\cdot, \cdot)$ is a continuous function from $E\times H^1 \to [0,+\infty)$.
  That is, whenever $|x_n - x|_{E} \to 0$ and $|y_n-y|_{H^1} \to 0$,
  \begin{equation} \label{eq:quasi-cont}
    \lim_{n \to \infty} V(x_n,y_n) = V(x,y).
  \end{equation}
\end{theorem}

\begin{theorem} \label{thm:quasipotential-comp}
  For any compact set $K \subset E$ and $s\geq 0$, the union of sublevel sets of the quasipotential
  \begin{equation} \label{eq:quasi-level-sets}
    \bigcup_{x \in K}\left\{y \in E: V(x,y) \leq s\right\}
  \end{equation}
  is a compact subset of $E$.
\end{theorem}
\begin{theorem} \label{thm:quasipotential-lsc}
  The quasipotential is a lower-semicontinuous generalized function from $E\times E \to [0,+\infty]$ in the sense that whenever $|x_n - x|_E \to 0$ and $|y_n -y|_E \to 0$,
    \begin{equation} \label{eq:quasi-lsc}
      V(x,y) \leq \liminf_{n \to \infty} V(x_n,y_n).
    \end{equation}
\end{theorem}

The proofs of Theorems \ref{thm:quasipoential-cont}, \ref{thm:quasipotential-comp}, and \ref{thm:quasipotential-lsc} will follow after a few  preliminary results.

\begin{lemma} \label{lem:control-prob-cont}
  Assume that $T_n >0$ and $T_n \to 0$, $x_n$ is a convergent sequence in $E$ with limit $x$,  and $u_n \in L^2([0,T_n]:H)$ with $\sup_n |u_n|_{L^2([0,T_n]:H)} <+\infty$. Then
  \begin{equation*}
    \lim_{n \to \infty} \left|X^{0,u_n}_{x_n}(T_n) -x\right|_E = 0.
  \end{equation*}
\end{lemma}
\begin{proof}
  Observe that
  \begin{align*}
    X^{0,u_n}_{x_n}(T_n) -x = &S(T_n)x_n - x + \int_0^{T_n} S(T_n-s)F(X^{0,u_n}_{x_n}(s))ds \\
    &+ \int_0^{T_n}S(T_n-s)G(X^{0,u_n}_{x_n}(s))u_n(s)ds.
  \end{align*}
  By \eqref{eq:control-sup-time-bound} and \eqref{eq:F-growth}, there exists $C>0$ such that $ | S(T_n-s)F(X^{0,u_n}_{x_n}(s))|_E \leq C$. With this estimate and \eqref{eq:Y-bound} it follows that
  \[|X^{0,u_n}_{x_n}(T_n) -x|_E \leq |S(T_n)x_n - x|_E + CT_n + CT_n^{\frac{1}{4}}|u_n|_{L^2([0,T_n]:H)}.\]
  The right hand side converges to zero because $x_n \to x$ and the semigroup is $C_0$ continuous.
\end{proof}

\begin{lemma} \label{lem:quasipotential-finite}
  $V(x,y)<+\infty$ if and only if $x =y$ or $y \in H^1$.
\end{lemma}
\begin{proof}

  By the definition \eqref{eq:quasipotential-def}, $V(x,x) =0$ for all $x \in E$ because $X^0_x(0)=x$. If $x \not =y$ and
   $V(x,y)<+\infty$, then there exists $t>0$ and $u \in L^2([0,t]:H)$ such that $X^{0,u}_x(t)=y$. In particular,
  \[y = S(t)x + \int_0^t S(t-s)F(X^{0,u}_x(s))ds + \int_0^t S(t-s)G(X^{0,u}_x(s))u(s)ds.\]
  Because $t>0$, Lemmas \ref{lem:conv-reg} and \ref{lem:semigroup-reg} guarantee that $y \in H^1$.

  Now we show that whenever $y \in H^1$, $V(x,y)<+\infty$. Let $y \in H^1$ and %as in the proof of the previous lemma
  let
  \begin{equation*}
    \psi(t) := \begin{cases}
      X^0_y(t) & \text{ if } t \in [0,1]\\
      X^0_y(2-t) & \text{ if } t \in (1,2].
    \end{cases}
  \end{equation*}

  As we argue in the proof of Theorem \ref{thm:reversed}, $\psi = X^{0,u}_y$ where
  \begin{equation*}
    u(t) = \begin{cases}
      0 & \text{ if } t \in [0,1]\\
      -2G^{-1}(X^0_y(2-t))(AX^0_{y}(2-t) + F(X^0_{y}(2-t))) & \text{ if } t \in (1,2].
    \end{cases}
  \end{equation*}
  and $\frac{1}{2}\int_0^2 |u(s)|_H^2 ds <+\infty$.
     Similarly to the argument in the proof of   Theorem \ref{thm:control}, define $v \in L^2([0,T]:H)$ and $X^{0,v}_x(t)$ such that

  \begin{align*}
    v(t):= &G^{-1}(X^{0,v}_x(t))G(X^{0,u}_y(t))u(t) + G^{-1}(X^{0,v}_x(t))(F(X^{0,u}_y(t)) - F(X^{0,v}_x(t))) \nonumber\\
    &-G^{-1}(X^{0,v}_x(t)) \frac{(X^{0,v}_x(t) - X^{0,u}_y(t))|x-y|_E}{|X^{0,v}_x(t) - X^{0,u}_y(t)|_E}\mathbbm{1}_{\{|X^{0,v}_x(t)-X^{0,u}_y(t)|_E>0\}}.
  \end{align*}

  Notice that this is slightly different than the control in   Theorem \ref{thm:control}. We calculate that
  \begin{align*}
    \frac{d^{-}}{dt}|X^{0,v}_x(t) - X^{0,u}_y(t)|_E \leq -|x-y|_E \mathbbm{1}_{\{|X^{0,v}_x(t) - X^{0,u}_y(t)|_E>0\}}.
  \end{align*}
  Because $|X^{0,v}_x(0) - X^{0,u}_y(0)|_E = |x-y|_E$, these calculations show that $X^{0,v}_x(t) = X^{0,u}_y(t)$ for all $t\geq1$. In particular, $X^{0,v}_x(2) = y$.
  In this case, $\frac{1}{2}\int_0^2 |v(s)|_H^2ds$ is not necessarily small (if $|x-y|_E$ is big), but it is necessarily finite. Therefore,
  $V(x,y)<+\infty$.
\end{proof}

\begin{lemma} \label{lem:quasipotential-to-0}
  Whenever $y, y_n \in H^{1}$ and  $|y_n -y|_{H^1} \to 0$, we have
  \begin{equation*}
    \lim_{n \to \infty} V(y,y_n) = 0.
  \end{equation*}
\end{lemma}

\begin{proof}[Proof of Lemma \ref{lem:quasipotential-to-0}]
  We construct trajectories $\varphi_n \in C([0,T_n]:E)$ satisfying $\varphi_n(0) = y$, $\varphi_n(T_n) = y_n$ and $I^{T_n}_{y}(\varphi_n) \to 0$. In the finite dimensional setting studied by Freidlin and Wentzell this could be accomplished using a straight line connecting $y$ and $y_n$. In our setting, however, the unboundedness of $A$ prevents straight lines from being helpful.

  Let $y_n,y \in H^1$ with $|y_n -y|_{H^1} \to 0$. Let $\delta_n =  |y-y_n|_{E}$ and note that $\delta_n \to 0$ by Sobolev embedding.
  Define the trajectory
  \begin{equation*}
    \psi_n(t) = \begin{cases}
      X^0_{y_n}(t) & \text{ if } t \in [0,\delta_n/2]\\
      X^0_{y_n}(\delta_n -t) & \text{ if } t \in (\delta_n/2,\delta_n].
    \end{cases}
  \end{equation*}

  Such a trajectory starts at $y_n \in H^{1}$, follows the unperturbed dynamics for $\delta_n/2$ time, and then reverses, retracing its trajectory until it reaches back to $y_n$.
  As discussed in the proof of Theorem \ref{thm:reversed}, $\psi_n= X^{0,u_n}_{y_n}$ where
  \begin{equation*}
    u_n(t) = \begin{cases}
      0 & \text{ if } t \in [0,\delta_n/2]\\
      -2G^{-1}(X^0_{y_n}(\delta_n -t))(AX^0_{y_n}(\delta_n-t) + F(X^0_{y_n}(\delta_n-t))) & \text{ if } t \in (\delta_n/2,\delta_n].
    \end{cases}
  \end{equation*}
  %Then the controlled trajectory $X^{0,u_n}_{y_n}$ follows the unperturbed trajectory for $\delta_n/2$ time and then it backtracks, following its revered trajectory until it reaches back to $y_n$ at time $\delta_n$.

  The energy of this control is
  \begin{equation*}
    \frac{1}{2}\int_0^{\delta_n} |u_n(s)|_H^2 ds = 2\int_0^{\delta_n/2} |G^{-1}(X^0_{y_n}(t))(AX^0_{y_n}(t) + F(X^0_{y_n}(t)))|_H^2 dt.
  \end{equation*}

  By   \eqref{eq:control-bound},
  \begin{equation*}
    \frac{1}{2}\int_0^{\delta_n} |u_n(s)|_H^2 ds \leq C \left(  |y_n|_{H^1}^2 - |S(\delta_n/2)y_n|_{H^1}^2 \right) + C\delta_n \left(1 + |X^0_{y_n}|_{C([0,\delta_n/2]:E)}^{1+\rho_{*}} \right)^2
  \end{equation*}

Because $|y_n - y|_{H^1}\to 0$, $\delta_n \to 0$, and $S(t)$ is a $C_0$ semigroup on $H^1$,
$$\lim_{n \to \infty} |S(\delta_n/2)y_n)|_{H^1} = \lim_{n \to \infty}|y_n|_{H^1} = |y|_{H^1}.$$

  Furthermore, by Lemma \ref{lem:control-prob-cont}, $|X^0_{y_n}|_{C([0,\delta_n/2]:E)} \to |y|_E$.
  From these calculations, we can see that
  \begin{equation*}
    \lim_{n \to \infty} \frac{1}{2}\int_0^{\delta_n} |u_n(s)|_H^2 ds =0.
  \end{equation*}

  Of course, these $\psi_n$ are not exactly the trajectories that we want for the proof because they do not connect $y$ and $y_n$. So we construct $v_n$ as in Theorem \ref{thm:control} so that $X^{0,v_n}_y(\delta_n) = X^{0,u_n}_{y_n}(\delta_n) = y_n$. By that theorem it is clear that
  \begin{equation*}
    \lim_{n \to \infty} \int_0^{\delta_n}|v_n(t)|_H^2 dt = 0.
  \end{equation*}
  By definition, $V(y,y_n) \leq \frac{1}{2}\int_0^{\delta_n}|v_n(t)|_H^2 dt\to 0.$
\end{proof}

\begin{lemma} \label{lem:quasipotential-continuity-in-1st}
  If $x_n \to x \in E$ and $y \in H^1$, then
  \begin{equation*}
    \lim_{n \to \infty} V(x_n, y) = V(x,y).
  \end{equation*}
\end{lemma}

\begin{proof}
  This is a consequence of Theorem \ref{thm:control}. Let $\eta>0$.  We showed in Lemma \ref{lem:quasipotential-finite} that $V(x,y)<+\infty$. If $x \not =y$, then it is clear from definition \eqref{eq:quasipotential-def} that there exists $T>0$ and $u \in L^2([0,T]:H)$ such that $X^{0,u}_x(T)=y$ and $\frac{1}{2}\int_0^T |u(s)|_H^2 ds < V(x,y) + \eta/2$. If $x=y \in H^1$, then arguing using a reversed process as in the proof of Lemma \ref{lem:quasipotential-to-0}, we can show that the same property holds.

  By Theorem \ref{thm:control}, there exists $N \in \mathbb{N}$ large enough so that for any $n >N$ there exists $v_n \in L^2([0,T]:H)$ such that $X^{0,v_n}_{x_n}(T) = X^{0,u}_x(T) = y$ and
  \[\frac{1}{2}\int_0^T |v_n(s)|_H^2 ds \leq \frac{1}{2}\int_0^T |u(s)|_H^2 ds + \frac{\eta}{2} < V(x,y) + \eta.\]

  This proves
  \begin{equation*} %\label{eq:V-1st-conv-upper-bound}
    \lim_{n \to \infty} V(x_n,y) \leq V(x,y).
  \end{equation*}

  We can use a similar argument to show that
  \begin{equation*} %\label{eq:V-1st-conv-lower-bound}
    \lim_{n \to \infty} V(x_n,y) \geq V(x,y).
  \end{equation*}

 By the definition of the quasipotential, there exists a sequence $T_n \in [0,\infty)$ and $u_n \in L^2([0,T_n]:H)$ such that $X^{0,u_n}_{x_n}(T_n) = y$ and
  \begin{equation*}
    \frac{1}{2}\int_0^{T_n}|u_n(s)|_H^2 ds < V(x_n,y) + \frac{1}{n}.
  \end{equation*}

  We now explore two cases: $\limsup_n T_n =0$ and $\limsup_n T_n >0$.   Assume first that $\limsup_{n \to \infty} T_n=0$. By Lemma \ref{lem:control-prob-cont} and the fact that $X^{0,u_n}_{x_n}(T_n)=y$,
  \begin{equation*} %\label{eq:y-minus-x}
    |y-x|_E =0.
  \end{equation*}
  It is clear that $V(x,y) = V(y,y)=0$.

  On the other hand, if $\limsup_n T_n >0$, then we can find a subsequence of $T_n$ and a $\delta>0$ such that $T_n>\delta$ for all $n$. Fix $\eta>0$.
  By possibly decreasing $\delta$, Theorem \ref{thm:control} guarantees that whenever $|x_n-x|<\delta$, there exists a $v_n \in L^2([0,T_n]:H)$ such that $\frac{1}{2}\int_0^{T_n}|v_n(s)|_H^2 ds\leq \frac{1}{2}\int_0^{T_n}|u_n(s)|_H^2 ds + \frac{\eta}{2}$ and $X^{0,v_n}_{x}(T_n) = X^{0,u_n}_{x_n}(T_n) = y$. We can conclude that for large $n$,
  \[V(x,y) \leq V(x_n,y) + \eta.\]
\end{proof}

%Now, we are in position to prove Theorem \ref{thm:quasipoential-cont}.
\begin{proof}[Proof of Theorem \ref{thm:quasipoential-cont}]
   By Lemma \ref{lem:quasipotential-finite}, $V(\cdot,\cdot)$ is finite on $E\times H^1$.

  To prove \eqref{eq:quasi-cont}, let  $x_n \to x$ in $E$ and $y_n \to y$ in $H^{1}$. The definition of the quasipotential guarantees that
  \begin{align}
    &V(x_n,y_n) \leq V(x_n,y) + V(y,y_n), \label{eq:quasipotential-bound1}\\
    &V(x_n,y) \leq V(x_n,y_n) + V(y_n,y). \label{eq:quasipotential-bound2}
  \end{align}

  Lemma \ref{lem:quasipotential-continuity-in-1st} guarantees that $V(x_n,y) \to V(x,y)$ and $V(y_n,y) \to V(y,y) = 0$. Lemma \ref{lem:quasipotential-to-0} guarantees that $V(y,y_n) \to 0$. Therefore \eqref{eq:quasi-cont} holds.
\end{proof}

\begin{proof}[Proof of Theorem \ref{thm:quasipotential-comp}]
  Let $K \subset E$ be compact and $s \geq 0$.
  Let $x_n \in K$ and $y_n \in E$ be any sequences such that $V(x_n,y_n) \leq s$. Because $K$ is compact, there exists a subsequence (relabeled $x_n$) such that $x_n \to x$ in $K$. By the definition of $V$ there exist $T_n>0$ and $u_n \in L^2([0,T_n]:H)$ such that $X^{0,u_n}_{x_n}(T_n) = y_n$ and $\frac{1}{2}\int_0^{T_n} |u_n(t)|_H^2 dt\leq s+ \frac{1}{n}$.

  We divide this argument into two cases: $\lim_n T_n =0$ and $\limsup_n T_n>0$.

  If $\lim_n T_n=0$, then by Lemma \ref{lem:control-prob-cont} it follows that
  \begin{equation*}
    \lim_{n \to \infty}|y_n-x|_E = 0
  \end{equation*}
  implying $y=x$. Since $V(x,x)=0$, $y=x$ is in the sublevel set of $V(x,\cdot)$.

On the other hand, if $\limsup_n T_n>0$ then there exists a subsequence of $T_n$ and a $t_0>0$ such that $T_n> 2 t_0$ for all $n \geq N$. Note that
\[y_n= X^{0, \tilde u_n}_{z_n}(t_0)\]
where $z_n = X^{0,u_n}_{x_n}(T_n - t_0)$ and $\tilde{u}_n(t) = u_n(t+t_0)$. By (\ref{eq:control-sup-x-bound}) we get that
\[|z_n|_E \leq C(1 + (T_n -t_0)^{-\frac{1}{\rho}} + |u_n|_{L^2([0,T_n]:H)}). \]
Because $T_n -t_0>t_0$ for $n \geq N$, we know that $\sup_n |z_n|_E <+\infty$. By (\ref{eq:control-sup-time-bound}), we get that
\[\sup_n \sup_{s \in [0,t_0]} |X^{0,\tilde{u}_n}_{z_n}(s)|_E < +\infty.\]
Because $f$ is continuous, we have
\[\sup_n \sup_{s \in [0,t_0]} |F(X^{0,\tilde{u}_n}_{z_n}(s))|_E<+\infty.\]
Now, we write
\begin{align}
y_n &= X^{0, \tilde u_n}_{z_n}(t_0) \nonumber\\
&= S(t_0)z_n + \int_0^{t_0} S(t_0-s) F(X^{0,\tilde{u}_n}_{z_n}(s))ds + \int_0^{t_0} S(t_0-s) G(X^{0,\tilde{u}_n}_x(s))\tilde u_n(s)ds.\nonumber
\end{align}
By Lemma \ref{lem:conv-reg},
\begin{align*}
  &\sup_n \left|\int_0^{t_0} S(t_0-s) F(X^{0,\tilde{u}_n}_{z_n}(s))ds \right|_{H^1}\leq  C\sup_n \left(\int_0^{t_0}|F(X^{0,\tilde{u}_n}_{z_n}(s))|_H^2ds\right)^{\frac{1}{2}} \\
  & \leq C\sup_n \sup_{s \in [0,t_0]}|F(X^{0,\tilde{u}_n}_{z_n}(s))|_E<+\infty.
\end{align*}
Also, by Lemma \ref{lem:conv-reg} and the boundedness of $G$,
\begin{align*}
  &\sup_n \left|\int_0^{t_0} S(t_0-s) G(X^{0,\tilde{u}_n}_{z_n}(s))\tilde u_n(s)ds \right|_{H^1} \leq  C\sup_n |\tilde{u}_n(s)|_{L^2([0,t_0]:H)}<+\infty.
\end{align*}
Finally, because $S(t)$ is an analytic semigroup,
\[\sup_n |S(t_0)z_n|_{H^1} \leq Ct_0^{-\frac{1}{2}}\sup_n |z_n|_H \leq C t_0^{-\frac{1}{2}}|z_n|_E < +\infty.\]
Therefore, we  obtain that
\[\sup_n |y_n|_{H^1} <+\infty.\]

  By the Sobolev embedding theorem, $H^1$ embeds compactly into $E$. Therefore, there exists $y \in E$ such that $|y_n -y|_E \to 0$.
  Because the mapping $x \mapsto |x|_{H^1}$ is lower semi-continuous in the $E$ norm, $|y|_{H^1} \leq \liminf_n |y_n|_{H^1}$ and $y \in H^1$. Arguing as in  \eqref{eq:quasipotential-bound2}, it holds that
  \begin{equation*}
    V(x_n,y) \leq V(x_n, y_n) + V(y_n,y).
  \end{equation*}

  By Lemma \ref{lem:quasipotential-continuity-in-1st}, $V(x_n,y) \to V(x,y)$ and $V(y_n,y) \to 0$. Therefore we can conclude that
  \begin{equation} \label{eq:quasi-lsc-proof}
    V(x,y) \leq \liminf_{n \to \infty} V(x_n,y_n) \leq s.
  \end{equation}

  This implies compactness of \eqref{eq:quasi-level-sets}.
\end{proof}

\begin{proof}[Proof of Theorem \ref{thm:quasipotential-lsc}]
  This is a consequence of Theorem \ref{thm:quasipotential-comp}. Let $x_n \to x$ and $y_n \to y$ in $E$. If $\liminf_n V(x_n,y_n) = +\infty$, then \eqref{eq:quasi-lsc} holds trivially. If $\liminf_n V(x_n,y_n)=:s<\infty$, then the arguments leading to \eqref{eq:quasi-lsc-proof} prove our result.
\end{proof}

\subsection{Quasipotential and its regularity properties}\label{SS:RegularityQuasipotential}

\begin{theorem} \label{thm:quasipotential-lower-bound}
  For any $\hat{x} \in E$ and $a>0$, there exists $R>0$ such that whenever $|x|_E>R$,
 \begin{equation*}
  V(\hat{x}, x) >a.
 \end{equation*}
\end{theorem}

\begin{proof}
  Recall from the definition of the rate function \eqref{eq:rate-fct-def} that the quasipotential can be described in terms of controls as
  \begin{equation*}
    V(\hat{x},x) = \inf \left\{\frac{1}{2}\int_0^T|u(s)|_H^2 ds: u \in L^2([0,T]:H), X^{0,u}_{\hat{x}}(T) = x, T>0 \right\}.
  \end{equation*}

  To derive a lower bound on $V(\hat{x},x)$ for $|x|_E>R$, we need to find a lower bound on the $L^2([0,T]:H)$ norm for any controlled process satisfying $|X^{0,u}_{\hat{x}}(T)|_E>R$.

  We divide this into the cases where $T\leq1$ and $T>1$ and use estimates \eqref{eq:control-sup-time-bound} and \eqref{eq:control-sup-x-bound} respectively.

  If $T\leq1$, then it follows from \eqref{eq:control-sup-time-bound} that
  \[R \leq |X^{0,u}_{\hat{x}}(T)|_E \leq C(1 + |\hat{x}|_E + |u|_{L^2([0,T]:H)}).\]

  Therefore, it follows that
  \begin{equation} \label{eq:u-lower-1}
    \frac{R}{C} - |\hat{x}|_E \leq |u|_{L^2([0,T]:H)}.
  \end{equation}

  If $T>1$, then we use the fact that $X^{0,u}_{\hat{x}}(T) = X^{0,\tilde u}_{X^{0,u}_{\hat{x}}(T-1)}(1)$ when
  $\tilde u(t) = u(t+T-1)$ for $t \in [0,1]$. It follows from \eqref{eq:control-sup-x-bound},
  \[R \leq |X^{0,u}_{\hat{x}}(T)|_E \leq \sup_{x \in E} |X^{0,\tilde u}_{x}(1)|_E \leq C( 2 + |\tilde u|_{L^2([0,1]:H)}).\]

  It follows that
  \begin{equation} \label{eq:u-lower-2}
    \frac{R}{C} - 2 \leq |\tilde u|_{L^2([0,1]:H)} \leq |u|_{L^2([0,T]:H)}.
  \end{equation}

  Then depending on whether $T\leq1$ or $T>1$, \eqref{eq:u-lower-1} or \eqref{eq:u-lower-2} holds. Consequently we can choose $R$ big enough so that if $u \in L^2([0,T]:H)$ is such that $|X^{0,u}_{\hat{x}}(T)|_E>R$, then $\frac{1}{2}\int_0^T|u(s)|_H^2 ds > a$.
  This proves the result.
\end{proof}

\begin{theorem} \label{thm:V-D_R=V-D}
  %Let $x_*$ be any point in $E$ and $D \subset E$ be an open, unbounded domain that contains $x_*$.
Let $D$ be defined by \eqref{Eq:DomainAttraction}. For any $R>0$, let $D_R = D \cap B_R$ where $B_R = \{x \in E: |x|_E<R\}$. If $V(x_*, \partial D)< +\infty$, then there exists $R>0$ such that $V(x_*, \partial D) = V(x_*, \partial D_R)$.
\end{theorem}

\begin{proof}
The boundary $\partial D_R$ can be decomposed into
\[\partial D_R = (\partial D \cap B_R) \cup (\partial B_R \cap D).\]
Let $a := V(x_*, \partial D)$. By the lower bound of the quasipotential, Theorem \ref{thm:quasipotential-lower-bound}, there exists $R>0$ such that whenever $|x|\geq R$, $V(x_*,x)\geq a + 1.$
In particular, this proves that any approximate minimizing points $x \in \partial D$ with the property that $V(x_*,x)< V(x_*, \partial D) + \e$ for $\e \in (0,1)$, have the property that $|x|_E<R$ and therefore $x \in \partial D\cap B_R$. Furthermore, any $x \in \partial B_R \cap D$ has the property that $|x|_E=R$ and therefore $V(x_*,x)> a + 1$. This proves that any approximate minimizers of $V(x_*, \partial D)$ belong to $\partial D_R$ and any approximate minimizers of $V(x_*,\partial D_R)$ belong to $\partial D$. Consequently, $V(x_*, \partial D) = V(x_*, \partial D_R)$.
\end{proof}

Now we prove certain internal regularity properties of the quasipotential.  Define for $\rho>0,T>0$,
\begin{equation}
  D^{\rho,T} = \{x \in D: \sup_{t \geq T} |X^0_x(t) - x_*|_E< \rho\}.\label{Eq:DefDrhoT}
\end{equation}

 Because $D$ is a domain of attraction, for sufficiently small $\rho>0$, we shall have that
\begin{equation} \label{eq:union-of-D-rho-T}
  D = \bigcup_{T>0} D^{\rho,T}.
\end{equation}

\begin{theorem}[Internal regularity near the boundary]\label{L:InnerBoundaryRegularityQP}
  For any $\rho>0$
  \begin{equation*}
    \lim_{T \to \infty} V(x_*, \partial D^{\rho,T}) = V(x_*, \partial D).
  \end{equation*}
\end{theorem}
\begin{proof}
  Fix $\rho>0$.
  First we prove  that for any $T>0$,
  \begin{equation*}
    V(x_*, \partial D^{\rho,T}) \leq V(x_*, \partial D).
  \end{equation*}

   Assume $V(x_*, \partial D)<\infty$. By the definition of $V$ there exist $T_0>0$ and $\varphi \in C([0,T_0]:E)$ such that $\varphi(0) =x_*$ and $\varphi(T_0) \in \partial D$. By continuity, there must exist $t \in [0,T_0]$ such that $\varphi(t) \in \partial D^{\rho,T}$. Then it is clear from the definition of $V$ that
  \begin{equation*}
    V(x_*, \partial D^{\rho,T}) \leq I^t_{x_*}(\varphi) \leq I^{T_0}_{x_*}(\varphi).
  \end{equation*}

  We can take the infimum over all $T_0>0$ and $\varphi \in C([0,T_0]:E)$ satisfying $\varphi(0)=x_*$ and $\varphi(T_0) \in \partial D$ to conclude
  \begin{equation*}
    V(x_*, \partial D^{\rho,T}) \leq V(x_*, \partial D).
  \end{equation*}

  Now we show that
  \begin{equation*}
    \liminf_{T \to \infty}V(x_*, \partial D^{\rho,T}) \geq V(x_*, \partial D).
  \end{equation*}

  Let $T_n \to \infty$ be a sequence such that
  \[\lim_{n \to \infty} V(x_*, \partial D^{\rho,T_n}) = \liminf_{T \to \infty} V(x_*, \partial D^{\rho,T}).\]

  There exist $y_n \in \partial D^{\rho,T_n}$ such that $V(x_*, y_n) \leq V(x_*, \partial D^{\rho,T_n}) + \frac{1}{T_n}.$
  By Theorem \ref{thm:quasipotential-comp}, there exists a subsequence such that $y_n$ converges to a limit $y$. Clearly $y \not \in D^{\rho,T}$ for any $T>0$. Therefore $y \not \in D$ by \eqref{eq:union-of-D-rho-T}. On the other hand, $y \in \bar{D}$ so it follows that $y \in \partial D$. Therefore, by Theorem \ref{thm:quasipotential-lsc}
  \begin{align*}
    &V(x_*, \partial D) \leq V(x_*, y) \leq \liminf_{n \to \infty} V(x_*, y_n) \leq \liminf_{n \to \infty} V(x_*, \partial D^{\rho,T_n}) \\
    &\leq \liminf_{T \to \infty} V(x_*, \partial D^{\rho,T}).
  \end{align*}

\end{proof}

Next for $R<\infty$, let us recall that  $D_{R}=\left\{x\in D: |x|_{E}\leq R\right\}$ and define
\begin{equation}
 D^{\rho,T}_{R}= D^{\rho,T}\cap B_{R} = \{x \in D_{R}: \sup_{t \geq T} |X^0_x(t) - x_*|_E< \rho\}.\label{Eq:DefDrhoT_R}
\end{equation}

A  consequence of Theorems \ref{thm:V-D_R=V-D} and \ref{L:InnerBoundaryRegularityQP} is the following Corollary.
\begin{corollary}\label{L:InnerBoundaryRegularityQP_R}
 For any $\rho>0$, there is $R$ large enough so that
  \begin{equation*}
    \lim_{T \to \infty} V(x_*, \partial D^{\rho,T}_{R}) = V(x_*, \partial  D).
  \end{equation*}
\end{corollary}
\begin{proof}
An inspection of the proof of Theorem \ref{thm:V-D_R=V-D} shows that for any $\rho>0$, there is $R$ large enough so that uniformly in $T$ we have
  \begin{equation*}
 V(x_*, \partial D^{\rho,T}_{R}) = V(x_*, \partial  D^{\rho,T}).
  \end{equation*}

  Then, Theorem \ref{L:InnerBoundaryRegularityQP} concludes the proof of the corollary.
\end{proof}

\subsection{Continuity properties of the quasipotential  $\tilde{V}_{D}$ and $\hat{V}_{D}$}\label{SS:ContinuityQuasipotentialHatTilde}
In this section, we prove lower semi-continuity as well as compactness of level sets of the quasipotentials $\tilde{V}_{D}(\cdot,\cdot)$ and $\hat{V}_{D}(\cdot,\cdot)$ defined in (\ref{eq:V-tilde-D}) and (\ref{eq:V-hat-D}) respectively.

Let us first discuss the lower-semi-continuity properties of $\tilde{V}_D$.
\begin{theorem}[Lower semi-continuity of $\tilde{V}_D$] \label{thm:V-tilde-D-lsc}
  If $x_n \in \bar{D}$ converge in $E$ norm to $x \in \bar{D}$ and $y_n \in \bar{D}$ converge to $y \in \bar{D}$ in $E$ norm, then
  \begin{equation}
    \liminf_{n \to \infty} \tilde{V}_D(x_n,y_n) \geq \tilde{V}_D(x,y).
  \end{equation}
\end{theorem}

\begin{proof}[Proof of Theorem \ref{thm:V-tilde-D-lsc}]
  The result is trivially true if $\liminf_n \tilde{V}_D(x_n,y_n) =+\infty$ or $\tilde{V}_D(x,y)=0$. Assume $\liminf_n \tilde{V}_D(x_n,y_n) < \infty$ and $\tilde{V}(x,y) >0$. Let $\gamma>0, \rho>0$ be arbitrary.

  By \eqref{eq:V-inequality}, $\tilde{V}_D(x_n,y_n) \geq \tilde{V}_D^{\rho/2}(x_n,y_n)$. Because we assumed that this is finite for large $n$, there exist $T_{n,1}>0$ and $\varphi_{n,1} \in C([0,T_{n,1}]:E)$ such that $\varphi_{n,1}(0) = x_n$, $\varphi_{n,1}(T_{n,1}) = y_n$, $\varphi_{n,1}(t) \in D \cup B(x_n,\rho/2)\cup B(y_n,\rho/2)$ for $t \in [0,T_{n,1}]$ and $I_{x_n}^{T_{n,1}}(\varphi_{n,1}) < \tilde{V}_D^{\rho/2}(x_n,y_n) + \gamma/3$.

  By Lemma \ref{lem:control-prob-cont}, if $\liminf_n T_{n,1}=0$, then $y=x$ and $\tilde{V}_D(x,y)=0$. Because the lemma is trivially true when $\tilde{V}_D(x,y)=0$, we can assume that
  $$t_0:=\min\left\{\liminf_n T_{n,1},\rho/2\right\}>0.$$

  For large $n$, we can guarantee that $|x_n-x|_E<t_0$. By possibly increasing $n$, Theorem \ref{thm:control} guarantees that there exists $\varphi_{n,2}\in C([0,T_{n,1}]:E)$ such that $\varphi_{n,2}(0)=x$, $\varphi_{n,2}(t) = \varphi_{n,1}(t)$ for $t \in [t_0,T_{n,1}]$, $|\varphi_{n,1}(t) - \varphi_{n,2}(t)|_E<\rho/2$ for $t \in [0,t_0]$ and $I_{x}^{T_{n,1}}(\varphi_{n,2}) < \tilde{V}^{\rho/2}_D(x_n,y_n) + 2\gamma/3.$

  By Theorem \ref{thm:reversed}, we can find $T_3>0$ small enough so that the revered process starting at $y$ has action $I_{X^0_y(T_3)}^{T_3}(X^0_y(T_3-\cdot)) < \gamma/6$. By possibly decreasing $T_3$, we can also guarantee that $|X^0_y(t)-y|_E< \rho/2$ for all $t \in [0,T]$. Then the process
  \[\varphi_3(t) = \begin{cases}
    X^0_y(t) & \text{ for } t \in [0,T_3/2]\\
    X^0_y(T_3-t) & \text{ for } t \in [T_3/2,T_3]
  \end{cases}\]
  has the properties that $|\varphi_3(t)-y|_E< \rho/2$ for all $t \in [0,T_3]$ and $I_y^{T_3}(\varphi_3)< \gamma/6$.

  Then Theorem \ref{thm:control} %\ref{thm:quasipoential-cont}
  guarantees that for $n$ large enough there exists $\varphi_{n,4} \in C([0,T_3]:E)$ such that $\varphi_{n,4}(0) = y_n$, $\varphi_{4,n}(T_3)=y$, $|\varphi_{n,4}(t) - y|_E< \rho$ for all $t \in [0,T_3]$ and $I_{y_n}^{T_3}(\varphi_{n,4}) < \gamma/3$.

  We concatenate setting
\[\varphi(t) = \begin{cases}
    \varphi_{n,2}(t) & \text{ for } t \in [0,T_{n,1}]\\
    \varphi_{n,4}(t-T_{n,1}) & \text{ for } t \in [T_{n,1},T_{n,1}+T_3],
  \end{cases}\]
  and set $T_n:=T_{n,1}+T_3$. Then for large $n$, $\varphi_n$ has the properties that $\varphi_n(0) = x$, $\varphi_n(T_n) = y$, $\varphi(t) \in D\cup B(x,\rho) \cup B(y,\rho) $ for $t \in [0,T_n]$, and
  \[I^{T_n}_x(\varphi_n) < \tilde{V}^{\rho/2}_D(x_n,y_n) + \gamma.\]

  Therefore, for large $n$
  \[\tilde{V}^\rho_D(x,y) \leq \tilde{V}^{\rho/2}_D(x_n,y_n) + \gamma.\]

  Because $\gamma$ was arbitrary and $\tilde{V}^{\rho/2}_D(x_n,y_n) \leq \tilde{V}_D(x_n,y_n)$,
  \[\tilde{V}^\rho_D(x,y) \leq  \liminf_{n \to \infty} \tilde{V}_D(x_n,y_n).\]

  Finally, we can take the limit as $\rho \to {0}$ to conclude that
  \[\tilde{V}_D(x,y) \leq \liminf_{n \to \infty} \tilde{V}_D(x_n,y_n).\]
\end{proof}

\begin{theorem}[Compactness of level sets for $\tilde{V}_D$] \label{thm:V-tilde-compact-level-sets}
  For any compact $K \subset \bar{D}$ and $s \geq 0$, the set
  \begin{equation} \label{eq:V_D-level-set}
    \bigcup_{x \in K} \{y \in \bar{D}: \tilde{V}_D(x,y) \leq s\}
  \end{equation}
  is compact.
\end{theorem}
\begin{proof}
  By \eqref{eq:V-inequality}, $V(x,y) \leq \tilde{V}_D(x,y)$. This means that the set in \eqref{eq:V_D-level-set} is a subset of the set in (\ref{eq:quasi-level-sets}). By Theorem \ref{thm:quasipotential-comp}, therefore, we know that the set in \eqref{eq:V_D-level-set} is pre-compact. Then the compactness is an immediate consequence of Theorem \ref{thm:V-tilde-D-lsc}. Specifically, if $x_n \in K$ and $y_n \in \bar{D}$ such that $\tilde{V}_D(x_n,y_n) \leq s$. Then there exists a subsequence such that $x_n \to x \in K$  and $y_n \to y$ in $E$ norm. By Theorem \ref{thm:V-tilde-D-lsc}, $\tilde{V}_D(x,y) \leq \liminf_n \tilde{V}_D(x_n,y_n) \leq s.$
\end{proof}

Theorems analogous to (\ref{thm:V-hat-D-lsc}) and (\ref{thm:V-tilde-compact-level-sets}) are also true for $\hat{V}_D$. Since the proof of these results are identical to those for $\tilde{V}_D$ we only include the statements and not the proofs.
\begin{theorem}[Lower semi-continuity of $\hat{V}_D$] \label{thm:V-hat-D-lsc}
  If $x_n \in \bar{D}$ converge in $E$ norm to $x \in \bar{D}$ and $y_n \in \bar{D}$ converge to $y \in \bar{D}$ in $E$ norm, then
  \begin{equation}
    \liminf_{n \to \infty} \hat{V}_D(x_n,y_n) \geq \hat{V}_D(x,y).
  \end{equation}
\end{theorem}

\begin{theorem}[Compactness of level sets for $\hat{V}_D$] \label{thm:V-hat-compact-level-sets}
  For any compact $K \subset \bar{D}$ and $s \geq 0$, the set
  \begin{equation} \label{eq:hat-V_D-level-set}
    \bigcup_{x \in K} \{y \in \bar{D}: \hat{V}_D(x,y) \leq s\}
  \end{equation}
  is compact.
\end{theorem}
%\begin{proof}
%  The proof is the same as that of Theorem \ref{thm:V-tilde-compact-level-sets} and thus omitted.
%\end{proof}

\subsection{Properties of the $V$ equivalence classes $K_i$}\label{SS:EquivalenceClasses}

In Assumption \ref{A:K_equivalenceClassV}, we assumed there there exist a finite collection $K_1, ..., K_N$ of $V$ equivalence classes that contain all the $\omega$-limit points of the deterministic dynamical system $X^0_x(t)$ for $x \in \partial D$.

A Corollary of Theorem \ref{thm:quasipotential-comp} is that each $K_i$ $i \in \{1,...,N\}$ is a compact subset of $E$.
\begin{corollary} \label{cor:equivalence-class}
  If $K$ is a $V$ equivalence class then $K$ is compact.
\end{corollary}
\begin{proof}
  Let $x \in K$. Then it is clear that $K$ is a subset of the sublevel set $\{y \in E: V(x,y)=0\}$, which is compact by Theorem \ref{thm:quasipotential-comp}. Therefore $K$ is precompact. We can prove that $K$ is compact by proving that it is also closed. Suppose $y_n \in K$ with $y_n \to y \in E$. Theorems \ref{thm:quasipoential-cont} and \ref{thm:quasipotential-lsc} guarantee that
  \begin{align*}
    V(y,x) = \lim_{n \to \infty} V(y_n,x) = 0 \ \text{ and } \
    0 \leq V(x,y) \leq \liminf_{n \to \infty} V(x,y_n) =0.
  \end{align*}

    Therefore $x \sim y$ and $y \in K$.
\end{proof}

Because of the way we defined $\hat{V}$ in \eqref{eq:hat-V-D_rho}-\eqref{eq:V-hat-D}, we will need the following result
\begin{lemma} \label{lem:paths-stay-near-K_i}
  For any any $\rho>0$, there exists $\gamma>0$ such that whenever $T>0$ and $\varphi \in C([0,T]:E)$ has the properties that for some $i \in \{1,...,N\}$, $\varphi(0) \in K_i$, $\varphi(T) \in K_i$, and $I_{\varphi(0)}^T(\varphi) < \gamma$, it follows that
  \begin{equation} \label{eq:stay-near-K_i}
    \dist_E(\varphi(t), K_i) < \rho \text{ for all } t \in [0,T].
  \end{equation}
\end{lemma}
\begin{proof}
  Fix $i \in \{1,...,N\}$. Let $T_n>0$, $\varphi_n \in C([0,T_n]:E)$ be arbitrary sequences such that $x_n:=\varphi_n(0) \in K_i$, $y_n:=\varphi_n(T_n) \in K_i$,  $I_{x_n}^{T_n}(\varphi_n)  \to 0$.  Let $t_n \in [0,T_n]$ be an arbitrary sequence and let $z_n = \varphi_n(t_n)$.

  By Theorem \ref{thm:quasipotential-comp} and the fact that $K_i$ is compact, there exists a subsequence such that $x_n \to x \in K_i$ and $z_n \to z \in E$ and $y_n \to y \in K_i$. Furthermore, by Theorem \ref{thm:quasipotential-lsc},
  \[V(x,z) + V(z,y) =0,\]
  implying that $z \sim x \sim y$. Because $K_i$ is a $V$ equivalence class and by Lemma \ref{lem:K_i-in-boundary}, $z \in K_i$. Because the sequences $\varphi_n$ and $t_n$ were arbitrary, we have shown that
  \[\lim_{\gamma \to 0} \sup_{T>0}\sup_{\substack{\{\varphi \in C([0,T]:E):\\ I_{\varphi(0)}^T(\varphi)<\gamma,\\ \varphi(0) \in K_i, \varphi(T) \in K_i\}}}\ \sup_{t \in [0,T]} \dist_E(\varphi(t), K_i)=0. \]
\end{proof}

\section{Proof of Theorem \ref{thm:meanExitAsymptotics}: Logarithmic asymptotic for the exit time}\label{S:ExitTimeAsymptotics}
The proof of Theorem \ref{thm:meanExitAsymptotics} is given in Subsections \ref{SS:UpperBoundExitTime} and \ref{SS:LowerBoundExitTime}, where the upper and lower bounds respectively are being established.

\subsection{Upper bound of the exit time}\label{SS:UpperBoundExitTime}
In this section we prove an upper bound for the exponential growth rate of the exit time as $\e \to 0$, see Theorem \ref{thm:upper-bound}.
\begin{theorem} \label{thm:upper-bound}
  For any $x \in D$,
  \begin{equation} \label{eq:upper-bound-exp}
    \limsup_{\e \to 0} \e \log \E \tau^\e_x \leq V(x_*, \partial D)
  \end{equation}
  and for any $\gamma>0$,
  \begin{equation} \label{eq:upper-bound-prob}
    \lim_{\e \to 0} \Pro( \e \log \tau^\e_x > V(x_*, \partial D) + \gamma) =0.
  \end{equation}
\end{theorem}
We will prove Theorem \ref{thm:upper-bound} at the end of this section. First we introduce several lemmas.

Notice that the domain $D$ is (possibly) unbounded, which leads to some complications in the analysis of exit times because we can only prove uniform large deviations principles over bounded domains (see Theorem \ref{thm:LDP}).

\begin{lemma} \label{lem:close-to-unif-attract}
  For $0<\rho,T<\infty$ let $D^{\rho,T}$ be given by (\ref{Eq:DefDrhoT}). For any $h>0$, and $\rho >0$,  there exists $T>0$ such that for all $x \in D$
  \begin{equation*}
    \dist_E(X^0_x(1), D^{\rho,T}) < h.
  \end{equation*}
  \end{lemma}

\begin{proof}
Corollary \ref{cor:pre-compact-all-init-cond} guarantees that the set
  \begin{equation*}
    K^1 = \{X^0_x(1): x \in D\}
  \end{equation*}
  is pre-compact in $E$. For any $h>0$ there exists a finite collection $\{z_1,...,z_N\} \subset D$ such that
  \[\overline{K^1} \subset \bigcup_{k=1}^N B(z_k, h).\]
  Because $D = \bigcup_{T>0} D^{\rho,T}$ and there are only a finite number of $z_k$, there exists $T>0$ such that $z_k \in D^{\rho,T}$ for all $k \in \{1,...,N\}$. This proves that for all $x \in D$,
  \[\dist_E(X^{0}_{x}(1),D^{\rho,T})<h.\]
\end{proof}

\begin{lemma} \label{lem:control-to-rho}
  For any $\gamma>0$ and $\rho>0$, there exists $T>0$ such that for every $x \in D$ there exists $\varphi^x \in C([0,T]:E)$ satisfying $\varphi^x(0) = x$, $\varphi(T) \in B(x_*,\rho)$, and $I_x^T(\varphi)<\gamma$.
\end{lemma}

\begin{proof}
  By \eqref{eq:control-sup-x-bound}, there exists $R>0$ such that $\sup_{x \in E} |X^0_x(1)|_E \leq R$. By Theorem \ref{thm:control}, there exists $h\in (0,1)$ such that whenever $x , y \in E$ with $|x|_E, |y|_E\leq R$ and  $|x-y|_E< h$, there exists a control $u_1 \in L^2([0,h]:H)$ such that $X^0_y(t) = X^{0,u_1}_x(t)$ for all $t\geq h$ and $\frac{1}{2}\int_0^h |u_1(s)|_H^2 ds < \gamma$.
  By Lemma \ref{lem:close-to-unif-attract} there exists a time $T_1$ such that
  for all $x \in D$
  \[
    \dist_E(X^0_x(1), D^{\rho,T_1}) < h.
  \]
  This means that for any $x \in D$, we can find $y \in D^{\rho,T_1}$ such that $|X^0_x(1)-y|_E <h$. By Theorem \ref{thm:control} and our choice of $h$ there is also a control $u_1 \in L^2([0,h]:H)$ such that $X^{0,u_1}_{X^0_x(1)}(h) = X^0_y(h)$ and $\frac{1}{2}\int_0^{h} |u_1(t)|_H^2 dt < \gamma$. Let $T:=1 +  T_1$ and define the control $u \in L^2([0,T]:H)$ by
  \[u(t) = \begin{cases}
     0 & \text{ if } t \in [0,1)\\
     u_1(t-1) & \text{ if } t \in [1, 1+h)\\
     0 & \text{ if } t \in [1+h,1+T_1].
  \end{cases}\]

  Note that by time homogeneity $X^{0,u}_x(1+h) = X^{0,u_1}_{X^0_x(1)}(h)$ and we observed above that $X^{0,u_1}_{X^0_x(1)}(h) = X^0_y(h)$. Because $u(t)= 0$ for $t>1+h$ it follows that $X^{0,u}_x(t) = X^0_y(t - 1)$ for $t\in [1+h,1+T_1]$. In particular, because $X^0_y(T_1) \in B(x_*, \rho)$, it follows that $X^{0,u}_x(T) \in B(x_*,\rho)$.

  The result follows by setting $\varphi:= X^{0,u}_x$.
\end{proof}

\begin{lemma} \label{lem:control-to-x+}
  For any $\gamma>0$, there exists $\rho >0$ such that for any $x \in B(x_*,\rho)$ there exists a path $\varphi \in C([0,\rho]:E)$ such that $\varphi(0) = x$, $\varphi(\rho) = x_*$ and  $I^\rho_x(\varphi)<\gamma$.
\end{lemma}
\begin{proof}
    This is an immediate consequence of Theorem \ref{thm:control} because $X^0_{x_*}(t) = x_*$ for all $t>0$.
\end{proof}

\begin{lemma} \label{lem:control-x+-to-outside}
  For any $\gamma>0$ there exists a $T>0$ and a path $\varphi \in C([0,T]:E)$ such that $\varphi(0) = x_*$, $\varphi(T) \in (\overline{D})^c$ and $I_{x_{*}}^T(\varphi) \leq V(x_*,\partial D) + \gamma$.
\end{lemma}

\begin{proof}
 By the definition of $V(x_*, \partial D)$, there exists $z \in \partial D$, $T_1>0$, and $\varphi_1 \in C([0,T_1])$ such that $\varphi_1(0) = x_*$, $\varphi_1(T_1) = z$ and $I^{T_1}_{x_{*}}(\varphi_1) \leq V(x_*, \partial D) + \frac{\gamma}{3}$.

  By Assumptions \ref{assum:D-boundary} and \ref{A:AttractionProperty}, we can find $y\notin\bar{D}$, close enough to $z\in\partial D$, say $|y-z|_{E}<\delta/2$, such that by Theorem \ref{thm:control} there exists a path $\varphi_{2}\in C([0,\delta],E)$ with $\varphi_{2}(0)=z$ and $\varphi_{2}(\delta)\notin\bar{D}$ and $I^{\delta}_{z}(\varphi_{2})\leq \frac{\gamma}{3}$. 

Let $T:=T_1 +  \delta$. The path
  \[\varphi(t) :=\begin{cases}
    \varphi_1(t) & \text{ if } t \in [0,T_1)\\
    \varphi_2(t-T_1) & \text{ if } t \in [T_1, T_1 + \delta]
  \end{cases}\]
  has the properties that $\varphi(0) = x_*$, $\varphi(T) \in (\overline{D})^c$, and $I_{x_{*}}^T(\varphi) \leq V(x_*, \partial D) + \gamma$.
\end{proof}

\begin{lemma} \label{lem:exit-prob-low-bound}
  For any $\gamma>0$ and $R>0$, there exists $T>0$ such that
  \begin{equation*}
    \liminf_{\e \to 0} \inf_{x \in D_{R}} \e \log \Pro(\tau^\e_{x} \leq T) \geq -V(x_*, \partial D) - \gamma,
  \end{equation*}
{where we recall that $D_R= \{x \in D: |x|_E \leq R\}$.  }
\end{lemma}

\begin{proof}

Let $\gamma>0$. By Lemma \ref{lem:control-to-x+} there exists $\rho >0$ such that whenever $|z-x_*|_E < \rho$, there exists $\varphi_2^z \in C([0,\rho]:E)$ such that $\varphi^z_2(0) = z$, $\varphi^z_2(\rho) = x_*$ and $I^\rho_z(\varphi^z_2) < \frac{\gamma}{3}$.

By Lemma \ref{lem:control-to-rho}, there exists $T_1>0$ such that for every $x \in D$, there exists $\varphi_1^x \in C([0,T_1]:E)$ such that $\varphi^x_1(0) = x$, $|\varphi^x_1(T_1) -x_*|_E< \rho$, and $I_x^{T_1}(\varphi_1^x) < \frac{\gamma}{3}$.

  By Lemma \ref{lem:control-x+-to-outside}, there exists $T_2>0$ and $\varphi_3 \in C([0,T_2]:E)$ such that $\varphi_3(0) = x_*$, $\varphi_3(T_2) \not \in \overline{D}$, and $I_{x_*}^{T_2}(\varphi_3) \leq V(x_*, \partial D) + \frac{\gamma}{3}$.

  Let $T := T_1+\rho+ T_2$ and define the path
  \begin{equation*}
    \varphi^x(t) = \begin{cases}
      \varphi_1^x(t) & \text{ if } t \in [0,T_1)\\
      \varphi_2^{\varphi_1^x(T_1)}(t - T_1) & \text{ if } t \in [T_1, T_1 + \rho) \\
      \varphi_3(t-T_1 - \rho) & \text{ if } t \in [T_1 + \rho, T]
    \end{cases}
  \end{equation*}

    This is a continuous path satisfying $\varphi^x(0) = x$, $\varphi^x(T)=\varphi_3(T_2)=:y \not \in \overline{D}$, and $I_x^T(\varphi^x) \leq V(x_*, \partial D) + \gamma$.

  Because $y \not \in \overline{D}$ there exists $\delta>0$ such that $B(y, \delta) \subset (\overline{D})^c$. Consequently, for any $x \in D$,
  \[\Pro(\tau^\e_{x} \leq T) \geq \Pro(|X^\e_x - \varphi^x|_{C([0,T]:E)}< \delta).\]

Recall that the LDP lower bound \eqref{eq:ldp-low} is only proven to be uniform over bounded subsets of initial conditions. For this reason we prove uniformity over the set $D_R$ for any $R>0$. By the LDP lower bound \eqref{eq:ldp-low},
  \begin{align*}
    &\liminf_{\e \to 0} \inf_{x \in D_{R}} \e \log\Pro(\tau^\e_{x} \leq T) \geq \liminf_{\e \to 0} \inf_{x \in D_{R}} \e \log\Pro(|X^\e_x - \varphi^x|_{C([0,T]:E)}<\delta)\\
    &\geq -V(x_*, \partial D) - \gamma.
  \end{align*}
\end{proof}

\begin{proof}[Proof of Theorem \ref{thm:upper-bound}]
  For $R>0$, $T>0$ to be specified later and any $x \in D$ and $\e>0$, we decompose the probability $\Pro(\tau^\e_x\geq T+1)$ as
  \begin{equation*}
    \Pro(\tau^\e_x \geq T+1) \leq \Pro(X^\e_x(1) \in D_R, \tau^\e_x \geq T+1) + \Pro(|X^\e_x(1)|_E>R, \tau^\e_x\geq T+1).
  \end{equation*}

  By the Markov property,
  \begin{equation*}
    \Pro(\tau^\e_x \geq T+1) \leq \Pro(X^\e_x(1) \in D_R) \sup_{y \in D_R} \Pro(\tau^\e_y \geq T) + \Pro(|X^\e_x(1)|_E>R).
  \end{equation*}

  This is less than
  \begin{align*}
    &\Pro(\tau^\e_x \geq T+1) \\
    &\leq (1- \Pro(|X^\e_x(1)|_E >R)) \sup_{y \in D_R} \Pro(\tau^\e_y \geq T) + \Pro(|X^\e_x(1)|_E>R).
  \end{align*}

  By the Markov property, for any $T>0$, $R>0$, $\e>0$, $x \in D$,  and $k \in \mathbb{N}$,
  \begin{align} \label{eq:Markov-exit}
    &\Pro(\tau^\e_{x} >k(T+1)) \nonumber\\
    &\leq \left(\sup_{x\in D} \left((1- \Pro(|X^\e_x(1)|_E >R)) \sup_{y \in D_R} \Pro(\tau^\e_y \geq T) + \Pro(|X^\e_x(1)|_E>R) \right)\right)^k
  \end{align}

  For any $x \in D$ and $\e>0$,
  \begin{equation*}
    \E\left(\tau^\e_{x} \right) \leq (T+1)\left(1+ \sum_{k=1}^\infty \Pro(\tau^\e_{x} >k(T+1))\right).
  \end{equation*}

  By \eqref{eq:Markov-exit} we see that
  \begin{align} \label{eq:E-tau-upper-bound}
  &\E\left(\tau^\e_{x} \right) \leq (T+1)\times\nonumber\\
   &\times\left(1+\sum_{k=1}^\infty \left(\sup_{x \in D} \left((1- \Pro(|X^\e_x(1)|_E >R)) \sup_{y \in D_R} \Pro(\tau^\e_y \geq T) + \Pro(|X^\e_x(1)|_E>R) \right)\right)^k\right)\nonumber\\
   &\leq \frac{T+1}{1 -\left(\sup_{x \in D} \left((1- \Pro(|X^\e_x(1)|_E >R)) \sup_{y \in D_R} \Pro(\tau^\e_y \geq T) + \Pro(|X^\e_x(1)|_E>R)\right)\right)}.
   \end{align}

The denominator of \eqref{eq:E-tau-upper-bound} can be written as
  \begin{align*}
    &1 -\sup_{x \in D}\left((1- \Pro(|X^\e_x(1)|_E >R)) \sup_{y \in D_R} \Pro(\tau^\e_y \geq T) + \Pro(|X^\e_x(1)|_E>R)\right) \\
    &=\inf_{x \in D}\left(1 - \Pro(|X^\e_x(1)|_E>R)\right) \left(1 - \sup_{y \in D_R} \Pro(\tau^\e_y \geq T) \right)\\
    &= \inf_{x \in D} (1 - \Pro(|X^\e_x(1)|_E>R))\inf_{y \in D_R} \Pro(\tau^\e_y \leq T)
   \end{align*}

  By \eqref{eq:X-big-zero-prob}, we can choose $R>0$ big enough and $\e_0>0$ small enough so that for $\e\in (0,\e_0)$
  \begin{equation} \label{eq:bigger-than-R-prob}
    \sup_{x \in E}\Pro(|X^\e_x(1)|_E>R)<\frac{1}{2}.
  \end{equation}

With this choice, the denominator of \eqref{eq:E-tau-upper-bound} is bounded by
  \begin{align*}
    &1 -\sup_{x \in D}\left((1- \Pro(|X^\e_x(1)|_E >R)) \sup_{y \in D_R} \Pro(\tau^\e_y \geq T) + \Pro(|X^\e_x(1)|_E>R)\right) \\
    &\geq \frac{1}{2} \inf_{y \in D_R} \Pro(\tau^\e_y \leq T)
  \end{align*}

  Let $\gamma>0$. By Lemma \ref{lem:exit-prob-low-bound}, there exists $T>0$ such that
   \begin{equation} \label{eq:exit-prob-low-bound}
     \liminf_{\e \to 0} \inf_{y \in D_{R}} \e \log \Pro(\tau^\e_{y} \leq T) \geq -V(x_*, \partial D) - \gamma.
   \end{equation}

   Then applying \eqref{eq:E-tau-upper-bound}, it follows that
  \begin{align}
  \limsup_{\e \to 0} \e \log \E(\tau^\e_{x}) &\leq \limsup_{\e \to 0}  \e \log (2(T+1)) - \liminf_{\e \to 0} \inf_{y \in D_{R}} \e \log \Pro(\tau^\e_{y} \leq T)\nonumber\\
  & \leq V(x_*,\partial D_{R}) + \gamma.\nonumber
  \end{align}

  We conclude the proof of the lemma by noticing that \eqref{eq:upper-bound-exp} follows because $\gamma$ was arbitrary and that \eqref{eq:upper-bound-prob} is a consequence of \eqref{eq:upper-bound-exp} and Chebyshev inequality. In particular, we have that
  \begin{align*}
    \Pro(\e \log \tau^\e_{x} > V(x_*, \partial D) + \gamma) &= \Pro\left(\tau^\e_x> \exp \left(\e^{-1} \left(V(x_*, \partial D_R) + \gamma \right) \right)\right)\\
    &\leq \E(\tau^\e_{x})\exp \left(-\e^{-1} \left(V(x_*, \partial D) + \gamma \right) \right),
  \end{align*}
  which converges to zero by \eqref{eq:upper-bound-exp}.
\end{proof}

\subsection{Lower bound of the exit time} \label{SS:LowerBoundExitTime}

Let us recall the definition of $D^{\rho,T}_{R}$ from (\ref{Eq:DefDrhoT_R}). The following lemma says that for each $R,T<\infty$, $D^{\rho,T}_R$ is uniformly attractive for $\rho$ small enough (recall the definition of uniformly attractive \eqref{eq:unif-attr-def}).
\begin{lemma}\label{L:UniformAttractiveD}
For each $R,T<\infty$, $D^{\rho,T}_R$ is uniformly attractive for $\rho$ small enough.
\end{lemma}
\begin{proof} %[Proof of Lemma \ref{L:UniformAttractiveD}]
Fix some $R,T<\infty$.  Because $D$ is open we can choose $\rho$ small enough so that $\{x \in E: |x-x_*|_E\leq \rho\} \subset D$. Let $0< \rho'<\rho$.
Assume by contradiction that there exists $x_n \in D^{\rho,T}_R$ and $T_n \uparrow + \infty$ such that
\[\inf_{t \in [0,T_n]}|X^0_{x_n}(t) - x_*|_E > \rho'.\]

By the Corollary 5.4, and the fact that
\[\sup_{x \in D^{\rho,T}_R} |X^{0}_x(T +t) - x_*|_E <\rho, \text{ for }   t \geq 0\]
the family $\{X^0_{x_n}(T + 1) : n \in \mathbb{N}\}$ is compact. There exists a subsequence such that $X^0_{x_n}(T+1) \to y \in \{x \in E: |x-x_*|_E \leq \rho\} \subset D$. Because $y \in D$, there exists $\tilde{T}>0$ such that
\[|X^0_{y}(\tilde{T}) - x_*|_E < \frac{\rho'}{2}.\]

For large enough $n$,  $|X^0_{x_n}(T+1+\tilde{T}) -X^0_y(\tilde{T})|_E< \frac{\rho'}{2}.$ But this contradicts our assumption that the $T_n \to \infty$. Therefore, $D^{\rho,T}_R$ must be uniformly attractive.
\end{proof}

 For a given $x\in D^{\rho,T}_R$, let us define the exit time
\[
\tau^{\e}_{D^{\rho,T}_{R}}=\inf\left\{t>0: X^{\e}_{x}(t)\notin D^{\rho,T}_{R}\right\}.
\]

 We have the following result for the lower bound of the exit time if the initial point $x \in D^{\rho,T}_R$.
\begin{lemma} \label{L:lower-bound-restricted}
  Let $\rho>0$ be small enough so that $D^{\rho,T}_R$ is uniformly attractive. Then for $R,T>0$ large enough and for any $x\in D^{\rho,T}_R$
  \begin{equation*} %\label{eq:exit-time-lower-exp}
    \liminf_{\e \to 0} \e \log \E \tau^\e_x \geq V(x_{*},\partial D).
  \end{equation*}

  Furthermore for any $\gamma>0$
  \begin{equation*} %\label{eq:exit-time-lower-prob}
    \lim_{\e \to 0} \Pro( \e \log \tau^\e_x < V(x_{*}, \partial D) - \gamma) = 0.
  \end{equation*}
\end{lemma}

\begin{proof}[Proof of Lemma \ref{L:lower-bound-restricted}]
Due to the uniform attractiveness and boundedness of $D^{\rho,T}_R$, Lemmas 8.8, 8.9 and 8.10 of \cite{BudhirajaDupuisSalins2018} remain valid here as far as exit properties of the process from $D^{\rho,T}_R$ are concerned. Then, arguing as in Theorem 5.7.11 of \cite{DemboZeitouni1997} we obtain that for $x\in D^{\rho,T}_R$ and any $\gamma>0$
  \begin{equation*}% \label{eq:exit-time-lower-prob2}
    \lim_{\e \to 0} \Pro( \e \log \tau^{\e}_{D^{\rho,T}_{R}} < V(x_{*}, \partial D^{\rho,T}_{R}) - \gamma) = 0.
  \end{equation*}

Then, Corollary \ref{L:InnerBoundaryRegularityQP_R} and  the almost sure  estimate $\tau^\e_x \geq \tau^{\e}_{D^{\rho,T}_R} $ gives for $T>0$ sufficiently large
 \begin{align} %\label{eq:exit-time-lower-prob3}
 \Pro( \e \log \tau^{\e}_{x} < V(x_{*}, \partial D) - 2\gamma) &\leq \Pro( \e \log \tau^{\e}_{x} < V(x_{*}, \partial D^{\rho,T}_R) - \gamma)\nonumber\\
 &\leq \Pro( \e \log \tau^{\e}_{D^{\rho,T}_R} < V(x_{*}, \partial D^{\rho,T}_R) - \gamma),\nonumber
   \end{align}
from which the result follows. The lower bound for the expectation of $\tau^\e_x$ follows by the last display and Chebyshev's inequality.
\end{proof}

Next we prove the full lower bound of the exit time for an initial point $x \in D$.
\begin{theorem} \label{thm:lower-bound}
  For any  $x \in D$ we have
  \begin{equation} \label{eq:exit-time-lower-exp-general}
    \liminf_{\e \to 0} \e \log \E \tau^\e_x \geq V(x_{*},\partial D).
  \end{equation}

  Furthermore for any $\gamma>0$ and for any  $x \in D$
  \begin{equation} \label{eq:exit-time-lower-prob-general}
    \lim_{\e \to 0} \Pro( \e \log \tau^\e_x \leq V(x_{*}, \partial D) - \gamma) = 0.
  \end{equation}
\end{theorem}

\begin{proof}[Proof of Theorem \ref{thm:lower-bound}]
Let $x \in D$. Set $R=|x|_E + 1$ so that $x \in D_R$.
Because $D_{R}$ is open, there exists $\rho_0$ such that $B(x_*,\rho_0) \subset D_{R}$. Because $x \in D$, there exists $T_0>0$ such that for all $T\geq T_0$, $X^0_x(T) \in B(x_*, \rho_0)$. Therefore, $x \in D^{\rho_0,T_0}_R$.

Thus the theorem 
follows directly from Lemma \ref{L:lower-bound-restricted}.
\end{proof}

\section{Proof of Theorem \ref{thm:exit-shape}: Large deviations lower and upper bounds for exit shape}\label{S:ExitShape}

In this section we analyze the limiting distribution of the exit shape $X^\e_x(\tau^\e_x)$ and prove Theorem \ref{thm:exit-shape}.

Let $\Theta \subset \bar{D}$ be any set that contains all $\omega-$limit points of the unperturbed process $X^0_x$ for all $x \in \bar D$. We demonstrate that it is exponentially unlikely for  $X^\e_x(t)$ to stay away from $\Theta$ for long periods of time.
As a first step here we prove that for any small $\rho>0$, there is a time $T_0>0$, independent of $x \in D$ (but depending on $\rho$) such that the unperturbed process has the property that $\dist_E(X^0_x(T_0), \Theta)< \rho/2$ for all $x \in D$.
\begin{lemma} \label{lem:unif-attract}
  For any small $\rho>0$, there exists $T_0>0$ such that
  \begin{equation*}
    \sup_{x \in D} \dist_E\left(X^0_x(T_0), \Theta\right) < \rho/2.
  \end{equation*}
\end{lemma}

\begin{proof}
  In the finite dimensional setting considered by Friedlin and Wentzell, this uniform convergence was a consequence of the compactness of $\bar D$. In our setting, we do not assume that $\bar D$ is compact and we do not even assume that $D$ is bounded. Even though the initial conditions are not in a compact set, we know that $\{X^0_x(1), x \in \bar D\}$ is a pre-compact set by Corollary \ref{cor:pre-compact-all-init-cond}.

  Suppose by contradiction that there exists $x_n \in \bar{D}$ and $T_n \to +\infty$ such that $\dist_E(X^0_{x_n}(t), \Theta)\geq \rho/2$ for all $t \in [0,T_n]$. By compactness, there exists a subsequence (relabeled $(x_n, T_n)$) along which $X^0_{x_n}(1) \to y$ in $\bar{D}$ for some $y \in \bar{D}$.
 % Because $K_i$ were defined in a way to guarantee $\lim_{t \to \infty}\dist_E( X^0_y(t), \Theta) = 0$,
   Because of the definition of $\Theta$ we are guaranteed that $\lim_{t \to \infty}\dist_E( X^0_y(t), \Theta) = 0$, which means that
  there exists $\tilde T_0>0$ such that
  $$\dist_E\left(X^0_y(T), \Theta\right) < \rho/4$$
   for any $T\geq \tilde{T}_0$. By continuity in initial conditions (Theorem \ref{thm:continuity}), for large enough $n$, $|X^0_{X^0_{x_n}(1)}(\tilde T_0) - X^0_y(\tilde T_0)|_E < \rho/4$. Therefore for large $n$,
  \[
   \dist_E\left(X^0_{x_n}(\tilde T_0 + 1), \Theta\right) = \dist_E \left(X^0_{X^0_{x_n}(1)}(\tilde T_0), \Theta\right) < \rho/2,
   \]
  contradicting the assumption.
\end{proof}

\begin{lemma} \label{lem:control-lower-bound}
  Let $\rho,T_0>0$ be as in Lemma \ref{lem:unif-attract}. There exists $a>0$ such that whenever $X^{0,u}_x$ is a controlled process with $x \in \bar D$, $u \in L^2([0,T_0+1]:H)$ and the properties that $X^{0,u}_x(t) \in \bar{D}$ for $t \in [0,T_0]$ and
  \begin{equation*}
    \dist_E\left(X^{0,u}_x(t), \Theta\right) \geq 3\rho/4 \text{ for } t \in [0,T_0+1],
  \end{equation*}
  the control $u$ must have the lower bound
  \begin{equation*}
    \frac{1}{2} \int_0^{T_0+1} |u(s)|_H^2 ds > a.
  \end{equation*}
\end{lemma}

\begin{proof}
  Assume by contradiction that there is a sequence $|u_n|_{L^2([0,T_0+1]:H)} \to 0$ and $x_n \in \bar{D}$ such that
  \begin{equation*}
    \dist_E(X^{0,u_n}_{x_n}(t), \Theta) \geq 3\rho/4 \text{ for } t \in [0,T_0+1].
  \end{equation*}

  By Corollary \ref{cor:pre-compact-all-init-cond}, a subsequence of this has the property that $X^{0,u_n}_{x_n}(1) \to y \in \bar{D}$. By continuity with respect to control and initial condition (Theorem \ref{thm:continuity}),
  \[\lim_{n \to \infty}|X^{0,u_n}_{x_n}(1+T_0) - X^0_y(T_0)|_E = 0. \]

  Lemma \ref{lem:unif-attract} guarantees that
  \[\dist_E\left(X^0_y(T_0), \Theta\right)< \rho/2,\]
  contradicting the assumption that
  \[\dist_E\left(X^{0,u_n}_{x_n}(1+T_0), \Theta\right) \geq 3\rho/4.\]
\end{proof}

\begin{lemma} \label{lem:impossible-to-stay-away}
For small $\rho>0$,
  \begin{equation}
    \lim_{T\to \infty} \limsup_{\e\to 0} \sup_{x \in D} \e \log \Pro\left(X^\e_x(t) \in D, \dist_E\left(X^\e_x(t),\Theta\right)\geq \rho, t \in [0,T]\right) = -\infty.\nonumber
  \end{equation}
\end{lemma}

\begin{proof}
  Let $T_0>0$ and $a>0$ be as in  Lemma \ref{lem:control-lower-bound}. By applying Lemma \ref{lem:control-lower-bound} $k$ times, for any $k \in \mathbb{N}$, we know that whenever
  $X^{0,u}_x$ is a controlled process with $x \in \bar D$, $u \in L^2([0,k(T_0+1)]:H)$ and the properties that $X^{0,u}_x(t) \in D$ for $t \in [0,k(T_0 + 1)]$ and
  \begin{equation*}
    \dist_E\left(X^{0,u}_x(t), \Theta\right) \geq 3\rho/4 \text{ for } t \in [0,k(T_0+1)],
  \end{equation*}
  the control $u$ must have the lower bound
  \begin{equation*}
    \frac{1}{2} \int_0^{k(T_0+1)} |u(s)|_H^2 ds > ka.
  \end{equation*}

  This means that we have the inclusion that for any $x \in D$
  \begin{align}%\label{eq:inclusion}
    &\Bigg\{\varphi \in C([0,k(T_0+1)]:E): \varphi(0)=x, \varphi(t) \in D,\nonumber \\
    &\qquad\qquad\dist_E\left(\varphi(t), \Theta\right) \geq \rho, t \in [0,k(T_0+1)] \Bigg\}\nonumber\\
    &\subset
    \Bigg\{\varphi \in C([0, k(T_0+1):E]): \varphi(0)=x, \nonumber\\
     &\qquad\qquad \dist_{C([0,k(T_0+1)]:E)} \left(\varphi, \Phi^{k(T_0+1)}_x(ka) \right) \geq \rho/4 \Bigg\}.\nonumber
  \end{align}

  The uniform large deviations principle (Theorem \ref{thm:LDP}) was proven over bounded subsets of $E$. Unfortunately, $D$ may be an unbounded set. To deal with this situation, we use the fact that Corollary \ref{cor:exp-est-X} guarantees that $X^\e_x$ is exponentially unlikely to remain in an unbounded part of the set after positive amounts of time. In particular, for any $V>0$ we can find $R>0$ such that
  \begin{equation} \label{eq:bounded-set-less-V}
    \limsup_{\e \to 0} \sup_{x \in D} \e \log \Pro(|X^\e_x(1)|_E \geq R) \leq -V.
  \end{equation}

  By the Markov property, letting $D_R = \{x \in D: |x|_E \leq R\}$,
  \begin{align} \label{eq:Markov-bounded-set}
    &\sup_{x \in D}\Pro\left(X^\e_x(t) \in D, \dist_E\left(X^\e_x(t),\Theta\right) \geq \rho,  t \in [0,1+ k(T_0+1)] \right)\nonumber \\
    &\leq \sup_{x \in D}\Pro(|X^\e_x(1)| \geq R) + \sup_{x \in D_R}   \Pro\left(X^\e_x(t) \in D, \dist_E\left(X^\e_x(t), \Theta\right) \geq \rho,  t \in [0,k(T_0+1)] \right).
  \end{align}

  By the uniform large deviation upper bound \eqref{eq:ldp-up},
  \begin{align*}
    &\lim_{\e \to 0} \sup_{x \in D_R} \e \log \Pro\left( X^\e_x(t) \in D, \dist_E\left(X^\e_x(t), \Theta\right) \geq \rho,  t \in [0,k(T_0+1)] \right)\\
    & \leq \lim_{\e \to 0} \sup_{x \in D_R} \e \log \Pro \left( \dist_{C([0,k(T_0+1)]:E)} \left(X^\e_x, \Phi^{k(T_0+1)}_x(ka) \right)\geq \rho/4\right)\\
    &\leq -ka.
  \end{align*}

  By \eqref{eq:Markov-bounded-set} and \eqref{eq:bounded-set-less-V},
  \begin{align}
    &\limsup_{\e \to 0} \sup_{x \in D} \e \log\Pro\left(X^\e_x(t) \in D, \dist_E\left(X^\e_x(t), \Theta\right) \geq \rho,  t \in [0,1+ k(T_0+1)] \right)\nonumber\\
    &\leq -\min\{V, ka\}.\nonumber
  \end{align}

  Both $V>0$ and $k \in \mathbb{N}$ are arbitrary, implying the result.
\end{proof}

\subsection{Proof of upper bound \eqref{eq:shape-LDP-upper}} \label{SS:Shape-LDP-up}

In this subsection we prove the upper bound \eqref{eq:shape-LDP-upper}. For presentation purposes, we first define and discuss the properties of an appropriate Markov chain $Z_n$, in subsection \ref{SS:Markov-transitions}. The actual proof of the upper bound \eqref{eq:shape-LDP-upper} follows then in subsection \ref{SSS:Shape-LDP-up}.

\subsubsection{Transitions for a related Markov chain} \label{SS:Markov-transitions}
Fix $s \geq 0$. By the definition of $\hat{V}_D$ and Lemma \ref{lem:paths-stay-near-K_i}, $\hat{J}(y_1)=\hat{J}(y_2)$ for any $i \in \{1,...,N\}$ and  $y_1, y_2 \in K_i$. We can call this quantity $\hat{J}(K_i)$. We reorder the $V$ equivalence classes $K_i$ and possibly lower the value of $N$ so that $K_1,...,K_N$ is the set of $V$ equivalence classes that are subsets of $\partial D$ and satisfy $\hat{J}(K_i) \leq s$.

With $\hat{\Psi}$ defined by  \eqref{eq:shape-LDP-upper} and for a given $\delta>0$ define $K_{N+1}$ by
\begin{equation*} %\label{eq:K_N+1}
  K_{N+1} := \left\{ x \in \partial D: \dist_E(x, \hat{\Psi}(s)) \geq \delta \right\}.
\end{equation*}
%where $\hat{\Psi}$ is from \eqref{eq:shape-LDP-upper}.

In addition, let us set $K_0=\{x_*\}$.

By possibly decreasing $\delta>0$ we can guarantee that $K_{N+1}$ contains all of the $V$ equivalence classes $K_j$ satisfying $ K_j \subset \partial D$ and $\hat{J}(K_j)>s$. By Assumption \ref{A:K_equivalenceClassV}, the set $\bigcup_{i=0}^{N+1} K_i$ contains all of the $\omega$-limit points of $X^0_x$ for $x \in \bar D$ and therefore Lemma \ref{lem:impossible-to-stay-away} holds for this collection, i.e., we may set $\Theta=\bigcup_{i=0}^{N+1} K_i$.

For small $\rho>0$ and for $i\in\{0,\cdots, N+1\}$ let $g_i$, $G_i$ be defined as follows
\begin{align}
  g_i: &= \left\{x \in D: \dist_E(x, K_i)<\rho \right\}\nonumber\\
  G_i:&= \left\{x \in D: \dist_E(x, K_i) = 2\rho \right\}.\nonumber
\end{align}

With these definitions, define the stopping times
\begin{align}
  &\tau_0 = 0,\nonumber\\
  &\sigma_k = \inf\{t\geq\tau_k: X^\e_x(t) \in \bigcup_{i=0}^{N} G_i\cup \partial D \},\nonumber\\
  &\tau_{k+1} = \inf\{t\geq\sigma_k: X^\e_x(t) \in \bigcup_{i=0}^{N} g_i \cup \partial D\}.\nonumber %\label{eq:stopping-times}
\end{align}

We consider the Markov chain $Z_n = X^\e_x(\tau_n)$ on $\bigcup_{i=0}^{N+1} g_i \cup \partial D$. Notice that if $Z_n \in \partial D$, then the Markov chain stops and $Z_m=Z_n$ for all $m \geq n$.

\begin{remark}
We remark here that the sequences of stopping times $\{\sigma_{k}\}_{k \geq 1}$ and $\{\tau_{k}\}_{k\geq 0}$ is different than the ones used in \cite{Day1990} or \cite{FWbook}. In the case of finite dimensions, \cite{Day1990} used a reflected process in order to deal with the characteristic boundary. It is unclear how to do that in infinite dimensions. These considerations led us to this choice of stopping times and then to the corresponding Markov chain $Z_n$.
\end{remark}

\begin{lemma}[Transition probabilities] \label{lem:g0-to-sets}
  For any $\gamma>0$, there exists $\rho_0>0$ such that for all $\rho \in (0,\rho_0)$ and $j \in \{1,...,N +1\}$
    \begin{equation} \label{eq:g0-to-bdry-low}
    \liminf_{\e \to 0} \inf_{x \in g_0} \e \log \Pro(Z_1 \not \in g_0 | Z_0=x) \geq -V(x_*,\partial D) - \gamma.
  \end{equation}

  \begin{equation} \label{eq:g0-to-gi-up}
    \limsup_{\e \to 0} \sup_{x \in g_0} \e \log \Pro(Z_1 \in g_j | Z_0 =x) \leq -\tilde{V}_{D}(x_*,K_j) + \gamma.
  \end{equation}

  For any $\gamma>0$, there exists $\rho_0>0$ such that for all $\rho \in (0, \rho_0)$ and $i \not =j $ with $i \in \{1,....,{N}\}$  and $j \in \{1,....,{N+1}\}$,
  %such that $K_i$ is pre-compact and all points in $K_i$ are equivalent
  \begin{equation} \label{eq:gi-to-bdry-low}
    \liminf_{\e \to 0} \inf_{x \in g_i} \e \log \Pro(Z_1 \in \partial D | Z_0=x) \geq - \gamma
  \end{equation}
  \begin{equation} \label{eq:gi-to-gj-up}
   \limsup_{\e \to 0} \sup_{x \in g_i} \e \log \Pro(Z_1 \in g_j | Z_0 =x) \leq -\tilde{V}_{D}(K_i,K_j) + \gamma.
    %\limsup_{\e \to 0} \sup_{x \in g_i} \e \log \Pro(Z_1 \in g_j | Z_0 =x) \leq -V_D(K_i,K_j) + \gamma.
  \end{equation}
\end{lemma}

Before proving the lemma we make the following remark.
\begin{remark}
The fact that $K_i \subset \partial D$ for $i \geq 1$ implies that there are points in $g_i$ that are arbitrarily close to $\partial D$. This is fundamentally different from the situation in Chapter 6.3 of \cite{FWbook} where they assume that there is a positive distance between the $K_i$ and $\partial D$. %One consequence of this lack of distance is that we can only prove %\eqref{eq:g0-to-gi-low} and
%\eqref{eq:g0-to-bdry-up} in the cases where we begin in $g_0$.
By beginning at $x \in g_i$ arbitrarily close to $\partial D$ we see that only the trivial bounds
\begin{equation}
  \inf_{x \in g_i} \Pro(Z_1 \in g_j| Z_0 =x) =  0 \text{ for }j \not = i
\end{equation}
and
\begin{equation}
  \sup_{x \in g_i} \Pro(Z_1 \in \partial D | Z_0 = x) = 1.
\end{equation}
hold.
\end{remark}

We will use the following preliminary results in the proof of Lemma \ref{lem:g0-to-sets}. 

\begin{lemma} \label{lem:V-cont-with-G-rho}
  We emphasize the dependence of $G_i$ on $\rho$ with the notation $G_i^{(\rho)}$. For any $i \in \{0,..,N\}$ and $j \in \{1,...,N+1\}$
  \begin{equation}
    \liminf_{\rho \to 0} \tilde{V}_{D}(G^{(\rho)}_i, G^{(\rho)}_j) \geq \tilde{V}_{D}(K_i,K_j).
  \end{equation}
\end{lemma}
\begin{proof}
  Assume that $\liminf_{\rho \to 0} \tilde{V}_{D}(G^{(\rho)}_i, G^{(\rho)}_j)$ is finite because the result is trivial otherwise.
  Let $\rho_n \downarrow 0 $ be a sequence of positive numbers such that
  \[\lim_{n \to \infty} \tilde{V}_{D}(G^{(\rho_n)}_i, G^{(\rho_n)}_j) = \liminf_{\rho \to 0} \tilde{V}_{D}(G^{(\rho)}_i, G^{(\rho)}_j).\]

  By the definition of $\tilde V_D$ \eqref{eq:V-tilde-of-sets}, there exist $x_n \in  G^{(\rho_n)}_i$ and $y_n \in  G^{(\rho_n)}_j$ such that $\tilde{V}_{D}(x_n,y_n) \leq  \tilde{V}_{D}(G^{(\rho_n)}_i, G^{(\rho_n)}_j)+ \frac{1}{n}$. By the compactness of $K_i$ (Corollary \ref{cor:equivalence-class}) and the definition of $G^{(\rho)}_i$, there exists a subsequence such that $x_n \to x \in K_i$. Because the set $\{x\} \cup \bigcup_n\{x_n\}$ is compact, Theorem \ref{thm:V-tilde-compact-level-sets} guarantees that there exists a subsequence such that $y_n \to y$. By the definitions of $G^{(\rho)}_j$, we know that $y \in K_j$. By Theorem \ref{thm:V-tilde-D-lsc},
  \[\tilde{V}_{D}(K_i, K_j) \leq \tilde{V}_{D}(x,y) \leq \liminf_{n \to \infty} \tilde{V}_{D}(x_n,y_n) \leq \liminf_{\rho \to 0} \tilde{V}_{D}(G^{(\rho)}_i, G^{(\rho)}_j).\]
\end{proof}

\begin{lemma}\label{L:ExitTrajectory}
Assume  $V(x_*, \partial D)<+\infty$. Then for any $\eta>0$, there exists some $T>0$ and  $\varphi\in C([0,T]:E)$ with the property that $\varphi(0)=x_*$, $\varphi(T) \in \partial D$, $\varphi(t) \not \in \partial D $ for $t \in [0,T)$  and $I_{x_*}(\varphi)<V(x_*,\partial D) + \eta$. In addition, by decreasing $\rho$ one can guarantee that $\varphi(t)$ does not hit $g_0$ after hitting $G_0$.
\end{lemma}
\begin{proof}[Proof of Lemma \ref{L:ExitTrajectory}]
Let $\eta>0$. Let $T>0$ and $\psi \in C([0,T]:E)$ such that $\psi(0) = x_*$, $\psi(T) \in \partial D$ and
\[I_{x_*}^T(\psi) < V(x_*, \partial D) + \frac{\eta}{3}.\]

By Theorem \ref{thm:control}, there exists $\delta_1>0$ such that we can connect $x_*$ to any controlled path starting in the $\delta_1$-ball around $x_*$ without increasing the rate function by more than $\eta/3$. Similarly, by the local Lipschitz continuity of $f$, we can find $\delta_2>0$ such that whenever $|x-x_*|_E<\delta_2$, it follows that $|F(x) - F(x_*)|_E< \frac{1}{4}$. Find a time $t_0>0$ such that $|\psi(t_0) - x_*|_E< \min\{\delta_1,\delta_2\}$. Let
\[\psi_2(t) = \psi(t+t_0) \text{ for } t \in [0,T-t_0]. \]
Clearly
\[I_{\psi(t_0)}^{T-t_0}(\psi_2) \leq I_{x_*}^T(\psi) < V(x_*, \partial D) + \frac{\eta}{3}.\]

Now we apply Theorem  \ref{thm:control} to get a path $\psi_3 \in C([0,T-t_0]:E)$ such that $\psi_3(0) = x_*$, $\psi_3(t) = \psi_2(t)$ for $t \geq |\psi(t_0) - x_*|_E$, and $I_{x_*}^{T-t_0}(\psi_3) < V(x_*,\partial D) + \frac{2\eta}{3}. $

We can assume that $\psi_3(t) \not = x_*$ for all $t \in (0,T-t_0]$. If this were not the case, then we could just begin $\psi_3$ at the last time that $\psi_3(t) = x_*$. This would only decrease its rate function.

Now we bound the upper and lower derivatives of the real-valued function $t \mapsto |\psi_3(t) - x_*|_E$. At $t=0$ $\psi_3(0) = x_*$. We calculate that for any $\xi \in \mathcal{O}$,
\begin{align*}
  & \liminf_{h \downarrow 0} \frac {|\psi_3(h) - x_*|_E}{h} \\
  &\geq \liminf_{h \downarrow 0} \frac{|\psi_3(h,\xi) - x_*(\xi)|}{h}\\
  &  \geq \left|\frac{\partial \psi_3}{\partial t}(0,\xi)\right|\\
  & \geq \left|A x_*(\xi) + f(\psi_2(0,\xi)) - \frac{\psi_3(t,\xi) - \psi_2(t,\xi)}{|\psi_3(t) - \psi_2(t)|_E} \mathbbm{1}_{\{|\psi_3(t) - \psi_2(t)|_E>0\}} \right|.
\end{align*}

Observe that because $x_*$ is an equilibrium of the unperturbed system, $A x_* = F(x_*)$. Choose $\xi \in \mathcal{O}$ such that $|\psi_3(0,\xi) - \psi_2(0,\xi)| = |\psi_3(0) - \psi_2(0)|_E$. By the triangle inequality,
\begin{align*}
  & \liminf_{h \downarrow 0} \frac {|\psi_3(h) - x_*|_E}{h}\geq 1 - |F(\psi_2(0)) - F(x_*)|_E  > \frac{3}{4}.
\end{align*}

This means that there exists $t_1 \in (0,t_0)$ such that
\begin{equation}
  |\psi_3(t) - x_*|_E > \frac{3}{4}t \text{ for } t \in (0,t_1).
\end{equation}

Looking at upper-left derivatives, for any $t \in (0,t_1)$, and $\xi \in \mathcal{O}$ such that $|\varphi(t,\xi) - x_*(\xi)| = |\varphi(t) - x_*|_E$,
\begin{align*}
  & \limsup_{h \downarrow 0} \frac{|\psi_3(t) - x_*|_E - |\psi_3(t-h) - x_*|_E}{h}\\
  &\leq \limsup_{h \downarrow 0} \frac{(\psi_3(t,\xi) - \psi_3(t-h,\xi))\sgn(\psi_3(t,\xi)-x_*(
  \xi))}{h} \\
  &= \frac{\partial \psi_3}{\partial t}(t,\xi) \sgn(\psi_3(t,\xi)- x_*(\xi))
  \\
  &\leq \left(A \psi_3(t,\xi) + f(\psi_2(t,\xi)) - \frac{\psi_3(t,\xi) - \psi_2(t,\xi)}{|\psi_3(t) - \psi_2(t)|_E}  \right)\sgn(\psi_3(t,\xi) - x_*(\xi))\\
  &\leq A (\psi_3(t,\xi) - x_*(\xi))\sgn(\psi_3(t,\xi) - x_*(\xi)) \\
  &\qquad+ (f(\psi_2(t,\xi)) - f(x_*(\xi)))\sgn(\psi_3(t,\xi) - x_*(\xi))\\
   &\qquad- \frac{\psi_3(t,\xi) - \psi_2(t,\xi)}{|\psi_3(t) - \psi_2(t)|_E}\sgn(\psi_3(t,\xi) - x_*(\xi)).
\end{align*}

The previous line follows because $A x_* = F(x_*)$. Because $A$ is a second-order elliptic differential operator,
\[A (\psi(t,\xi) - x_*(\xi))\sgn(\psi_3(t,\xi) - x_*(\xi)) \leq 0.\]

By our original choices of $\delta_1, \delta_2$, and for sufficiently small $t>0$,
\[(f(\psi_2(t,\xi)) - f(x_*(\xi)))\sgn(\psi_3(t,\xi) - x_*(\xi)) \leq |F(\psi_2(t)) - F(x_*)|_E < \frac{1}{4}. \]

It is clear that
\[-\frac{\psi_3(t,\xi) - \psi_2(t,\xi)}{|\psi_3(t) - \psi_2(t)|_E}\sgn(\psi_3(t,\xi) - x_*(\xi)) \leq  1.\]

Therefore, we have that
\begin{equation}
  \limsup_{h \downarrow 0} \frac{|\psi_3(t) - x_*|_E - |\psi_3(t-h) - x_*|_E}{h} < \frac{5}{4}.
\end{equation}

There exists $t_2 \in (0,t_1)$ such that
\begin{equation}
  \frac{3h}{4} < |\psi_3(h) - x_*|_E < \frac{5h}{4} \text{ for } h \in (0,t_2).
\end{equation}

There cannot be a $\rho>0$ and $0<h_1< h_2<t_2$ such that $\psi_3(h_1) \geq 2\rho$ and $\psi_3(h_2) \leq \rho$ because this would imply that
\begin{equation}
  2 \rho < \frac{5}{4}h_1 \text{ and } \rho> \frac{3}{4}h_2,
\end{equation}
which is a contradiction. Finally, because $\psi_3$ is continuous and $|\psi_3(t) - x_*|_E>0$ for all $t> t_2$, there exists $\rho>0$ such that
\[\rho < \min_{t \in [t_2, T-t_0]} |\psi_3(t) - x_*|_E\]

For this value of $\rho>0$ we can conclude that $\psi_3(t)$ never hits $\gamma_\rho$ after hitting $\Gamma_\rho$. The trajectory $\psi_3(t)$ is the claimed $\varphi(t)$ concluding the proof of the lemma.
\end{proof}

\begin{proof}[Proof of Lemma \ref{lem:g0-to-sets}]

  We begin with the lower bounds.

  Let us first prove \eqref{eq:g0-to-bdry-low}.  Using Lemma \ref{lem:control-to-x+} and the fact that $X^0_{x_*}(t)=x_*$ for all $t>0$, by choosing $\rho>0$ small enough we can guarantee that for any $x \in g_0$ there exists $\varphi_1^x \in C([0,\rho]:E)$ such that $\varphi_1^x(0)=x$, $\varphi_1^x(\rho)=x_*$, and $I_x^\rho(\varphi^x_1)< \gamma/3$.

  If $V(x_*, \partial D)<+\infty$, then there exists some $T>0$,  $\varphi_2\in C([0,T]:E)$ with the property that $\varphi_2(0)=x_*$, $\varphi_2(T) \in \partial D$, $\varphi_2(t) \not \in \partial D $ for $t \in [0,T)$. By decreasing $\rho$ we can guarantee that $\varphi_2(t)$ does not hit $g_0$ after hitting $G_0$, see Lemma \ref{L:ExitTrajectory}.

  By Assumption \ref{assum:D-boundary} and the fact that $\varphi_2(T) \in \partial D$, there exists $z \not \in \bar{D}$ such that $|z-\varphi_2(T)|<\rho$. % and $z \not \in \bigcup_{i \not =0} g_i $.
  By Assumption \ref{A:AttractionProperty} $X^0_z(t) \not \in \bar{D}$ for all $t>0$. By Theorem \ref{thm:control}, there exists $\varphi_3 \in C([0,\rho]:E)$ such that $\varphi_3(0) = \varphi_2(T)$, $\varphi_3(\rho) \not \in \bar{D}$, and $I^\rho_{\varphi_2(T)}(\varphi_3) < \frac{\gamma}{3}.$

  Concatenate these paths letting
  \[\varphi^{x}(t) :=\begin{cases}
    \varphi_1^{x}(t) & \text{ if } t \in [0,\rho)\\
    \varphi_2(t-\rho) & \text{ if } t \in [\rho, T+\rho)\\
    \varphi_3(t-T-\rho) & \text{ if } t \in [T+\rho, T+2\rho).
  \end{cases}
  \]
  %  $\varphi^x = \varphi^x_1 \cup \varphi_2 \cup \varphi_3$.

  Then $\varphi^x(0) = x$, $\varphi^x(T + 2\rho) \not \in \bar{D}$, $\varphi^x(t)$ does not hit $g_0$ after leaving $G_0$, and $I_x^{T+2\rho}(\varphi^x)< V(x_*, \partial D) + \gamma$. %If $a:= \min\{\dist_E(\varphi^x(T+2\rho),D), \rho\}>0$.
  We can choose $a>0$ small enough so that any $\psi \in C([0,T]:E)$ with the property $|\psi - \varphi^x|_{C([0,T+2\rho]:E)}<a/2$ will not hit $g_0$ after hitting $G_0$ and will reach one of the $G_i$ $i \in\{0,...,N+1\}$ or the boundary $\partial D$.
  Then
  \begin{equation*}
    \Pro(Z_1 \not \in g_0 | Z_0 =x) \geq \Pro(|X^\e_x - \varphi^x|_{C([0,T+2\rho]:E)}< a/2).
  \end{equation*}

  By the LDP lower bound, \eqref{eq:g0-to-bdry-low} follows.

  Now we prove \eqref{eq:gi-to-bdry-low}. Let $\delta$ respond to the $\gamma$ in Theorem \ref{thm:control}. Let $\rho< \delta/4$. $K_i$ is  compact by Corollary \ref{cor:equivalence-class}. Therefore there exists a finite collection $\{y_1,...,y_l\}$ such that $K_i \subset \bigcup_{j=1}^l B(y_j,\rho/8)$. By Assumption \ref{assum:D-boundary}, there exists a finite collection of $\{z_1,...,z_l\}$ such that $|y_i - z_i|< \rho/8$ and $z_i \not \in \bar{D}$. Therefore $K_i \subset \bigcup_{i=1}^l B(z_i, \rho/4)$. Additionally, by the continuity of the unperturbed equation with respect to the initial condition (Theorem \ref{thm:continuity}), we can choose these $z_1,...,z_l$ such that
  \begin{equation} \label{eq:zj-close-to-Ki}
     \sup_{t \in[0,\rho/4]} \max_{j \in \{1,...,l\}} \dist_E(X^0_{z_j}(t), K_i)< \rho/4.
   \end{equation}

     Let $x\in g_{i}$ and let $y\in K_{i}$ be such that $|x-y|_{E}<\rho$. By Theorem \ref{thm:control} there exists a path $\phi_{1}^{x}\in C([0,\rho];E)$ such that $\phi_{1}^{x}(0)=x$, $\phi_{1}^{x}(\rho)=X^{0}_{y}(\rho)$ and $I^{\rho}_x(\phi_{1}^x)<\gamma/2$. By definition of $K_{i}$ we have that $X^{0}_{y}(t)\in K_{i}$  and $\dist(\phi_1^x(t), K_i)< \rho$ for all $t \in [0,\rho]$. There exists $z_j\in\{z_1, \cdots, z_l\}$ such that $|X^{0}_{y}(\rho) - z_j|_E< \rho/4$. %Let us set $z_{i^{*}}=\textrm{argmin}\{z_{i}, i=1,\cdots, l | |X^{0}_{y}(\rho) - z_i|_E\}$.
    By Theorem \ref{thm:control} there exists a path $\phi_{2}^{X^{0}_{y}(\rho)}\in C([0,\rho/4];E)$ such that $\phi_{2}^{X^{0}_{y}(\rho)}(0)=X^{0}_{y}(\rho)$, $\phi_{2}^{X^{0}_{y}(\rho)}(\rho/4)=X^{0}_{z_{j}}(\rho/4)$, $|\phi_2^{X^0_y}(t) - X^0_{z_j}(t)| \leq \rho/4$ for $t \in [0,\rho/4]$, and $I^{\rho/4}_{X^{0}_{y}(\rho)}(\phi_{2}^{X^{0}_{y}(\rho)})<\gamma/2$. We next concatenate the two paths $\phi_{1}$ and $\phi_2$ to get the path  $\phi^{x}\in C([0,5\rho/4];E)$ such that $\phi^{x}(0)=\phi_{1}^{x}(0)$, $\phi^{x}(\rho)=\phi_{2}^{X^{0}_{y}(\rho)}(0)$, $\phi^{x}(5\rho/4)=\phi_{2}^{X^{0}_{y}(\rho)}(\rho/4) = X^0_{z_j}(\rho/4)$ and $I^{5\rho/4}_{x}(\phi^{x})<\gamma$. By Theorem \ref{thm:control}, and \eqref{eq:zj-close-to-Ki}, $\dist_E(\phi^x(t),K_i)< \rho$ for all $t \in [0, 5\rho/4]$. Consequently, $\dist(\phi^x(t), G_i)>\rho$ for $t \in [0, 5\rho/4]$. Since $z_1,...,z_l \not \in \bar D$,  there exists $\beta \in (0, \rho)$ such that $\dist_E(X^0_{z_i}(\rho/4), D)>\beta$ for all $j \in \{1,...,l\}$.

Therefore, we have the inclusion
\[
\left\{|X^\e_x - \varphi^x|_{C([0,5\rho/4]:E)}< \frac{\beta}{2}\right\}\subseteq\left\{X^\e_x(\tau_1) \in \partial D\right\}.
\]
We can conclude from the LDP lower bound \eqref{eq:ldp-low} that
  \begin{align}
    & \liminf_{\e \to 0} \inf_{x \in g_i} \e \log \Pro(X^\e_x(\tau_1) \in \partial D) \geq\nonumber\\
     &\qquad\geq\liminf_{\e \to 0} \inf_{x \in g_i} \e \log \Pro(|X^\e_x - \varphi^x|_{C([0,5\rho/4]:E)}< \frac{\beta}{2}) \geq -\gamma.\nonumber
  \end{align}

  Now we do the  arguments for the upper bounds.  By Lemma \ref{lem:V-cont-with-G-rho}, for any given $\gamma>0$, there exists $\rho_0$ such that for all $\rho \in (0, \rho_0)$, {$\tilde{V}_{D}(G^{(\rho)}_i, G^{(\rho)}_j) \geq \tilde{V}_{D}(K_i,K_j)-\gamma$}. A consequence is that any path $\varphi \in C([0,T]:E)$ with the property that $\varphi(0) \in  G_i$ and $\varphi(T) \in G_j$ has the property that $I_{\varphi(0)}^T(\varphi) \geq \tilde{V}_{D}(K_i,K_j)-\gamma$. 
  Therefore we have the containment for any $x \in G_i$ and $T>0$,
  \begin{align} \label{eq:claim}
    &\left\{ \varphi \in C([0,T]:E): \varphi(0) =x, \varphi(T) \in g_j, \varphi(t) \in D, t \in [0,T)\right\}\nonumber\\
     &\subset \left\{\varphi \in C([0,T]:E): \dist_{C([0,T]:E)}(\varphi, \Phi^T_x(\tilde{V}_{D}(K_i, K_j) - \gamma))\geq\rho \right\}.
  \end{align}

  We will be able to get the positive distance $\rho$ in the set on the right-hand side because  Lemma \ref{lem:V-cont-with-G-rho}  states that paths $\varphi \in \Phi^T_x(\tilde{V}_{D}(K_i, K_j) - \gamma)$ cannot reach $G_j = \{y \in D : \dist_E(y,K_j)=2\rho\}$. The set on the left-hand side of the above expression is a set of paths that reach $g_j = \{y \in D: \dist_E(y, K_j)<\rho\}$. This means that such paths need to be at least distance $\rho$ from the level sets.

  By the strong Markov property we can restart the process at $\sigma_1$ when the process first reaches $G_i$ to see that
  \begin{equation*}
    \sup_{x \in g_i} \Pro(X^\e_x(\tau_1) \in g_j) = \sup_{x \in G_i} \Pro(X^\e_x(\tau_1) \in g_j).
  \end{equation*}

  For $T>0$ to be chosen later, we can decompose the events
  \begin{equation*}
    \sup_{x \in G_i} \Pro(X^\e_x(\tau_1) \in g_j) \leq \sup_{x \in G_i} \Pro(\tau_1>T) + \sup_{x \in G_i} \Pro(X^\e_x(\tau_1) \in g_j, \tau_1\leq T).
  \end{equation*}

  Let $T>0$ be big enough so that Lemma \ref{lem:impossible-to-stay-away} guarantees that
  \begin{equation} \label{eq:big-time}
    \limsup_{\e \to 0} \sup_{x \in G_i} \e \log \Pro(\tau_1>T) < - \tilde{V}_{D}(K_i, K_j).
  \end{equation}

  By \eqref{eq:claim} along with the LDP upper bound \eqref{eq:ldp-up} it follows that
  \begin{align} \label{eq:small-time}
    &\limsup_{\e \to 0} \sup_{x \in G_i} \e \log \Pro(X^\e_x(\tau_1) \in g_j, \tau_1<T) \nonumber\\
    &\leq \limsup_{\e \to 0} \sup_{x \in G_i} \e \log \Pro(\dist_{C([0,T]:E)}(X^\e_x, \Phi_x^T(\tilde{V}_{D}(K_i, K_j) - \gamma)) \geq \rho)\nonumber\\
     &\leq -\tilde{V}_{D}(K_i, K_j) +\gamma.
  \end{align}

  It follows from \eqref{eq:big-time} and \eqref{eq:small-time} that
  \begin{equation*}
    \limsup_{\e \to 0} \sup_{x \in g_i} \e \log \Pro(X^\e_x(\tau_1) \in g_i) \leq -\tilde{V}_{D}(K_i, K_j) + \gamma,
  \end{equation*}
  proving \eqref{eq:gi-to-gj-up}. The proof of \eqref{eq:g0-to-gi-up} is similar and thus omitted.
 % Now we prove \eqref{eq:gi-to-gj-up} \textcolor{red}{[K:Proof is the same, but we needs to be include due to the behavior of the paths near the boundary]}.
\end{proof}

Now we define some stopping times for the Markov chain. Let
\begin{equation*} %\label{eq:eta-stopping-time}
  \eta_i  = \min\{n \in \mathbb{N}: Z_n \not \in g_i\}.
\end{equation*}

\begin{lemma}
  Given $\gamma>0$ there exists $\rho_0>0$ such that for all $\rho \in (0, \rho_0)$, for any $i \not=j$, $i \in\{1,\cdots,N\}$, $j \in\{1,\cdots,N+1\}$
  \begin{equation} \label{eq:Z-eta-upper-gi}
    \limsup_{\e \to 0} \sup_{x \in g_i} \e \log \Pro(Z_{\eta_i} \in g_j | Z_0 = x) \leq - (\tilde{V}_{D}(K_i,K_j) - \gamma).
   % \limsup_{\e \to 0} \sup_{x \in g_i} \e \log \Pro(Z_{\eta_i} \in g_j | Z_0 = x) \leq - (V_{D}(K_i,K_j) - \gamma).
  \end{equation}
  \begin{equation} \label{eq:Z-eta-lower-gi}
    \liminf_{\e \to 0} \inf_{x \in g_i} \e \log \Pro(Z_{\eta_i} \in \partial D | Z_0 = x) \geq -\gamma.
  \end{equation}
    and for any $j \in \{1,...,N+1\}$,
  \begin{equation} \label{eq:Z-eta-upper-g0}
    \limsup_{\e \to 0} \sup_{ x \in g_0} \e \log \Pro(Z_{\eta_0} \in g_j | Z_0 =x) \leq -(\tilde{V}_{D}(x_*, K_j) - V(x_*, \partial D) - \gamma).
  \end{equation}
 % \begin{equation} \label{eq:Z-eta-lower}
%    \liminf_{\e \to 0} \inf_{x \in g_0} \e \log \Pro(Z_{\eta_0} \in g_j | Z_0 =x) \geq -(V_D(x_*, K_j) - V(x_*, \partial D) + \gamma).
%  \end{equation}
\end{lemma}

\begin{proof}
Fix $i \in \{0,...,N\}, j \in \{1,....,N+1\}$. We decompose the event over all of the possible values of $\eta_i$. By the strong Markov property,
  \begin{align*}
    &\sup_{x \in g_i}\Pro(Z_{\eta_i} \in g_j | Z_0=x) = \sup_{x \in g_i} \sum_{k=1}^\infty \Pro(\eta_i =k, Z_k \in g_j | Z_0=x) \\
    &\leq\sum_{k=1}^\infty \left( \sup_{x \in g_i} \Pro(Z_1 \in g_i | Z_0 = x) \right)^{k-1} \left(\sup_{x \in g_i} \Pro(Z_1 \in g_j | Z_0=x) \right)\\
    &\leq \frac{\sup_{x \in g_i} \Pro(Z_1 \in g_j | Z_0=x)}{1 - \sup_{x \in g_i} \Pro(Z_1 \in g_i | Z_0 = x)}
    \leq \frac{\sup_{x \in g_i} \Pro(Z_1 \in g_j | Z_0=x)}{\inf_{x \in g_i} \Pro(Z_1 \not \in g_i | Z_0 = x) }.%\\
    %&\leq \frac{\sup_{x \in g_i} \Pro(Z_1 \in g_j | Z_0=x)}{\inf_{x \in g_i} \Pro(Z_1 \in \partial D | Z_0 = x) }\\
  \end{align*}

 Then \eqref{eq:Z-eta-upper-gi} follows from \eqref{eq:gi-to-gj-up} and \eqref{eq:gi-to-bdry-low} using the estimate
  \[
  \Pro(Z_1 \not \in g_i | Z_0 = x)\geq \Pro(Z_1  \in \partial D | Z_0 = x).
  \]
  \eqref{eq:Z-eta-upper-g0} follows from \eqref{eq:g0-to-gi-up} and \eqref{eq:g0-to-bdry-low}.

  \eqref{eq:Z-eta-lower-gi} is a simple consequence of \eqref{eq:gi-to-bdry-low} and the fact that
  \begin{equation*}
    \inf_{x \in g_i}\Pro(Z_{\eta_i} \in \partial D | Z_0=x) \geq \inf_{x \in g_i} \Pro(Z_1 \in \partial D | Z_0=x).
  \end{equation*}
\end{proof}

\subsubsection{Proof of upper bound \eqref{eq:shape-LDP-upper}} \label{SSS:Shape-LDP-up}

Fix $s>0$ $\delta>0$. Let $K_i$, $i \in \{0,...,N+1\}$ be defined as in Section \ref{SS:Markov-transitions}. We reorder these sets so that
\begin{align}
  L_0&:= V(x_*, \partial D)\nonumber\\%\label{eq:L0}
    L_i&:=\hat{V}_{D}(x_*, K_i)\nonumber %\label{eq:Li}
\end{align}
are in increasing order $L_0 \leq L_1 \leq L_2 \leq ... \leq L_{N+1}$.
Let us recall
\begin{equation}
  K_{N+1}:= \left\{y \in \partial D: \dist_E(y, \hat\Psi(s))\geq \delta \right\}.\nonumber
\end{equation}

Set $L_{N+1}:={\hat{V}_{D}}(x_*, K_{N+1}) = \inf_{y \in K_{N+1}}  \hat{J}(y) + V(x_*, \partial D)$. Notice that
\begin{equation} \label{eq:L_N+1-bound}
  L_{N+1} \geq s + L_0.
\end{equation}

First we derive a lower bound on $\tilde V_D(K_i,K_j)$.
\begin{lemma} \label{lem:V-lower-bound}
    For any $i \in \{1,2, ..., N\}$ and $j \in \{1,...,N+1\}$,
  \begin{equation}
   \tilde{V}_{D} (K_i, K_j) \geq L_j - L_i. \label{eq:V-lower-bound}
  \end{equation}
\end{lemma}

{
\begin{remark} \label{rem:Why-hat}
Notice that $L_i$, $ i\in \{1,...,N\}$, were defined in terms of $\hat{V}_{D}$ instead of $\tilde{V}_{D}$. Validity of this Lemma \ref{lem:V-lower-bound} with the $L_i$ defined in terms of $\tilde{V}_{D}$ would actually imply that lower and upper bounds for the exit shape distribution match in Theorem \ref{thm:exit-shape}, which would yield a large deviations principle for the exit shape distribution. We believe that Lemma \ref{lem:V-lower-bound} is not true when the $L_i$ are defined in terms of $\tilde{V}_D$ unless we impose extra assumptions on our domain $D$. This is the reason that we introduced $\hat{V}_D$ in \eqref{eq:V-hat-D} and defined the $L_i$ in terms of this quasipotential. We will come back to this issue in Section \ref{S:ExitShape_AdditionalCond}.
\end{remark}

\begin{proof}[Proof of Lemma \ref{lem:V-lower-bound}]

 Fix $i \in \{1,...,N\},j \in \{1,...,N+1\}$.
  The result is trivial if $\tilde{V}_D(K_i,K_j)=+\infty$ or $L_i=+\infty.$ We assume that $\tilde{V}_D(K_i,K_j)<+\infty$ and $L_i<+\infty.$

  Because $K_i$  is compact, Theorems \ref{thm:V-tilde-compact-level-sets} and \ref{thm:V-hat-compact-level-sets}  imply that there exist minimizers $y_1 \in K_i, y_2 \in K_i$, and $y_3 \in K_j$ such that
  \begin{align}
    &\hat{V}_D(x_*, K_i) = \hat{V}_D(x_*, y_1)
    &\tilde{V}_D(K_i,K_j) = \tilde{V}_D(y_2,y_3).
  \end{align}

  We also observe that for any $\rho>0$,
  \begin{align}
    &\hat{V}^\rho_D(x_*, y_1) \leq \hat{V}_D(x_*, y_1)
    &\tilde{V}^\rho_D(y_2,y_3) \leq \tilde{V}_D(y_2,y_3).
  \end{align}

  Let $\rho>0$ and $\gamma>0$ be given.   Because $\hat{V}_{D}(x_*,K_i)< +\infty$, there exists $T_1>0$ and $\varphi_1 \in C([0,T_1]:E)$ such that $\varphi_1(0)=x_*, \varphi_1(T_1) =y_1$, $\varphi_1(t) \in D \cup B(y_1,\rho) \cup \bigcup_{l=1}^N B(K_l,\rho)$ for all $t \in [0,T_1]$ and
  \begin{equation}
    I_{x_*}^{T_1}(\varphi_1) < \hat{V}^\rho_D(x_*, y_1) +\gamma/3 \leq \hat{V}_D(x_*, y_1) + \gamma/3 \leq \hat{V}_D(x_*, K_i) + \gamma/3.
  \end{equation}

  Because $K_i$ is a $V$ equivalence class and $y_1, y_2 \in K_i$, Lemma \ref{lem:paths-stay-near-K_i} guarantees that there exists $T_2>0$ and $\varphi_2 \in C([0,T_2]:E)$ such that $\varphi_2(0)=y_1, \varphi_2(T_2) = y_2$, $\dist_E(\varphi(t), K_i)<\rho$ for all $t \in [0,T_2]$ and
  \begin{equation}
    I_{y_1}^{T_2}(\varphi_2) < \gamma/3.
  \end{equation}

  Because $\tilde{V}_{D}(K_i,K_j)<\infty$, there exists $T_3>0$ and $\varphi_3 \in C([0,T_3]:E)$ such that $\varphi_3(0)=y_2$, $\varphi_3(T_3) = y_3$, {$\varphi(t) \in D \cup B(y_2,\rho) \cup B(y_3,\rho) $} for all $t \in [0,T_3]$ and
  \begin{equation}
    I^{T_3}_{y_2}(\varphi_3) < \tilde V_D^\rho(y_2,y_3) + \gamma/3 \leq \tilde{V}_D(y_2,y_3) + \gamma/3 \leq \tilde{V}_D(K_i,K_j) + \gamma/3.
  \end{equation}

  Now we concatenate by setting $T:=T_1+T_2+T_3$ and
  \[
     \varphi(t) :=\begin{cases}
    \varphi_1(t) & \text{ if } t \in [0,T_1)\\
    \varphi_2(t-T_1) & \text{ if } t \in [T_1, T_1 + T_2)\\
    \varphi_3(t-T_1-T_2) & \text{ if } t \in [T_1+T_2, T_1 + T_2+T_3),
  \end{cases}
  \]

   Notice that $\varphi \in C([0,T]:E)$ and has the properties that $\varphi(0) = x_*$, $\varphi(T) = y_3 \in K_j$, $\varphi(t) \in D \cup B(y_3,\rho) \cup \bigcup_{l=1}^N B(K_l,\rho)$ for all $t \in [0,T]$ and
  \begin{equation}
    I^T_{x_*}(\varphi) \leq \hat{V}_D(x_*, K_i) + \tilde{V}_D(K_i,K_j) + \gamma.
  \end{equation}

  By the definition of $\hat{V}_D^\rho$,
  \begin{equation}
    \hat{V}_D^\rho(x_*, y_3) \leq I^T_{x_*}(\varphi).
  \end{equation}

  Therefore,
  \begin{equation}
    \hat{V}_D^\rho(x_*, y_3) \leq \hat{V}_D(x_*, K_i) + \tilde{V}_D(K_i,K_j) + \gamma.
  \end{equation}

  Because $\gamma$ and $\rho$ were arbitrary, we can conclude that
  \begin{equation}
    \hat{V}_D(x_*, K_j) \leq \hat{V}_D(x_*, y_3) \leq \hat{V}_D(x_*, K_i) + \tilde{V}_D(K_i,K_j),
  \end{equation}
  proving the theorem.

\end{proof}

Define the stopping times for $0 \leq i \leq N$
\begin{equation*} %\label{eq:zeta-stopping-times}
  \zeta_i : = \min\left\{n : Z_n \not \in \bigcup_{j=i}^N g_j\right\}.
\end{equation*}

Unfortunately our set-up does not satisfy the assumptions of  Lemma 6.3.3 in \cite{FWbook} because there is no distance between the $g_i$ and $\partial D$. We now establish a similar result in our setting.
\begin{lemma} \label{lem:Z_zeta}
  For any $\gamma>0$, there exists $\rho_0>0$ such that for all $\rho \in (0, \rho_0)$ and for any $0 \leq j \leq i \leq N$
  \begin{equation} \label{eq:Z_zeta-result}
    \limsup_{\e \to 0} \sup_{x \in g_i} \e \log \Pro(Z_{\zeta_j} \in g_{N+1} | Z_0 = x) \leq -(L_{N+1} - L_i -\gamma).
  \end{equation}
\end{lemma}
\begin{proof}
  We prove this using an induction-like argument. The case $j=i=N$ is an immediate consequence of \eqref{eq:Z-eta-upper-gi} and \eqref{eq:V-lower-bound} because $\eta_N = \zeta_N$ { and because by construction $\hat{V}_{D}\leq \tilde{V}_{D}$.}

  Now assume that we have already proven the result for $j = l+1$ and for all $ i \in\{l+1,... ,N\}$. Now we prove it for $j =i=l$.
  There are several ways that  our Markov chain can satisfy $Z_0=x \in g_l$ and $Z_{\zeta_{l}} \in g_{N+1}$. The Markov chain can go from $g_l$ to $\bigcup_{i=l+1}^N g_i$ and back to $g_l$ any number of times $k=0,1,2,3,...$. Then the process can go directly from $g_l$ to $g_{N+1}$ or it can go  from $g_l$ directly to any of the $g_i$ $i \in\{ l+1, ...,N\}$ and then hit $g_{N+1}$ before hitting the boundary $\partial D$, $g_l$ again or $g_m$ for any $m<l$.
  \begin{align} \label{eq:Z_zeta-decomposition}
    &\sup_{x \in g_l} \Pro(Z_{\zeta_l} \in g_{N+1})\leq\nonumber\\
    & \sum_{k=0}^\infty \left(\left(\sup_{x \in g_l} \Pro\left(Z_{\eta_l} \in \bigcup_{i=l+1}^N g_i | Z_0=x\right) \right) \left(\sup_{x \in \bigcup_{i=l+1}^N g_i} \Pro\left(Z_{\zeta_{l+1}} \in g_l | Z_0 =x \right) \right) \right)^k\nonumber\\
    &\quad\times \Bigg(\sup_{x \in g_l} \Pro(Z_{\eta_l} \in g_{N+1} | Z_0=x)\nonumber\\
    &\qquad+ \sum_{i=l+1}^N \left(\sup_{x \in g_l} \Pro(Z_{\eta_l} \in g_i | Z_0=x) \right) \left(\sup_{x \in g_i} \Pro\left(Z_{\zeta_{l+1}} \in g_{N+1} | Z_0=x \right) \right)\Bigg)
  \end{align}

  For small enough $\rho>0$, \eqref{eq:gi-to-bdry-low} guarantees that for small $\e>0$
  \begin{equation}
    \inf_{x \in \bigcup_{i=l+1}^Ng_i} \Pro( Z_{\zeta_{l+1}} \in \partial D) \geq \inf_{x \in \bigcup_{i=l+1}^N g_i} \Pro(Z_1 \in \partial D) \geq e^{-\frac{\gamma}{3\e}}.\nonumber
  \end{equation}

  Then the infinite sum in \eqref{eq:Z_zeta-decomposition} can be bounded by
  \begin{align}
    &\sum_{k=0}^\infty \left(\left(\sup_{x \in g_l} \Pro\left(Z_{\eta_l} \in \bigcup_{i=l+1}^N g_i | Z_0=x\right) \right) \left(\sup_{x \in \bigcup_{i=l+1}^N g_i} \Pro\left(Z_{\zeta_{l+1}} \in g_l | Z_0 =x \right) \right) \right)^k\nonumber\\
    &\leq \sum_{k=0}^\infty \left(\sup_{x \in \bigcup_{i=l+1}^N g_i} \Pro\left(Z_{\zeta_{l+1}} \in g_l | Z_0 =x \right) \right)^k\nonumber\\
    &\leq \sum_{k=0}^\infty \left( 1 - \inf_{x \in \bigcup_{i=l+1}^N g_i} \Pro\left(Z_{\zeta_{l+1}} \in \partial D | Z_0 =x \right)\right)^k\nonumber\\
    &\leq \sum_{k=0}^\infty \left( 1 - e^{-\frac{\gamma}{3\e}}\right)^k \leq e^{\frac{\gamma}{3\e}}.\label{eq:inf-sum-bound}
  \end{align}

  As for the finite sum in \eqref{eq:Z_zeta-decomposition},  equations \eqref{eq:Z-eta-upper-gi} and \eqref{eq:V-lower-bound} (or \eqref{eq:Z-eta-upper-g0} if $l=0$) tell us that for small enough $\rho$ and $\e$ and any $i \in \{l+1,...,N+1\}$,
  \begin{equation} \label{eq:eta-bound}
   \sup_{x \in g_l} \Pro(Z_{\eta_l} \in g_i | Z_0=x) %\leq e^{-(\tilde{V}_{D}(K_l, K_i) - \gamma/3)/\e}
   \leq e^{-(L_i - L_l - \gamma/3)/\e}.
   % \sup_{x \in g_l} \Pro(Z_{\eta_l} \in g_i | Z_0=x) \leq e^{-(V_{\bar{D}}(K_l, K_i) - \gamma/3)/\e} \leq e^{-(L_i - L_l - \gamma/3)/\e}.
  \end{equation}
  %The last inequality follows from Lemma \ref{lem:V-lower-bound}.
  By the inductive step, for $i \in \{l+1, ..., N\}$ and small enough $\rho$ and $\e$,
  \begin{equation} \label{eq:inductive-hyp}
    \sup_{x \in g_i} \Pro\left(Z_{\zeta_{l+1}} \in g_{N+1} | Z_0=x \right) \leq e^{-(L_{N+1}-L_i - \gamma/3)/\e}.
  \end{equation}

  Then by \eqref{eq:Z_zeta-decomposition}, \eqref{eq:inf-sum-bound}, \eqref{eq:eta-bound}, and \eqref{eq:inductive-hyp},
  \begin{equation*}
    \sup_{x \in g_l} \Pro(Z_{\zeta_l} \in g_{N+1} | Z_0=x) \leq \sum_{i=l+1}^N e^{-(L_{N+1} - L_l - \gamma)/\e}
  \end{equation*}
  and consequently
  \begin{equation*}
    \limsup_{\e \to 0} \sup_{x \in g_l} \e\log\Pro(Z_{\zeta_l} \in g_{N+1} | Z_0=x) \leq -(L_{N+1}-L_l - \gamma).
  \end{equation*}

  We are almost done with the proof. We have proven \eqref{eq:Z_zeta-result} for $j=i=l$. It remains to prove the result for $j=l$ and  $i \in \{l+1, ..., N\}$. For such $i$ we observe that
  \begin{align*}
    &\sup_{x \in g_i} \Pro(Z_{\zeta_l} \in g_{N+1} | Z_0=x) \\
    &\leq \sup_{x \in g_i} \Pro(Z_{\zeta_{l+1}} \in g_{N+1}|Z_0=x) \\
    &\quad+ \left(\sup_{x \in g_i} \Pro(Z_{\zeta_{l+1}} \in g_l | Z_0=x)\right) \left(\sup_{x \in g_l} \Pro(Z_{\zeta_l} \in g_{N+1} | Z_0=x ) \right).
  \end{align*}

  By choosing $\rho$ and $\e$ small enough, we can bound the above expression by
  \begin{equation*}
    e^{-(L_{N+1}-L_i - \gamma)/\e} + e^{-(L_{N+1}-L_l - \gamma)/\e}.
  \end{equation*}

  Therefore,
  \begin{equation*}
    \limsup_{\e \to 0} \sup_{x \in g_i} \e\log \Pro(Z_{\zeta_l} \in g_{N+1} | Z_0=x) \leq -(L_{N+1}-L_i - \gamma).
  \end{equation*}
\end{proof}

Now we can complete the proof of \eqref{eq:shape-LDP-upper}.
\begin{proof}[Proof of \eqref{eq:shape-LDP-upper}]
  If the initial condition $x \in g_0$, then the probability of the exit shape belonging to $K_{N+1}$ is less than the probability of the Markov chain $Z_n$ reaching $g_{N+1}$ before reaching $\partial D$. In particular, if $x \in g_0$,
  \begin{align}
    &\Pro(\dist_E(X^\e_x(\tau^\e_x), \hat{\Phi}(s)) \geq \delta) = \Pro(X^\e_x(\tau^\e_x) \in K_{N+1}) \nonumber\\
    &\leq \Pro(Z_n \in g_{N+1} \text{ for some } n) = \Pro(Z_{\zeta_0} \in g_{N+1}).\nonumber
  \end{align}

  By Lemma \ref{lem:Z_zeta},
  \begin{align}
    &\limsup_{\e \to 0} \sup_{x \in g_0} \e \log \Pro(\dist_E(X^\e_x(\tau^\e_x), \hat{\Phi}(s)) \geq \delta)\nonumber\\
    &\leq \limsup_{\e \to 0} \sup_{x \in g_0} \e \log \Pro(Z_{\zeta_0} \in g_{N+1} | Z_0=x) \leq - (L_{N+1} - L_0 - \gamma).\nonumber
  \end{align}

  If $x \in D \setminus g_0$, then because of the fact that $D$ is the domain of attraction of $x_*$, there exists a small $\rho>0$ (depending on $x$) such that
  \begin{equation*}
    \limsup_{\e \to 0} \Pro(X^\e_x(\tau_1) \in g_0) = 1.
  \end{equation*}

  By the strong Markov property, for any $x \in D$
  \begin{align*}
    &\limsup_{\e \to 0 } \e \log \Pro(\dist_E(X^\e_x, \hat\Phi(s))\geq \delta) \\
    &\leq \limsup_{\e \to 0 } \e \log \left(\Pro(X^\e_x(\tau_1) \in g_0) \sup_{y \in g_0}\Pro(\dist_E(X^\e_y, \hat\Phi(s))\geq \delta)   \right)\\
    &\leq -(L_{N+1}-L_0 - \gamma).
  \end{align*}
  By \eqref{eq:L_N+1-bound} the above expression is less than $-(s-\gamma)$ and the result follows because $\gamma$ was arbitrary.

\end{proof}

\subsection{Proof of lower bound \eqref{eq:shape-LDP-lower}} \label{SS:Shape-LDP-low}

Let $r>0$ and $\rho>0$ be given. For $y \in \partial D$ let us set
\[
K_{N+1}^{r}=\left\{x\in\partial D: |x-y|_{E}\leq r\right\}
\]
and
\begin{align}
 g^{r}_{N+1}&= \left\{x \in D: \text{dist}(x,K_{N+1}^{r})<\rho\right\}\nonumber\\
 G^{r}_{N+1}&=\left\{x \in D: \text{dist}(x,K_{N+1}^{r})=2 \rho\right\}\nonumber
\end{align}

Let $K_1,...,K_N$ be given as in Assumption \ref{A:K_equivalenceClassV} and $K_0 = \{x_*\}$. Then Lemma \ref{lem:impossible-to-stay-away} holds with $\Theta=\bigcup_{i=0}^{N}K_i\cup K_{N+1}^r$. With only a slight modification of what we did earlier in this section we define
\begin{align*}
  &g_i:=\{x \in D: \dist_E(x,K_i)<\rho\} \setminus g^r_{N+1}\\
  &G_i:=\{x \in D: \dist_E(x,K_i)=2\rho\} \setminus G^r_{N+1}.
\end{align*}
%It is clear that we may assume w
With these definitions $g_{i}\cap g^{r}_{N+1}=\emptyset$ and $G_{i}\cap G^{r}_{N+1}=\emptyset$ for $i\in\{1,\cdots, N\}$.

Similarly to our construction of the Markov chain $Z_n$ defined earlier in this section, we now we define the  Markov chain $Z_n$  taking values on $\bigcup_{i=0}^{N} g_i \cup g^{r}_{N+1} \cup \partial D$. Notice now that the set $g_{N+1}$ used for the upper bound has been replaced by $g^{r}_{N+1}$ for the lower bound.

\begin{lemma}[Transition probabilities] \label{lem:g0-to-sets-lower-bound}
  For any {$\gamma,\delta>0$} there exists $\rho_0\in(0,\delta)$ such that for all $r,\rho \in (0,\rho_0)$ with $r+\rho<\rho_{0}$
    \begin{equation} \label{eq:g0-to-gi-low}
   \liminf_{\e \to 0} \inf_{x \in g_0} \e \log \Pro(Z_1 \in g^{r}_{N+1} | Z_0=x) \geq -\tilde{V}^{r}_D(x_*,y) - \gamma.
  \end{equation}

In addition, we have
 \begin{equation} \label{eq:g0-to-bdry-up}
    \limsup_{\e \to 0} \sup_{x \in g_0} \e \log \Pro(Z_1 \not \in g_0 | Z_0=x) \leq -V(x_*,\partial D) + \gamma.
  \end{equation}
\end{lemma}
\begin{proof}[Proof of Lemma \ref{lem:g0-to-sets-lower-bound}]
  Using Theorem \ref{thm:control} and the fact that $X^0_{x*}(t)=x_*$ for all $t>0$, by choosing $\rho>0$ small enough we can guarantee that for any $x \in g_0$ there exists $\varphi_1^x \in C([0,\rho]:E)$ such that $\varphi_1^x(0)=x$, $\varphi_1^x(\rho)=x_*$, and $I_x^\rho(\varphi^x_1)< \gamma/3$.

  If $\tilde{V}^{r}_D(x_*, y) =+\infty$, then \eqref{eq:g0-to-gi-low} holds trivially. If $\tilde{V}^{r}_D(x_*, y)<+\infty$, then there exists $T_{2}>0$ and $\varphi_2 \in C([0,T_{2}]:E)$ such that $\varphi_2(0) = x_*$, $\varphi_2(T_{2}) =y$, $\varphi_2(t) \in D\cup B(y,r) $ for all $t \in [0,T_{2}]$, and $I_{x_*}^{T_{2}}(\varphi_2)<\tilde{V}^{r}_D(x_*, y) + \gamma/3$.

 Let us now set $T^{*}_{2}=\inf\{t\in[0,T_{2}]: \varphi_{2}(t)\in\partial D\}$, i.e., the first time that $\varphi_{2}$ reaches $\partial D$.   Because $\varphi_2(t)\in D\cup B(y,r)$, we have that $\varphi_{2}(T^{*}_{2})\in K^{r}_{N+1}$.  In addition, since for all $i=1,\cdots, N$, $K_{i}\subset \partial D$, we get that for all $i=1,\cdots, N$ and all $t\in[0,T^{*}_{2}]$, $\varphi_{2}(t)\notin K_{i}$. Thus, for $\rho$ sufficiently small we can  also have that for all $i=1,\cdots, N$ and all $t\in[0,T^{*}_{2}]$, $\varphi_{2}(t)\notin G_{i}$.

 Let us set now $T=\rho+T_{2}$ and for $x\in g_{0}$ define
 \[\varphi^x(t) :=\begin{cases}
    \varphi_1^x(t) & \text{ if } t \in [0,\rho)\\
    \varphi_2(t-\rho) & \text{ if } t \in [\rho, \rho+T_{2}],
  \end{cases}\]
 and note that
 \[
 I_{x_*}^{T}(\varphi^x)<\tilde{V}^{r}_D(x_*, y) + 2\gamma/3.
 \]

  There exists $a\in(0,\rho)$ sufficiently small such that  any path $\psi \in C([0,T]:E)$ with the property that $|\psi - \varphi^x|_{C([0,T]:E)}<a$ has the property that $\psi(t)\notin g_i$ for $t\in[0,T]$ and $i=1,\cdots, N$. Such a $\psi$ will enter $g^{r}_{N+1}$ before touching any other $g_i$. Thus, the large deviations lower bound \eqref{eq:ldp-low} yields
\begin{align}
    &\liminf_{\e \to 0} \inf_{x \in g_0} \e \log \Pro(Z_1 \in g^r_{N+1} | Z_0 = x)  \nonumber\\
    &\qquad\geq \liminf_{\e \to 0} \inf_{x \in g_0} \e \log \Pro(|X^{\e}_{x}-\varphi^x|_{C([0,T]:E)}<a)\nonumber\\
    &\qquad  \geq  -(\tilde{V}^{r}_D(x_*, y) + 2\gamma/3)\nonumber
  \end{align}
  proving \eqref{eq:g0-to-gi-low}.

   The argument for \eqref{eq:g0-to-bdry-up} is basically the same as that of \eqref{eq:gi-to-gj-up} and \eqref{eq:g0-to-gi-up} and thus we omit its proof.
\end{proof}

\begin{lemma} \label{lem:easy-exit-near-bdry}
  %Let  $y \in H^1 \cap \partial D$ and set \textcolor{red}{$g_{N+1}:= \{x \in E: |x-y|_E<\rho\}$.
  %Let $\delta>0$}.
  For any $\gamma,\delta>0$, there exists $\rho_0\in(0,\delta)$ such that for all $r,\rho>0$ with $r+\rho<\rho_{0}$,
  \begin{equation} \label{eq:easy-exit-near-bdry}
    \liminf_{\e \to 0} \inf_{x \in g^{r}_{N+1}}\e \log  \Pro(|X^\e_x(\tau^\e_x) - y|< \delta) \geq -\gamma.
  \end{equation}
\end{lemma}

\begin{proof}
    Let $\gamma>0, \delta>0$. By Lemma \ref{lem:control-prob-cont}, there exist $\rho_0>0$ and $T_0>0$ such that whenever $X^{0,u}_x$ is a controlled trajectory with $\frac{1}{2}\int_0^{T_0}|u(s)|_H^2ds< \gamma$, $|x-y|_E<\rho_0$, then whenever $t \in [0,T_0]$,
  \begin{equation} \label{eq:close-to-y}
    |X^{0,u}_x(t) - y|_E < \delta/2.
  \end{equation}

  We can choose $\rho_0<T_0$. Furthermore, Theorem \ref{thm:control} guarantees that by possibly decreasing $\rho_0$ and letting $r+\rho<\rho_0$, we can construct a path connecting any $x \in g^{r}_{N+1}$ to $X^0_z(2\rho_0)$ as long as $|x-z|_E< 2\rho_0$. By Assumption \ref{assum:D-boundary} and the fact that $y \in \partial D$, there exists $z \not \in \bar D$ such that $|y-z|_E<\rho_0$. By Theorem \ref{thm:control}, for any $x \in g^{r}_{N+1}$ there exists $\varphi^x$ such that $\varphi^x(0)=x$, $\varphi^x(2\rho_0) = X^0_z(2\rho_0)$ and $I_x^{2\rho_0}(\varphi^x) <\gamma$. By \eqref{eq:close-to-y} we also know that $|\varphi^x(t) - y|_E<\delta/2$ for all $t \in [0,2\rho_0]$. By Assumption \ref{A:AttractionProperty} and Remark \ref{rem:stay-away}, $\varphi^x(2\rho_0) = X^0_z(2\rho_0) \not \in \bar D$.

   Let $0<a< \frac{1}{2} \min\{\dist_E(X^0_z(2\rho_{0}), D), \delta\}$.  Then, we get the estimate
   \begin{equation*}
     \Pro(|X^\e_x(\tau^\e_x) - y|_E<\delta) \geq \Pro(|X^\e_x - \varphi^x|_{C([0,2\rho_{0}]:E)} < a).
   \end{equation*}

   If $X^\e_x$ follows near $\varphi^x$, then it definitely exits $D$ and it stays within $\delta$ of $y$ the whole time. By the LDP lower bound,
   \begin{equation*}
     \liminf_{\e \to 0} \inf_{x \in g^{r}_{N+1}} \e \log \Pro(|X^\e_x(\tau^\e_x) - y|< \delta) \geq -\gamma.
   \end{equation*}
\end{proof}

Let us recall now the definition
\begin{equation*} %\label{eq:eta-stopping-time}
  \eta_0  = \min\{n \in \mathbb{N}: Z_n \not \in g_0\}.
\end{equation*}

\begin{lemma}
For any $\gamma,\delta>0$, there exists $\rho_0\in(0,\delta)$ such that for all $r,\rho>0$ with $r+\rho<\rho_{0}$,
 \begin{equation} \label{eq:Z-eta-lower}
    \liminf_{\e \to 0} \inf_{x \in g_0} \e \log \Pro(Z_{\eta_0} \in g^{r}_{N+1} | Z_0 =x) \geq -(\tilde{V}^{r}_D(x_*, y) - V(x_*, \partial D) + \gamma).
  \end{equation}
\end{lemma}

\begin{proof}
Decomposing the event of interest over all possible values of $\eta_{0}$ gives,
 \begin{align*}
    &\inf_{x \in g_0}\Pro(Z_{\eta_0} \in g^{r}_{N+1} | Z_0=x) = \inf_{x \in g_0} \sum_{k=1}^\infty \Pro(\eta_0 =k, Z_k \in g^r_{N+1} | Z_0=x) \\
    &\geq\sum_{k=1}^\infty \left( \inf_{x \in g_0} \Pro(Z_1 \in {g_0} | Z_0=x)\right)^{k-1} \left(\inf_{x \in g_0} \Pro(Z_1 \in g_{N+1}^r | Z_0=x) \right)\\
    &\geq \frac{\inf_{x \in g_0} \Pro(Z_1 \in g^r_{N+1} | Z_0=x)}{\sup_{x \in g_0} \Pro(Z_1 \not \in g_0 | Z_0 = x) }
  \end{align*}

  Then \eqref{eq:Z-eta-lower} follows from \eqref{eq:g0-to-gi-low} and \eqref{eq:g0-to-bdry-up}.
\end{proof}

Now, we can finally proceed with the proof of the lower bound \eqref{eq:shape-LDP-lower} for general $x \in D$.
\begin{proof}[Proof of \eqref{eq:shape-LDP-lower}]
Fix  $y \in \partial D$, $\delta>0$, and $\gamma>0$. We observe that for any $z \in g_0$,
\begin{equation*}
  \left\{|X^\e_z(\tau^\e_z) -y|<\delta  \right\} \supset \left\{ Z_{\eta_0} \in g^{r}_{N+1} \text{ and } |X^\e_{z}(\tau^\e_z)-y|<\delta \right\}.
\end{equation*}
By the strong Markov property,
\begin{align*}
  &\inf_{z \in g_0}\Pro\left( |X^\e_z(\tau^\e_z) - y|<\delta\right) \nonumber\\
  &\geq  \inf_{z\in g_0}\Pro \left(Z_{\eta_0} \in g^{r}_{N+1} |Z_0=z \right) \inf_{z_2 \in g^{r}_{N+1}} \Pro(|X^\e_{z_2}(\tau^\e_{z_2})-y|<\delta).
\end{align*}

{It follows from \eqref{eq:Z-eta-lower} and \eqref{eq:easy-exit-near-bdry}, that for small enough $\rho,r>0$,}
\begin{equation} \label{eq:exit-shape-lower-g_0}
  \lim_{\e \to 0} \inf_{z \in g_0} \e \log \Pro\left( |X^\e_z(\tau^\e_z) - y|<\delta\right) \geq -(\tilde{V}^{r}_D(x_*, y) - V(x_*, \partial D)) - \gamma.
\end{equation}

We can use the strong Markov property to show that \eqref{eq:exit-shape-lower-g_0} holds for any initial condition $x \in D$, and not just for $x \in g_0$. Let $x \in D \setminus \{x_*\}$ and let $\rho,r>0$ be small enough for \eqref{eq:exit-shape-lower-g_0} to hold. Possibly decrease the value of $\rho$ to guarantee that $\dist_E\left(x, \partial D \cup \{x_*\}\right) > 2\rho$, guaranteeing in particular that $x \not \in g_i$ for $i \in \{0,1,...,N+1\}$.
Because $x \in D$ and $D$ is a domain of attraction for $x_*$,
\[\lim_{\e \to 0} \Pro(X^\e_x(\tau_1) \in g_0)=1.\]
Then by \eqref{eq:exit-shape-lower-g_0} and the strong Markov property,
\begin{align}
  &\liminf_{\e \to 0}  \e \log \Pro(|X^\e_x(\tau^\e_x)-y|_E < \delta) \nonumber \\
  &\geq \liminf_{\e \to 0} \e \log\left( \Pro(X^\e_x(\tau_1) \in g_0)\inf_{z \in g_0} \Pro(|X^\e_z(\tau^\e_z)-y|_E < \delta)\right)\nonumber\\
  &\geq -(\tilde{V}^{r}_D(x_*, y) - V(x_*, \partial D)) - \gamma.\nonumber
\end{align}
%\textcolor{red}{[REMOVED $\inf_{x \in g_0}$ FROM PREVIOUS DISPLAY]}

The result follows because $\tilde{V}^{r}_D(x_*, y)\leq \tilde{V}_D(x_*, y)$ for all $r>0$ and because $\gamma>0$ was arbitrarily chosen.
\end{proof}
\section{LDP for exit shape under extra assumptions}\label{S:ExitShape_AdditionalCond}

In Section \ref{S:ExitShape} we proved Theorem \ref{thm:exit-shape}, which characterizes the large deviations lower and upper bounds for the exit shape $X^\e_x(\tau^\e_x)$. In this section, we show that under additional conditions the lower and upper bounds match, leading to a large deviations principle for {the exit shape $X^\e_x(\tau^\e_x)$.} %In particular this means that  as $\e$ converges to zero, this shape converges to the minimizers of $\inf_{y \in \partial D}V_{\bar D}(x_*, y)$.

As we mentioned in Remark \ref{rem:Why-hat}, Lemma \ref{lem:V-lower-bound} is the only step of the proof of the LDP upper bound for the exit time \eqref{eq:shape-LDP-upper} that requires us to use $\hat{V}_D$ instead of $\tilde{V}_D$.  Analogously to our setup section \ref{SSS:Shape-LDP-up}, but defining sets in terms of $\tilde{V}_D$ instead of $\hat{V}_D$, for $s\geq 0$ and $\delta>0$ let
\[K_{N+1}:=\left\{y \in \partial D: \dist_E(y, \tilde{\Psi}(s))\geq \delta \right\}\]
and set
\begin{align*}
  \tilde L_0&:=V(x_*,\partial D)\\
  \tilde L_i&:= \tilde{V}_D(x_*, K_i), \ \  i \in \{1,...,N+1\}
\end{align*}
ordered so that $L_i$ is increasing.

\begin{assumption} \label{assum:extra}
  Assume that we can prove the following strengthened version of Lemma \ref{lem:V-lower-bound}. For any $i \in \{1,...,N\}$ and $j \in \{1,....,N+1\}$,
  \begin{equation} \label{eq:suff-cond-for-shape-ldp}
    \tilde{V}_D(K_i,K_j) \geq \tilde L_j - \tilde L_i. %= \tilde{V}_D(x_*, K_j) - \tilde{V}_D(x_*,K_i).
\end{equation}
\end{assumption}
Under Assumption \ref{assum:extra}, we can prove the following strengthened version of Lemma \ref{lem:Z_zeta} using the same proof.

\begin{lemma} \label{lem:Z_zeta-extra}
  For any $\gamma>0$ there exists $\rho_0>0$ such that for all $\rho\in (0,\rho_0)$ and for any $0 \leq j \leq i \leq N$
  \begin{equation}
    \limsup_{\e \to 0} \sup_{x \in g_i} \e \log \Pro(Z_{\zeta_j} \in g_{N+1} | Z_0 =x) \leq -(\tilde{L}_{N+1} - \tilde{L}_i - \gamma).
  \end{equation}
\end{lemma}
Then we can prove under Assumption \ref{assum:extra} that the exit shape will satisfy a large deviations principle with respect to the rate function $\tilde{J}$.

\begin{theorem} \label{thm:exit-shape_LDP}
  In addition to the assumptions for Theorem \ref{thm:exit-shape} let Assumption \ref{assum:extra} hold. Then the exit shape $X^\e_x(\tau^\e_x)$ satisfies a large deviations principle with rate function  $\tilde{J}: \partial D \to [0,\infty]$ as defined in  Theorem \ref{thm:exit-shape}.
\end{theorem}
Theorem \ref{thm:exit-shape_LDP} follows from Lemma \ref{lem:Z_zeta-extra} using the arguments from section \ref{SSS:Shape-LDP-up}.

We conclude by discussing what is required for Assumption \ref{assum:extra} to hold.
In the finite dimensional situations considered by \cite{FWbook} and \cite{Day1990}, the authors assumed that the boundary of the attracting set was sufficiently smooth. Up to this point in this paper we have not made any assumptions about the regularity of the $\partial D$ and we were able to prove Theorems \ref{thm:meanExitAsymptotics} and \ref{thm:exit-shape} without assuming boundary regularity. One necessary ingredient for Assumption \ref{assum:extra} to hold, however,  is that the boundary is sufficiently regular near the $K_i$. Interestingly, our formulation of $\tilde{V}_D$ makes it unnecessary to impose any regularity assumptions on the rest of $\partial D$. Of course, characterizing boundary regularity is quite a challenge in infinite dimensional settings.

If one can formulate a necessary condition for boundary regularity near the $K_i$, then we would also need a version of Theorem \ref{thm:control} that can connect initial conditions near $K_i$ to controlled trajectories that start near $K_i$ in a way where those trajectories do not leave $D$ and do not increase the rate function by much. We proved Lemma \ref{lem:V-lower-bound} by building a trajectory that connects approximately minimizing trajectories of $\hat{V}_D(x_*, K_i)$ and $\tilde{V}_D(K_i,K_j)$ in such a way that the connecting trajectory has an arbitrarily low rate function and stays arbitrarily close to $K_i$ (but possibly exits $D$). In order for Assumption \ref{assum:extra} to hold, we need to be able to build similar connecting trajectories with arbitrarily low rate function that do not leave $D$.

In the case of a stochastic Allen-Cahn equation studied in \cite{FarisLasinio1982}, Faris and Jona-Lasinio characterize the $K_i$ sets, which are all singleton sets of unstable equilibria. Furthermore,  Faris and Jona-Lasinio characterize the stable and unstable eigenfunctions of the second derivative  operator of the dynamics at these equilibria. We believe that in this specific case, one may be able to use these eigenfunctions to prove that the boundary of $D$ is sufficiently well-behaved near the $K_i$ and to prove that Assumption \ref{assum:extra} is valid for the stochastic Allen-Cahn equation. Ultimately such calculations are outside the scope of the present paper and we leave this for future work.

\appendix

\section
{Proof of Theorems \ref{T:BOUNDEDNESSPROPOFX} and \ref{THM:CONTROL-BOUNDS}} \label{App:MildSolution}

Let us first present the proof of Theorem \ref{T:BOUNDEDNESSPROPOFX}. In preparation for that let us prove some intermediate results. Let us  consider $\Phi \in C([0,T]:E)$. We study the solution $v^\Phi$ to
\[v^\Phi(t) = \int_0^t S(t-s) F(v^\Phi(s) + \Phi(s))ds.\]
\begin{theorem} \label{thm:v-properties}
  For any $\Phi \in C([0,T]:E)$, $v^\Phi$ is well-defined.
  \begin{enumerate}
    \item There exists $C>0$ such that for any $T>0$,
    \begin{equation} \label{eq:v-bound-1}
      |v^\Phi|_{C([0,T]:E)} \leq C( 1 + |\Phi|_{C([0,T]:E)}).
    \end{equation}
    \item For any $T>0, R>0, p\in [1,\infty]$ there exists $C= C(T,R,p)$ such that for any $\Phi_1, \Phi_2 \in C([0,T]:E)$ satisfying $|\Phi_i|_{C([0,T]:E)} \leq R$ for $i=1,2$,
    \begin{equation*}
      |v^{\Phi_1} - v^{\Phi_2}|_{C([0,T]:E)} \leq C |\Phi_1 - \Phi_2|_{L^p([0,T]:E)}.
    \end{equation*}
  \end{enumerate}
\end{theorem}

\begin{proof}
   Let us set $v(t) := v^\Phi(t)$. Let $\delta_{v(t)} \in \partial |v(t)|_E$  such that
  \begin{align*}
    &\frac{d^-}{dt}|v(t)|_E \leq \left< Av(t) + F(v(t) + \Phi(t)), \delta_{v(t)} \right>_{E,E^\star}\\
    &\leq \left< A v(t) + F(v(t) + \Phi(t)) - F(\Phi(t)), \delta_{v(t)} \right>_{E,E^\star} + |F(\Phi(t))|_E
  \end{align*}

  By assumptions \eqref{eq:F-dissip} and \eqref{eq:F-growth}, and the dissipativity of $A$ we get
  \begin{equation*} %\label{eq:dvdt-bound}
    \frac{d^-}{dt}|v(t)|_E \leq - \lambda |v(t)|_E^{1 + \rho_{*}} + C(1 + |\Phi(t)|_E^{1+\rho_{*}}).
  \end{equation*}

   Let $\delta>0$ and let $$t_\delta = \inf\left\{t \in [0,T]: |v(t)|_E^{1 + \rho_*} \geq \frac{C}{\lambda} \left(1 + \sup_{s \in [0,T]}|\Phi(s)|_E^{1 + \rho_*} \right) + \delta\right\}.$$ Because $|v(0)|_E = 0$, $t_\delta>0$. Because  $t_\delta$ is, by definition, the first time that $|v(t)|_E$ achieves this level, $t \mapsto |v(t)|_E$ must be increasing at $t=t_\delta$. But at $t=t_\delta$
   \[\frac{d^-}{dt} |v(t_\delta)|_E \leq - \lambda |v(t)|_E^{1 + \rho_*} + C(1 + |\Phi|_{C([0,T]:E)}^{1 + \rho_*}) \leq -\delta, \]
   contradicting the fact that $t \mapsto |v(t)|_E$ is increasing. This contradiction proves that $t_\delta \not \in (0,T)$ and therefore, there exists a constant $C>0$ such that
   \[|v|_{C([0,T]:E)} \leq C \left( 1 +  |\Phi|_{C([0,T]:E)} \right) + \delta^{\frac{1}{1 + \rho_*}}.\]
   Because $\delta>0$ was arbitrary,
  %Now since $|v(t)|_E \bigg|_{t = 0} =0$, we get by comparison principle that
%  \[-\lambda |v(t)|_E^{1+\rho_{*}} + C \left(1 +|\Phi|_{C([0,t]:E)}^{1+\rho_{*}} \right) \geq 0.\]
%
  %Hence,
  there is a constant $C>0$ independent of $x$, $v$, $\Phi$, and $t$ such that
    \[|v(t)|_E \leq C \left( 1 +  |\Phi|_{C([0,t]:E)} \right).\]
 for all $t>0$.

   For the Lipschitz continuity, note that $F$ is locally Lipschitz.
   \begin{align}
    \frac{d^-}{dt}|v^{\Phi_1}(t) - v^{\Phi_2}(t)|_E &\leq \left< A(v^{\Phi_1}(t) - v^{\Phi_2}(t)) +  F(v^{\Phi_1}(t) + \Phi_1(t))\right.\nonumber\\
     &\quad \left.- F(v^{\Phi_2}(t) + \Phi_2(t)), \delta_t \right>_{E,E^{*}}\nonumber
    \end{align}
where $\delta_t \in \partial|v^{\Phi_1}(t) - \partial |v^{\Phi_2}(t)|_E$. Because $E$ is the space of continuous functions, $\delta_t$ is a convex combination of delta measures at maximizing $\xi  \in [0,L]$ and negative delta measures of minimizing $\xi \in [0,L]$. Without loss of generality, assume that $\delta_t$ is the delta measure at $\xi$ such that $v^{\Phi_1}(t,\xi) - v^{\Phi_2}(t,\xi) = |v^{\Phi_1}(t) - v^{\Phi_2}(t)|_E$. Then, we get that $\Delta (v^{\Phi_1}(t,\xi) - v^{\Phi_2}(t,\xi)) \leq 0$ because $\xi$ is a maximizer. By the local Lipschitz continuity of $F$,

   \begin{align}
   \frac{d^-}{dt}|v^{\Phi_1}(t) - v^{\Phi_2}(t)|_E &\leq C |v^{\Phi_1}(t) + \Phi_1(t) - v^{\Phi_2}(t) - \Phi_2(t)|_E \nonumber\\
   &\leq C|v^{\Phi_1}(t) - v^{\Phi_2}(t)|_E + C|\Phi_1(t) - \Phi_2(t)|_E.\nonumber
   \end{align}

   Therefore,
   \[|v^{\Phi_1}(t) - v^{\Phi_2}(t)|_E \leq C \int_0^t e^{C(t-s)} |\Phi_1(s) - \Phi_2(s)|_E ds.\]

   By H\"older's inequality, for any $t \in [0,T]$,
   \[|v^{\Phi_1}(t) - v^{\Phi_2}(t)|_E \leq C \left(e^{\frac{CpT}{p-1}} -1 \right)^{\frac{p-1}{p}}|\Phi_1 - \Phi_2|_{L^{{p}}([0,T]:E)}.\]

      The above equation along with a standard contraction mapping argument prove the well-posedness of $v^\Phi$.
\end{proof}

For any $T>0$ let $\mathcal{M}: C((0,T]:E) \to C([0,T]:E)$ be defined by
$\mathcal{M}(\Phi) = v^\Phi + \Phi$.
\begin{corollary}
  The mapping $\mathcal{M}$ is well-defined and
  \begin{enumerate}
    \item There exists $C>0$ such that for all $T>0$,
    \begin{equation}
      |\mathcal{M}(\Phi)|_{C([0,T]:E)} \leq C( 1 + |\Phi|_{C([0,T]:E)}).\nonumber
    \end{equation}
    \item For any $T>0, R>0$ and $p \in [1,\infty]$, there exists $C = C(T,R,p)$ such that if $|\Phi_i|_{C([0,T]:E)} \leq R$, $i=1,2$,
    \begin{equation}
      |\mathcal{M}(\Phi_1) - \mathcal{M}(\Phi_2)|_{L^p([0,T]:E)} \leq C|\Phi_1 - \Phi_2|_{L^p([0,T]:E)}.\nonumber
    \end{equation}
  \end{enumerate}
\end{corollary}

Now we show that as long as the nonlinearity $F$ satisfies Assumption \ref{assum:nonlinear}, we can bound $X^\e_x(t)$ and $X^{0,u}_x(t)$ in ways that are independent of the initial condition. For this purpose, for $x \in E$ and $\Phi \in C([0,T]:E)$, let
\[
v_x^\Phi(t) = S(t)x + \int_0^t S(t-s) F(v_x^\Phi(s) + \Phi(s))ds.
\]

The well-posedness of $v^\Phi_x$ is a consequence of Theorem \ref{thm:v-properties}.
\begin{theorem}
There exists a constant $C$ independent of $x$, $t$, and $\Phi$ such that for any $t>0$,
  \begin{equation} \label{eq:v-bound-sup-x}
    \sup_{x \in E} |v_x^\Phi(t)|_E \leq C( 1 + t^{-\frac{1}{\rho_{*}}} + \sup_{s \in [0, t]} |\Phi(s)|_E).
  \end{equation}
\end{theorem}
\begin{proof}
  For notational simplicity, let $v(t) := v_x^\Phi(t)$. For any $\delta_{v(t)} \in \partial |v(t)|_E$,
  \begin{align*}
    &\frac{d^-}{dt}|v(t)|_E \leq \left< Av(t) + F(v(t) + \Phi(t)), \delta_{v(t)} \right>_{E,E^\star}\\
    &\quad \leq \left< A v(t) + F(v(t) + \Phi(t)) - F(\Phi(t)), \delta_{v(t)} \right>_{E,E^\star} + |F(\Phi(t))|_E
  \end{align*}

  By \eqref{eq:F-dissip} and \eqref{eq:F-growth}, along with the dissipativity of $A$
  \begin{equation*}% \label{eq:dvdt-bound}
    \frac{d^-}{dt}|v(t)|_E \leq - \lambda |v(t)|_E^{1 + \rho_{*}} + C(1 + |\Phi(t)|_E^{1+\rho_{*}}).
  \end{equation*}

  Now we analyze $t^{\frac{1}{\rho_{*}}}|v(t)|_E$.
  \begin{align*}
    &\frac{d^-}{dt} t^{\frac{1}{\rho_{*}}}|v(t)|_E \leq \frac{1}{\rho_{*}}t^{\frac{1}{\rho_{*}}-1}|v(t)|_E + t^{\frac{1}{\rho_{*}}}\frac{d^-}{dt}|v(t)|_E\\
    &\leq \frac{1}{\rho_{*}} t^{\frac{1}{\rho_{*}} - 1} |v(t)|_E -\lambda  t^{\frac{1}{\rho_{*}}}|v(t)|_E^{1+\rho_{*}} + C t^{\frac{1}{\rho_{*}}}(1+|\Phi(t)|_E^{1+\rho_{*}})\\
    &\leq \frac{1}{t} \left(\frac{1}{\rho_{*}}t^{\frac{1}{\rho_{*}}}|v(t)|_E - \lambda \left(t^{\frac{1}{\rho_{*}}}|v(t)|_E \right)^{1+\rho_{*}} + Ct^{\frac{1 + \rho_{*}}{\rho_{*}}} \left( 1 + |\Phi(t)|_E^{\rho_{*}+1}\right) \right).
  \end{align*}

  By Young's inequality there is a large constant $C>0$ such that
  \[\frac{1}{\rho_{*}}t^{\frac{1}{\rho_{*}}}|v(t)|_E \leq \frac{\lambda}{2} \left(t^{\frac{1}{\rho_{*}}}|v(t)|_E \right)^{1 + \rho_{*}} + C.\]
  Therefore,
  \begin{align*}
    &\frac{d^-}{dt} t^{\frac{1}{\rho_{*}}}|v(t)|_E \leq \frac{1}{t} \left(-\frac{\lambda}{2}\left(t^{\frac{1}{\rho_{*}}}|v(t)|_E\right)^{1+\rho_{*}} + C \left(1 + t^{\frac{1+\rho_{*}}{\rho_{*}}} + t^{\frac{1+\rho_{*}}{\rho_{*}}}{|\Phi(t)|_E^{1+\rho_{*}}} \right) \right).
  \end{align*}

  Because $t^{\frac{1}{\rho_{*}}}|v(t)|_E \bigg|_{t = 0} =0$, the arguments in the proof of Theorem \ref{thm:v-properties} guarantee that
  \[-\frac{\lambda}{2}\left(t^{\frac{1}{\rho_{*}}}|v(t)|_E\right)^{1+\rho_{*}} + C \left(1 + t^{\frac{1+\rho_{*}}{\rho_{*}}} + t^{\frac{1+\rho_{*}}{\rho_{*}}}\sup_{s \in [0,t]}{|\Phi(s)|_E^{1+\rho_{*}}} \right) \geq 0.\]

  We can conclude that there is a constant $C>0$ independent of $x$, $v$, $\Phi$, and $t$ such that
  \[t^{\frac{1}{\rho_{*}}}|v(t)|_E \leq C\left(1 + t^{\frac{1}{\rho_{*}}} + t^{\frac{1}{\rho_{*}}}\sup_{s \in [0,t]}|\Phi(s)|_E \right).\]

  Consequently,
  \[|v(t)|_E \leq C \left(t^{-\frac{1}{\rho_{*}}} + 1 + \sup_{s \in [0,t]} |\Phi(s)|_E \right).\]
  This bound is independent of $x \in E$.
\end{proof}

Now, we are ready to prove Theorem \ref{T:BOUNDEDNESSPROPOFX}.
\begin{proof}[Proof of Theorem \ref{T:BOUNDEDNESSPROPOFX}]
  First, let us set
  \[
  v(t) = X^\e_x(t) - Y^\e_x(t) - S(t)x = \int_0^t S(t-s)F(X^\e_x(s))ds.
  \]

  By \eqref{eq:v-bound-1}, we have that
  \[\sup_{s \in [0,t]} |v(s)|_E \leq C \left(1 + \sup_{s \in [0,t]} |S(s)x + Y^\e_x(s)|_E \right) \leq C\left(1 + |x|_E + \sup_{s \in [0,t]}|Y^\e_x(s)|_E\right).\]

  Then
  \[\sup_{s \in [0,t]} |X^\e_x(s)|_E \leq \sup_{s \in [0,t]}|v(s)|_E + |x|_E + \sup_{s \in [0,t]} |Y^\e_x(s)|_E,\]
  proving \eqref{eq:X-bound-x}.

  Next, let $v(t) = X^\e_x(t) - Y^\e_x(t) = S(t)x + \int_0^t S(t-s)F(X^\e_x(s))ds$. Then by \eqref{eq:v-bound-sup-x}
  \[|v(t)|_E \leq C \left(1 + t^{-\frac{1}{\rho_{*}}} + \sup_{s \leq t} |Y^\e_x(s)|_E \right).\]

  Hence, we obtain
  \[|X^\e_x(t)|_E \leq |v(t)|_E + |Y^\e_x(t)|_E \leq C \left(1 + t^{-\frac{1}{\rho_{*}}} + \sup_{s \leq t} |Y^\e_x(s)|_E  \right),\]
  proving (\ref{eq:X-bound-sup}).

 Lastly, we prove (\ref{eq:X-big-zero-prob}). By \eqref{eq:X-bound-sup}, we can choose $R_0=3C$ (where $C$ is the constant from \eqref{eq:X-bound-sup}). It follows that
  \begin{equation*}
    \sup_{|x|_E>R_0} \Pro(|X^\e_x(1)|_E> R_0) \leq \sup_{|x|_E > R_0}\Pro\left(\sup_{s \in [0,1]} |Y^\e_x(s)|^p>1\right).
  \end{equation*}

  By the Chebyshev inequality and \eqref{eq:stoch-conv-bound}, it then follows that (\ref{eq:X-big-zero-prob}) holds.
\end{proof}

We conclude this section with the proof of Theorem \ref{THM:CONTROL-BOUNDS}.
\begin{proof}[Proof of Theorem \ref{THM:CONTROL-BOUNDS}]
  Let
  \begin{equation*}
    Y^{0,u}_x(t) : = \int_0^t S(t-s)G(X^{0,u}_x(s))u(s)ds.
  \end{equation*}
  By \eqref{eq:semigroup-regularity}, \eqref{eq:G-bound}, and the H\"older inequality,
  \begin{align} \label{eq:Y-bound}
    &\left|Y^{0,u}_x(t) \right|_E \leq \int_0^t |S(t-s)G(X^{0,u}_x(s))u(s)|_Eds\nonumber \\
    &\leq C\int_0^t (t-s)^{-\frac{1}{4}} |G(X^{0,u}_x(s))u(s)|_Hds \leq C \int_0^t (t-s)^{-\frac{1}{4}} |u(s)|_Hds\nonumber\\
    & \leq C \left(\int_0^t (t-s)^{-\frac{1}{2}}ds \right)^{\frac{1}{2}}\left(\int_0^t |u(s)|_H^2 ds \right)^{\frac{1}{2}}\nonumber\\
    &\leq Ct^{\frac{1}{4}}|u|_{L^2([0,t]:H)}.
  \end{align}

  Let $v(t) = X^{0,u}_x(t) - Y_x^{0,u}(t) - S(t)x$. Then
   \[
     v(t) = \int_0^t S(t-s)F(v(s) + S(s)x + Y^{0,u}_x(s))ds
     \]
   and
  by \eqref{eq:v-bound-1}, for any $T>0$,
  \begin{equation*}
    |v|_{C([0,T]:E)} \leq C ( 1 + |S(\cdot)x + Y^{0,u}_x|_{C([0,T]:E)}) \leq C(1 + |x|_E + |Y^{0,u}_x|_{C([0,T]:E)}).
  \end{equation*}

  Therefore it follows that
  \begin{equation*}
    |X^{0,u}_x|_{C([0,T]:E)} \leq |v|_{C([0,T]:E)} + |S(\cdot)x|_{C([0,T]:E)} + |Y^{0,u}_x|_{C([0,T]:E)}.
  \end{equation*}

  It follows from \eqref{eq:Y-bound} that
  \begin{equation*}
    |X^{0,u}_x|_{C([0,T]:E)} \leq C(1 + |x|_E + T^{\frac{1}{4}}|u|_{L^2([0,T]:H)}),
  \end{equation*}
  and \eqref{eq:control-sup-time-bound} holds.

  On the other hand, we can let $v_x(t) = X^{0,u}_x(t) - Y^{0,u}_x(t)$ so that
  \[
    v_x(t) = S(t)x + \int_0^t S(t-s)F(v(s) + Y^{0,u}_x(s))ds.
   \]

  It follows from \eqref{eq:v-bound-sup-x} that
  \begin{equation}
    \sup_{x \in E} |v_x(t)|_E \leq C(1 + t^{-\frac{1}{\rho_{*}}} + |Y^{0,u}_x|_{C([0,t]:E)}) \leq C(1 + t^{-\frac{1}{\rho_{*}}} + t^{\frac{1}{4}}|u|_{L^2([0,t]:H)}).
  \end{equation}

  Then
  \begin{equation} \label{eq:extra-t-dependence}
    \sup_{x \in E}|X^{0,u}_x(t)|_E \leq \sup_{x \in E} |v_x(t)|_E + \sup_{x \in E} |Y^{0,u}_x(t)|_E \leq C ( 1 + t^{-\frac{1}{\rho_{*}}} + t^{\frac{1}{4}}|u|_{L^2([0,t]:H)}).
  \end{equation}

    Now we demonstrate that the $t^{\frac{1}{4}}$ in front of the $|u|_{L^2([0,t]:H)}$ term in \eqref{eq:extra-t-dependence} can be removed.
    If $t \leq 1$, then $t^{\frac{1}{4}} \leq 1$ so \eqref{eq:control-sup-x-bound} holds.
    On the other hand, if $t>1$, then $X^{0,u}_x(t) = X^{0,\tilde u}_{X^{0,u}_x(t-1)}(1)$ where $\tilde u(s) = u(t-1 + s)$.
    Therefore, by \eqref{eq:extra-t-dependence},
    \begin{align*}
      \sup_{x \in E} |X^{0,u}_x(t)|_E \leq \sup_{x \in E} |X^{0,\tilde u}_x(1)| &\leq  C ( 1 + 1^{-\frac{1}{\rho_{*}}} + 1^{\frac{1}{4}}|\tilde u|_{L^2([0,1]:H)}) \\
      &\leq C ( 1 + t^{-\frac{1}{\rho_{*}}} + |u|_{L^2([0,t]:H)}),
    \end{align*}
    proving \eqref{eq:control-sup-x-bound}.

\end{proof}

\section{Time reversed processes}\label{A:ReversingProcesses}
We show that if $X^0_x$ solves
\begin{equation} \label{eq:unperturbed}
  dX^0_x(t) = AX^0_x(t) + F(X^0_x(t)), \ \ \ X^0_x(0) = x
\end{equation}
with initial condition $x \in H^1$, then the reversed dynamics $Y(t) : = X^0_x(T-t)$ {satisfy $I^T_y(Y)<\infty$ where $y =X^0_x(T)$ and $I^T_y$ is defined in \eqref{eq:rate-fct-def}}.

The most important first step is to study the regularity of the semigroup convolution.
For $u \in L^2([0,T]:H)$ let
\begin{equation*}
  \Lambda(u)(t) : = \int_0^t S(t-s)u(s)ds.
\end{equation*}

\begin{lemma} \label{lem:conv-reg}
  $\Lambda$ is a bounded linear operator from $L^2([0,T]:H)$ to $L^\infty([0,T]:H^1)\cap L^2([0,T]:H^2)$
\end{lemma}

\begin{proof}
  Let $u \in L^2([0,T]:H)$ and set $u_k(t):=\left<u(t),e_k\right>_H$, where $\{e_k\}$ form a CONS of $H$ that diagonalizes $A$ as described in Section \ref{S:Assumptions}. For any $t \in [0,T]$,
  \begin{align*}
    &\left| \int_0^t S(t-s) u(s)ds \right|_{H^1}^2 = \sum_{k=1}^\infty \alpha_k \left(\int_0^t e^{-\alpha_k(t-s)}u_k(s)ds \right)^2\\
    &\leq \sum_{k=1}^\infty \alpha_k \left(\int_0^t e^{-2\alpha_k(t-s)}ds \right)\left(\int_0^t |u_k(s)|^2 ds \right)\\
    &\leq \frac{1}{2}\int_0^t |u(s)|_H^2.
  \end{align*}

  As for the $L^2([0,T]:H^2)$ bound, we calculate that
  \begin{align*}
    &|\Lambda(u)|_{L^2([0,T]:H^2)}^2 =\int_0^T \left|A\int_0^t S(t-s)u(s)ds \right|_H^2 dt\\
    &=\sum_{k=1}^\infty \alpha_k^2 \int_0^T  \left(\int_0^t e^{-\alpha_k(t-s)} u_k(s)ds \right)^2dt.
  \end{align*}

  By Young's inequality for convolutions,
  \[\int_0^T \left(\int_0^t e^{-\alpha_k(t-s)} u_k(s)ds \right)^2dt \leq \left(\int_0^T e^{-\alpha_k t}dt \right)^2 \left(\int_0^T |u_k(t)|^2 dt\right).\]

  Therefore,
  \begin{align*}
    &\left| \Lambda u \right|_{L^2([0,T]:H^2)}^2
    \leq \sum_{k=1}^\infty \int_0^T |u_k(t)|^2dt \leq |u|_{L^2([0,T]:H)}^2.
  \end{align*}
\end{proof}

\begin{lemma} \label{lem:semigroup-reg}
  If $x \in H^1$, then $t \mapsto S(t)x \in C([0,T]:H^1)\cap L^2([0,T]:H^2)$ and
  \begin{equation*}
    \int_0^t |S(t)x|_{H^2}^2 dt = \frac{1}{2} \left(|x|_{H^1}^2 - |S(t)x|_{H^1}^2 \right).
  \end{equation*}
\end{lemma}

\begin{proof}
  The fact that $t \mapsto S(t)x \in C([0,T]:H^1)$ is the definition of $S(t)$ being a $C_0$ semigroup on $H^1$. As for the $L^2([0,T]:H^2)$ bound, if we let $x_k:=\left<x,e_k\right>_H$,
  \begin{align*}
    &\int_0^T |AS(t)x|_H^2dt = \sum_{k=1}^\infty \alpha_k^2 \int_0^T e^{-2\alpha_k t} x_k^2 ds \\
    &= \frac{1}{2} \sum_{k=1}^\infty \alpha_k(1 - e^{-2\alpha_k t}) x_k^2 = \frac{1}{2} \left(|x|_{H^1}^2 - |S(t)x|_{H^1}^2 \right).
  \end{align*}
\end{proof}

%Recall that
%\begin{equation*}
%  X^0_x(t) = S(t)x + \int_0^t S(t-s)F(X^0_x(s))ds.
%\end{equation*}

\begin{lemma} \label{lem:unperturbed-reg}
  If $x \in H^1$, then  $X^0_x \in L^\infty([0,T]:H^1) \cap L^2([0,T]:H^2)$.
\end{lemma}
\begin{proof}
  Lemma \ref{lem:semigroup-reg} guarantees that $t \mapsto S(t)x \in L^2([0,T]:H^2)$. Lemma \ref{lem:conv-reg} guarantees that
  \begin{align*}
    & \int_0^T \left|\int_0^t S(t-s)F(X^0_x(s))ds \right|_{H^2}^2 dt \leq \int_0^T |F(X^0_x(t))|_H^2 dt.
  \end{align*}

  We know that $E \subset H$ and that there exists $C>0$ such that for all $x \in E$, $|x|_H \leq C|x|_E$. We also have bounds on the growth on $F$ (see \eqref{eq:F-growth}) and can therefore conclude that
  \begin{align*}
    & \int_0^T \left|\int_0^t S(t-s)F(X^0_x(s))ds \right|_{H^2}^2 dt \leq CT \left( 1+ |X^0_x|_{C([0,T]:E)}^{1+\rho}\right)^2.
  \end{align*}

  The proof is then concluded by the a-priori bounds of Theorem \ref{THM:CONTROL-BOUNDS}.
\end{proof}

\begin{theorem}[Energy of a reversed path ] \label{thm:reversed}
  Let $T>0$ and $x \in H^1$. Let $Y(t) = X^0_x(T-t)$. That is, $Y(t)$ is a time-reversed version of the dynamics with terminal point $x \in H^1$. Then it follows that
  \begin{equation*}% \label{eq:reverse-energy-bound}
    I_y^T(Y) \leq C \left(  |x|_{H^1}^2 - |S(T)x|_{H^1}^2 \right) + CT \left(1 + |X^0_x|_{C([0,T]:E)}^{1+\rho} \right)^2
  \end{equation*}
  where $y:=Y(0)=X^0_x(T)$ and $C>0$ is independent of $T$ and $x$.
\end{theorem}

\begin{proof}
Let $x \in H^1$. Then $X^0_x(t)$ is a weak solution to \eqref{eq:unperturbed}. Let $T>0$ and $Y(t): = X^0_x(T-t)$. Then $Y$ is a weak solution to
\begin{align*}
  dY(t) &= -A Y(t) -F(Y(t)) = A Y(t) + F(Y(t)) -2 (AY(t) + F(Y(t)))\\
  &=A Y(t) + F(Y(t)) -2G(Y(t))G^{-1}(Y(t)) (AY(t) + F(Y(t))).
\end{align*}

Letting $y = X^0_x(T)$ and $u(t)= -2G^{-1}(Y(t))(AY(t) + F(Y(t)))$, we can write $Y(t) = X^{0,u}_y(t)$.
The energy of the control is bounded by (applying  \eqref{eq:G-inv-norm} and Lemmas \ref{lem:conv-reg} and \ref{lem:semigroup-reg})
\begin{align} \label{eq:control-bound}
  &\frac{1}{2}\int_0^T |u(s)|_H^2 ds\nonumber \\
   &\leq C \int_0^T \left(|AS(t)x|_{H}^2 + \left|A \int_0^t S(t-s)F(X^0_x(s)) ds \right|_{H}^2 + |F(X^0_x(t))|_H^2 \right)dt \nonumber\\
  &\leq  C \left(  |x|_{H^1}^2 - |S(T)x|_{H^1}^2 \right) + CT \left(1 + |X^0_x|_{C([0,T]:E)}^{1+\rho} \right)^2.
\end{align}
\end{proof}


\begin{thebibliography}{99}

\bibitem{Barret2015}
F. Barret, Sharp asymptotics of metastable transition times for one dimensional SPDEs,
   {\em Ann. Inst. H. Poincar\'{e} Probab. Statist.}, Vol. 51, No. 1, (2015), pp. 129--166.




\bibitem{BerglundGentz2013}
N. Berglund and B. Gentz, Sharp estimates for metastable lifetimes in parabolic SPDEs: Kramers' law and beyond,
{\em Electron. J. Probab.}, Vol.  18, No. 24, (2013), pp. 1--58.


\bibitem{BerglundWeber2017}
N. Berglund, G. Di Ges\`{u} and H. Weber, An Eyring-–Kramers law for the stochastic Allen–-Cahn equation in dimension two,
{\em Electron. J. Probab.}, Vol.  22, No. 41, (2017), pp. 1--28.

\bibitem{BovierEckhoffGayrandKlein2004}
 A. Bovier, M. Eckhoff, V. Gayrard and M. Klein, Metastability in reversible diffusion processes 1. Sharp estimates for capacities and exit times, \emph{J. Eur. Math. Soc.}, Vol. 6, No. 4,  (2004), pp. 399-424.

\bibitem{BovierGayrandKlein2005}
A. Bovier, V. Gayrard and M. Klein. Metastability in reversible diffusion processes II: Precise estimates for small eigenvalues, \emph{J. Eur. Math. Soc.} Vol. 7, No. 1, (2005), pp. 69-99.


\bibitem{BudhirajaDupuisSalins2018}
A. Budhiraja, P. Dupuis, and M. Salins, Uniform large deviatoions principles for Banach space valued stochastic differential equations, \emph{Trans. Amer. Math. Soc.}, Vol. 372, (2019), pp. 8363-8421.

\bibitem{CerraiSmoothing1999}
S. Cerrai, Smoothing properties of transition semigroups relative to SDEs with value in Banach spaces, {\em Probab. Theory Related Fields}, Vol. 113, No. 1, (1999), pp. 85--114.

\bibitem{c-2003}
S.Cerrai, Stochastic reaction-diffusion systems with multiplicative noise and non-Lipschitz reaction term, {\em Probab. Theory Related Fields}, Vol. 125, No. 2 (2003), pp. 271--304.

\bibitem{CerraiRDEAveraging1} S. Cerrai, A Khasminskii type averaging principle for stochastic reaction-diffusion equations,
{\em Ann. Appl. Probab.}, Vol. 19, No. 3, (2009), pp. 899--948.


\bibitem{Cerrai-RocknerLDP} S. Cerrai and M. Rockner, Large deviations for stochastic reaction-diffusion systems with multiplicaitve noise  and non-Lipschitz reaction term, {\em Ann.  Probab.}, Vol. 32, No. 1B, (2004), pp. 1100--1139.

\bibitem{cm-1997} F. Chenal and A. Millet, Uniform lare deviations for parabolic SPDEs and applications, {\em Stochastic Process. Appl.}, Vol. 72, No. 2, (1997), pp. 161--186.

\bibitem{dpz-1991}
G. Da Prato, A.J. Pritchard, and J. Zabczyk, On mimumum energy problems, {\em SIAM J. on Control Optim.}, Vol. 29, No. 1, (1991), pp.209--221.

\bibitem{Day1990}
M. Day, Large deviations results for the exit problem with characteristic boundary, {\em J. Math. Anal. and Appl.}, Vol. 147, (1990), pp. 134-153.

\bibitem{Debussche2013}
A. Debussche, M. H\"{o}gele, P. Imkeller, {\em The dynamics of nonlinear reaction-diffusion equations with small Lévy noise.}  Springer Lecture notes in Mathematics, Vol. 2085, 163p. (2013).

\bibitem{DemboZeitouni1997}
A. Dembo and O. Zeitouni, {\sc Large deviations techniques and applications}, Second edition, Springer, 1997.

\bibitem{FarisLasinio1982}
W.G. Faris and G. Jona-Lasinio, Large fluctuations for a nonlinear heat equation with noise, {\em J. Phys. A}, Vol. 15, No. 10, (1982), pp. 3025--3055.


\bibitem{f-1988}
M. Freidlin, Random perturbations of reaction-diffusion equations: the quasi-deterministic approximation, {\em Trans. Amer. Math. Soc.}, Vol. 305, No. 2, (1988), pp. 665--697.

\bibitem{FWbook} M. Freidlin and A. Wentzell, {\sc Random perturbations
of dynamical systems}, Second Edition, Springer, 1998.

\bibitem{g-2005}
E. Gautier, Uniform large deviations for the nonlinear Schr\"odinger equation with multiplicative noise, {\em Stochastic Process. Appl.}, Vol. 115, No. 12, (2005), pp. 1904--1927.

\bibitem{g-2008}
E. Gautier, Exit from a basin of attraction for stochastic weakly damped nonlinear Scr\"odinger equations, {\em Ann. Probab.}, Vol. 36, No. 3 (2008), pp. 896--930.


\bibitem{Hogele2019}
M. A. H\"{o}gele, The first exit problem of reaction-diffusion equations for small multiplicative L\'{e}vy noise, arXiv preprint: 1706.07745, (2019).



\bibitem{kx-1996}
G. Kallianpur, J. Xiong, Large deviations for a class of stochastic partial differential equations, {\em Ann. Probab.}, Vol. 24, No. 1, (1996), pp. 320--345.

\bibitem{l-2018}
D. Lipshutz, Exit time asymptotics for small noise stochastic delay equations, {\em Discrete Contin. Dyn. Syst.}, Vol. 38, No. 6 (2018), pp. 3099--3138.

\bibitem{Salins2018}
M. Salins, Equivalences and counterexamples between several definitions of the uniform large deviations principle,    \emph{Probability Surveys}, Vol. 16, (2019), pp. 99-142.

\bibitem{s-1992}
R.B. Sowers, Large deviations for a reaction--diffusion equation with non--Gaussian perturbations, {\em Ann. Probab.}, Vol 20, No. 1, (1992), pp. 504--537.



\end{thebibliography}
\end{document}